\documentclass[]{amsart}

\usepackage{latexsym,amssymb,amsmath,amsthm}
\usepackage{comment}
\usepackage{graphicx}
\newtheorem{thm}{Theorem}
\newtheorem{coro}[thm]{Corollary}
\newtheorem{lemma}[thm]{Lemma}

\theoremstyle{definition}
\newtheorem{defn}{Definition}

\title{A graph theoretical Gauss-Bonnet-Chern Theorem}
\author{Oliver Knill}
\date{November 21, 2011}
\address{
        Department of Mathematics \\
        Harvard University \\
        Cambridge, MA, 02138
        }

\begin{document}
\maketitle


\begin{abstract}
We prove a discrete Gauss-Bonnet-Chern theorem $\sum_{g \in V} K(g) = \chi(G)$ for finite 
graphs $G=(V,E)$, where $V$ is the vertex set and $E$ is the edge set of the graph.
The dimension of the graph, the local curvature form $K$ and the 
Euler characteristic are all defined graph theoretically.
\end{abstract}

\section{Introduction}

The Gauss-Bonnet-Chern theorem $\int_M K(x) \; = \chi(M)$ for a compact
$d$-dimensional Riemannian manifold $M$ generalizes the Gauss-Bonnet 
theorem for compact $2$-dimensional surfaces. It averages the Euler curvature form
$K(p)= P(\kappa(p))$, the Pfaffian $P$ of the curvature form $\kappa(p)$ 
over the manifold, and leads to the Euler characteristic $\chi(M)$.
First tackled by Allendoerfer and Fenchel for surfaces 
in Euclidean space and extended by 
Allendoerfer and Weil to closed Riemannian manifolds,
it was Chern, whose 100'th birthday we celebrate this year, 
who first gave an intrinsic proof \cite{Chern44}.
Modern proofs use Fermionic calculus \cite{Rosenberg, Cycon} which becomes 
especially elegant in Patodi's approach \cite{Cycon}.  \\

We introduce here an Euler curvature form $K(p)$ for graphs
which only depends on the number $V_k$ of $k$-dimensional pieces of the unit sphere 
$S(p)$ at $p$ with $k=1, \dots, d-2$. Since by definition, we have a $d$-dimensional 
graph $G$, the unit sphere $S(p)$ at a point is a $(d-1)$-dimensional graph. 
While Puiseux type formulas allow to discretize curvature in two dimensions,
the lack of a natural second order difference calculus for graphs prevents a 
straightforward translation of the classical Euler curvature form to graph theory so that we construct
it from scratch using some assumptions on graphs so that everything stays elementary. 
For four dimensional graphs for example, the curvature form $K(p)$ at $p$ depends only on the 
number of edges and faces of the three dimensional sphere $S(p)$ centered at $p$. \\

For two-dimensional graphs, the Euler curvature form is $K(p)=1-E(p)/6$, 
where $E(p)$ is the arc length of the unit circle $S(p)$.
This curvature traces back to a combinatorial curvature defined in \cite{Gromov87}, where differential geometry is
pushed to more general spaces. Gromov's graph theoretical curvature is up to a normalization  defined by
$K(p) = 1-\sum_{j \in S(p)} (1/2-1/d_j)$, where $d_j$ are the cardinalities of the neighboring face degrees 
for vertices $j$ in the sphere $S(p)$. For two dimensional graphs, where necessarily all faces are triangles,
this simplifies to $d_j=3$ so that $K=1-|S|/6$, where $|S|$ is the cardinality of the sphere $S(p)$ of radius 1. 
If $S$ is a cyclic graph, then the arc length and vertex cardinality of the sphere are the same. 
The combinatorial curvature illustrates the $1/5$-condition $|S(p)|>5$ for non-positive curvature and 
that $|S|>6$ leads to strictly negative curvature.
For two-dimensional graph with or without boundary, the result
appears in \cite{elemente11}, where $K=6-E(p)$ was scaled to avoid fractions. 
We do not rescale the discrete Euler form in this article because with increasing dimensions, the denominator terms 
become larger and would have to be scaled in a dimension-dependent manner to be rendered integer valued. We could
work with $(d+1)! K(p)$ which is an integer but leave the fractions.
For two dimensional graphs, the sphere $S(p)$ of a vertex $p$ is one-dimensional graph 
without boundary, a cyclic graph. In three dimensions, the Euler curvature form $K$ 
vanishes identically. For four-dimensional graphs, which is already a new case, the curvature form 
is $K = 1-E/6 + F/10$, where $E$ is the number of edges of the unit sphere $S(p)$ and where $F$ is 
the number of faces in $S(p)$. 
For 5 dimensional graphs, the curvature form is $K = -E/6 + F/4 - C/6$. It would be zero if $3F=2(C+E)$. 
We do not know yet whether it is always identically zero, even so the sum over the entire graph is. 
Also  the examples of 5-dimensional graphs we looked at, $K=0$. \\

\begin{figure} \scalebox{0.30}{\includegraphics{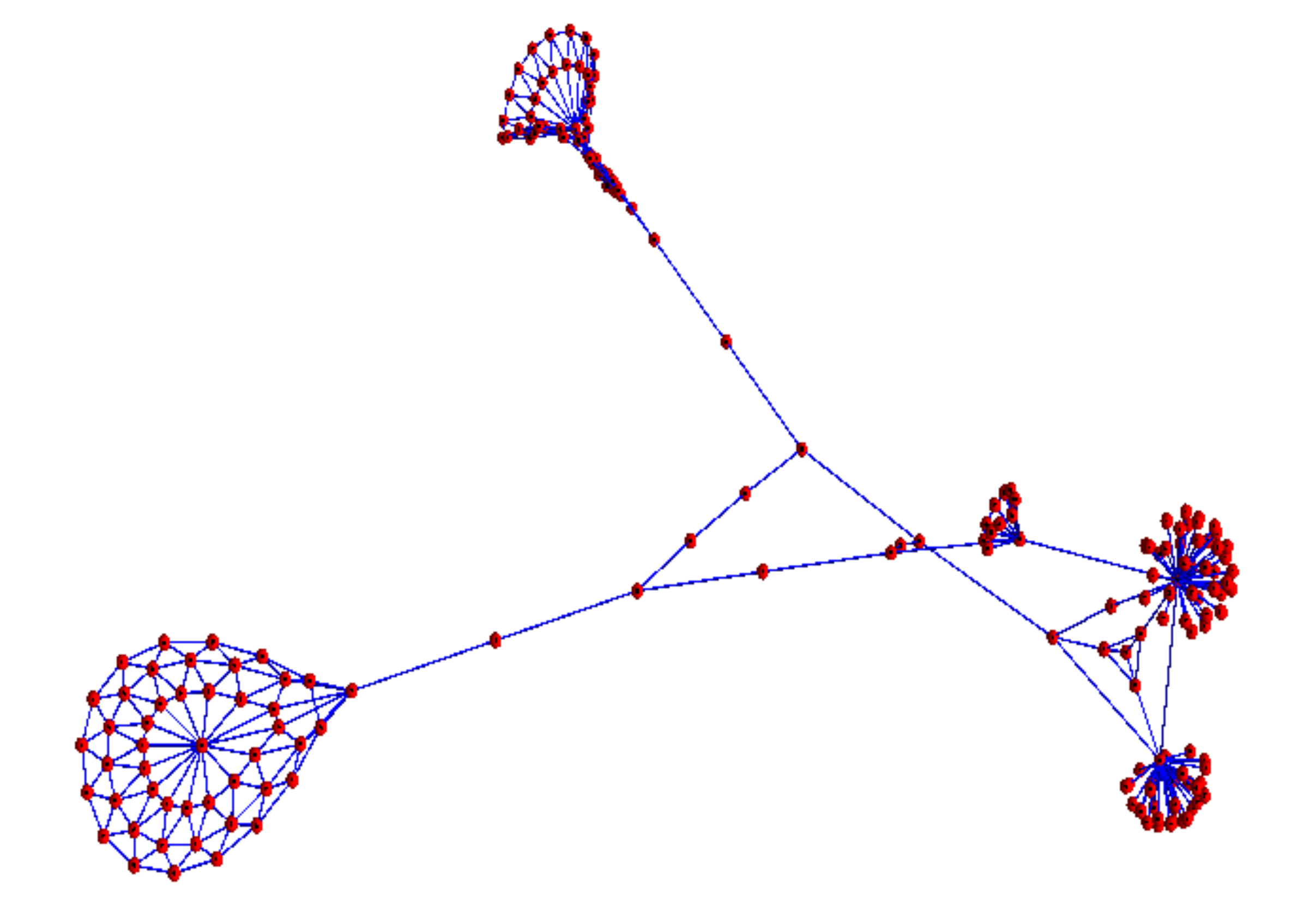}} \caption{
The sum of the 
curvature $K(v)$ is the Euler characteristic. The graph shown here
has dimension 32549/20580 and Euler characteristic $-4$. 
The curvatures $-21,-31/2,-19/6,-5/3,-3/2,-1,1/4$ appear once, $-1/2$ six times, $-1/4$ 3 times,
$60$ vertices have zero curvature, $50$ have curvature $1/6$ and $70$ have curvature $1/2$. 
These locally computed quantities add up to $-4$. 
} \label{example} \end{figure}

The core mechanism for relating global Euler characteristic and local curvature for general finite undirected 
graphs $G=(V,E)$ is the following: the total number $v_k$ of $k$ dimensional simplices in $G$ is related 
to the number $V_k(p)$ of $k-1$ dimensional simplices in the unit sphere $S_1(p)$ by $\sum_p V_{k-1}(p) = (k+1) v_k$. 
Then $\chi(G) = \sum_{k=0}^{\infty} (-1)^k v_k = \sum_{p \in V} K(p)$ with curvature 
\begin{equation}
\label{elementarycurvature}
 K(p) = 1+\sum_{k=1}^{\infty} (-1)^k \frac{V_{k-1}(p)}{k+1}  \;,
\end{equation}
where we use that $\sum_{p \in V} 1 =v_0$. 
We have just seen a very general Gauss-Bonnet theorem relating a local curvature form with a global quantity: \\

\begin{center}\parbox{12cm}{
{\bf For any finite graph $G=(V,E)$ without self-loops and multiple connections, 
we have $\sum_{p \in V} K(p) = \chi(G)$, where $K(p)$ is a curvature form defined in 
Equation~(\ref{elementarycurvature}) which depends only on the unit sphere $S_1(p)$ 
of a vertex $p$.} \\ } \end{center} 

Here are some examples. For a discrete graph $P_n$ with $n$ vertices without edges, 
the unit spheres are all empty and $K(p)=1$ sums up to $\chi(P_n)=n$.  
For the complete graph $K_n$, we have $\chi(K_n)=\sum_{k=0}^{n-1} (-1)^k {n \choose k+1}=1$ and 
since $S_1(p)$ is $K_{n-1}$ we have $K(p) = 1+\sum_{k=1}^{n-1} (-1)^k {n-1 \choose k}/(k+1) = 1/n$ adding up to $1$. 
For a tree, the curvature of a vertex $p$ is $1-{\rm deg}(p)/2$. 
The sum of the curvatures of a single tree is $1$ and the total curvature of a forest the number of trees. 
For a wheel graph $W_n$ with $n+1$ vertices and $2n$ edges and $n$ faces, we have $\chi(W_n)=(n+1)-2n+n=1$. The
central vertex $p$ has $K(p)=1-n/2+n/3=1-n/6$, the other $n$ points $q$ 
have $K(q)=1-3/2+2/3=1/6$. The total curvature is $1$. 
For an octahedron $G$, we have $v_0=6,v_1=12,v_2=8$ and $\chi(G)=6-12+8=2$ and 
since each unit sphere is the circular graph $C_4$ so that $K=1-4/2+4/3=1/6$
the total curvature is $\sum_p K(p)=2$. 
For a cube $G$ with $v_0=8,v_1=12,v_2=0$ we get $\chi(G)=-4$ reflecting the fact that we have $6$ holes.
Because $S_1(p)$ consists of $3$ discrete points, we have $K(p)=1-3/2=-1/2$ 
everywhere and the sum is $-4$. If we stellate the cube, it becomes
$2$ dimensional and the data $v_0=14,v_1=36,v_2=24$ lead to $\chi(G)=2$.
The original vertices have zero curvature $K(p)=1-6/2+6/3=0$, the six new ones
satisfy $K(p)=1-4/2+4/3=1/3$ and the total curvature is $2$. \\

As we have just seen, Gauss-Bonnet-Chern is elementary and works for general graphs. 
Despite that it can be stated and proven quickly, the topic is far from trivial. 
We explore it here in a setting close to differential geometry, where all unit spheres 
have properties known to spheres in $R^d$. While the topic is accessible to high school mathematics, 
there are many open questions. For example, we do not yet have cases of 5 dimensional 
graphs, where the Euler curvature form is not identically zero which makes us believe
that it is always zero in odd dimensions. Concrete examples can become tedious to 
compute by hand in higher dimensions and computers are essential for exploration.
A triangularization of the $5$ dimensional cube for example which can be realized 
as a convex polytope in $R^6$ leads to a graph with 536 vertices and 8216 edges which 
has five different types of $4$ dimensional unit spheres. Still, the curvature form 
$K(p) = -E/6 + F/4 - C/6$ for $d=5$ is identically zero at every point. 
Also combinatorial problems appear. Discrete versions of Bonnet-Schoenberg-Myers bounds 
show that there are only finitely many $d \geq 2$ dimensional graphs where all sectional 
curvatures are strictly positive meaning that every two dimensional subgraph 
has strictly positive curvature. Up to isomorphisms, there are six for $d=2$. How many are 
there in $d=3$? Also Ricci curvature $R(e)$, a function on the edge set $E$ which gives the sum over all 
curvatures of $2$ dimensional wheel graphs containing a given edge $e$ is 
unexplored. Any question in differential topology can be considered for $d$-dimensional graphs. 
For example, Chern proved that for four dimensional manifolds with nonpositive sectional 
curvatures, the Euler curvature form is nonnegative. Is it true that for a $4$ dimensional 
graph the curvature form $K = 1-E/6 + F/10$ is nonnegative if every two dimensional wheel 
subgraph of $G$ has $6$ or more spikes? 

\section{The theorem}

Throughout the paper we work in a graph theoretical setup. We especially do not
assume anywhere that graphs are embedded in any ambient Euclidean space, nor assume the graph to be a
triangularization of a smooth manifold. All graphs are finite undirected graphs
without self loops and without multiple edges. Any conceptual discretization of differential
geometry to graph theory requires to give a graph theoretical definition of what
dimension means for a graph. We work with the inductive graph dimension given
in \cite{elemente11} but make an additional assumption about the topology of the unit sphere.
The later is essential to be close to geometry. It even narrows the class of $d$ dimensional 
graphs because as we will see, there are $d$-dimensional graphs in which unit spheres can
have any prescribed Euler characteristic possible in the continuum. \\

Let $v_k$ denote the number of $k$-dimensional subgraphs of a graph $G=(V,E)$ which are
simplices, complete graphs $K_{k+1}$ with $(k+1)$ vertices.
The Euler characteristic of a $d$-dimensional graph $G$ is defined as
$$  \chi(G) = v_0-v_1+v_2-v_3 + \cdots + (-1)^{d} v_d \; . $$
For $d=2$ dimensions for example, we have $v_1=|V|=v, v_2=|E|=e$, $v_3=|F|=f$, the
number of vertices, edges and faces and $\chi(G) = v_0-v_1+v_2=v-e+f$.

\begin{defn}
A graph $G=(V,E)$ with $v=|V|$ vertices and $e=|E|$ edges is called a $d$-dimensional graph
without boundary if at every vertex $p$, the unit sphere
$S(p) = \{ v \in V \; | \; (p,v) \in E \; \}$ is a connected $(d-1)$-dimensional
graph without boundary of Euler characteristic $1+(-1)^{d-1}$. This is an inductive definition which 
also should apply to unit spheres. We additionally make the assumption that $(d+1) v_d = 2 v_{d-1}$.
Together with the assumption that a graph without any edges is zero-dimensional, this defines dimension inductively.
\end{defn}

By adding $1$ to the average of the dimensions of all spheres, we can define inductively dimension for all graphs, but the
dimension becomes fractional in general. Figure~(\label{example}) for example displays a fantasy graph of
dimension $1.5816...$. 

\section{The Euler Curvature form}

Define the fractions $a_n = (\frac{1}{2} - \frac{1}{n+2}) (-1)^{n}$,
and the integers $e_d = \frac{1+(-1)^{d}}{2}$. The later is $1$ for even $d$ and $0$ for odd $d$. 
Here are the first entries: \\

\begin{center}
\begin{tabular}{lllllllll} 
n      &    1 &     2 &     3 &     4 &     5 &     6 &     7 & ... \\ \hline
$a_n$  & -1/6 &   1/4 & -3/10 &   1/3 & -5/14 &   3/8 & -7/18 & ... \\
$e_n$  &    0 &     1 &     0 &     1 &     0 &     1 &     0 & ... \\
\end{tabular}
\end{center}

\begin{defn}
The Euler curvature form at a vertex $p$ in a $d$-dimensional graph $G$ is defined as
$K(p)=e_d+V_1 a_1$ if $d=2$ and
$$ K(p) = e_d+V_1 a_1 + V_2 a_2 + \cdots + V_{d-3} a_{d-3} + V_{d-2} (a_{d-2}+2a_{d-1}/d) \;   $$
for $d>2$.
\end{defn}

{\bf Remark 1.} The curvature form $K(p)$ depends on the $(d-2)$ sphere quantities $V_1(p), \dots ,V_{d-2}(p)$
if $d \geq 3$ and on only one sphere quantity $V_1(p)$ if $d=2$.\\
{\bf Remark 2.} If we would not use the hyper relations, we could work with the more elegant and symmetric
curvature $K(p) = e_d + \sum_{i=1}^{d-1} a_i V_i$. Without Euler characteristic assumption, we could  even
use the curvature Equation~(\ref{elementarycurvature}) given in the introduction. While less
geometric, it has the elegance of the Euler curvature in the continuum, if we define $V_{-1}=1$ so that
written as $K(p) = \sum_{k=0}^{\infty} (-1)^k V_{k-1}(p)/(k+1)$.  \\
{\bf Remark 3.} How can equation~(\ref{elementarycurvature}) be related to the Euler curvature form $P(\kappa(p))$
used in the classical Gauss-Bonnet-Chern? Here is as close as we could get: 
with a ``super operator" $D$ satisfying $V_{k-2}={\rm tr}(D^k)$ for $k \geq 1$ we have
$K(p) = {\rm tr}(\log(D+1)) = \log({\rm det}(D+1))$. If $D+1$ were skew symmetric then this is
$2 \log(P(D+1))$ with Pfaffian $P$ and would also establish $K(p)=0$ in odd dimensions. 
Discrete curvature $K(p)$ when summed up is formally close to a "height" construction 
$\log({\rm det}(L))$ in differential geometry, where $L$ is the Laplacian with a zeta regularized determinant.
Rewriting Gauss-Bonnet-Chern as $\chi(G) = (1/t) {\rm tr}(\log(D+1)^t)$ is a naive discrete analogue of
the McKean-Singer formula $\chi(M)={\rm str}(e^{tL})$ because $e^{tL} x$ solving $x'=Lx$ is replaced by 
$(D+1)^t x$ solving the discrete heat flow $x(t+1)-x(t) = D x(t)$. While this is only formal,
it adds hope that $K(p)$ is zero in odd dimensions and differential geometric assumptions done here. \\

For the following low-dimensional cases, we use the notation $V_0=V,V_1=E,V_2=F,V_3=C,V_4=S,V_5=H$:  \\

\begin{center}
\begin{tabular}{l|l|l}
dim&   Euler form  $K$           & simplified                    \\ \hline
d=2&  1-E/6                      & 1-E/6                         \\
d=3&   -E/6+(2/3) E/4            & 0                             \\
d=4&  1-E/6+F/4-(2/4) 3F/10      & 1 -E/6 + F/10                 \\
d=5&   -E/6+F/4-3C/10 +(2/5)C/3  &   -E/6 + F/4 - C/6            \\
d=6&  1-E/6+F/4-3C/10+S/3-(2/6)5S/14        & 1-E/6+F/4-3C/10+3S/14   \\
d=7&  -E/6+F/4-3C/10+S/3-(5H/14)+(2/7)3H/8&  -E/6+F/4-3C/10+S/3-H/4 \\
\end{tabular}
\end{center}

Our main result is that
integrating the Euler curvature form over the graph is equal to the Euler characteristic:

\begin{thm}[Discrete Gauss-Bonnet-Chern]
\label{gbc}
If $G=(V,E)$ is a finite $d$-dimensional graph, then 
$$  \sum_{p \in V} K(p) = \chi(G) \; . $$ 
\end{thm}

{\bf Remark 1.} Theorem~\ref{gbc} reduces to Gauss-Bonnet for two dimensional graphs 
discussed in the second part of \cite{elemente11}. 
As stressed in that article, all notions which enter the theorem are purely graph theoretical and no additional 
structure is required. In the same way that objects in differential topology are of 
primary interest, if they are defined intrinsically and in a coordinate independent way, 
we do not want to refer to an ambient structure or assume
non graph theoretically defined notions. A large field of discrete differential geometry
does include more information like angles or distances to be able to approximate given smooth structures.  
Discrete differential geometry has grown in many flavors: Regge calculus \cite{Misner} is used 
in general relativity, digital topology in computer vision, simplicial cohomology for algebraic topology,
or computational geometry in numerical methods of computer graphics \cite{Devadoss}. 
Graph theoretical approaches not using any other structures appear in network theory \cite{CohenHavlin}. \\
{\bf Remark 2.} The inductive dimension for graphs \cite{elemente11}
is closely related to the inductive Brower-Menger-Urysohn dimension for topological spaces
which goes back to the later work of
Poincar\'e \cite{HurewiczWallman,Engelking,Pears}. 
This dimension is the smallest $n$ such that every point has arbitrary small neighborhoods
$V$ such that the boundary of $V$ has inductive dimension $(n-1)$ and the 
empty set has dimension $-1$ \cite{HurewiczWallman}.
For a normal topological space with countable base, the inductive dimension is equal to
the topological dimension by Urysohn's metrization theorem.
It is different from the inductive graph dimension defined in \cite{elemente11} because both
"topological inductive dimension" and "covering dimension" are zero for finite graphs where the discrete topology of
all subsets it the only natural topology. There are other notions of "dimension":
the "topological dimension" of a graph always $1$ (e.g. \cite{munkres} page 305),
the "metric dimension" is the minimum number of vertices of a subset $A \subset V$ of $G=(V,E)$
such that every $v \in V$ is uniquely defined by the distances to points in $A$. The
"fractal dimension" for infinite graphs measures the growth rate $|B_r| \sim r^d$ \cite{CohenHavlin}. 
Also these are obviously different dimensions. \\
{\bf Remark 3.} The curvature form $K(p)$ defined here is local in the sense that it only depends on the
geometry of the unit ball $S_1(p)$ centered at the vertex $p$. Theorem~\ref{gbc} is interesting because it links 
deformable local properties with rigid global topological properties.  \\
{\bf Remark 4.} The discrete Gauss-Bonnet-Chern result is not a discrete approximation of the classical 
Gauss-Bonnet-Chern theorem because curvature takes a prescribed discrete set of values and can not 
become arbitrarily fine. Triangularizations allow to approximate $d$-dimensional 
manifolds by d-dimensional graphs but already the case of the two dimensional sphere, 
illustrated by geodesic domes, golf or buckyballs, shows that curvature is located on 12 points only, 
independent how fine the grid is. This implies for large graphs that most of the graph is flat.
In spirit the continuum and discrete results are the same result: the curvature is a local property which depends on the metric 
the Euler characteristic is global and a topological quantity. They are related in a global manner.  \\
{\bf Remark 5.} The previous remark does not exclude less naive connections between the discrete and continuum. 
Here is one: we could relate the continuous Gauss-Bonnet-Chern theorem with the discrete version by taking random 
triangularizations for which the expectation value of the discrete curvature is the classical curvature. \\
{\bf Remark 6.} The discrete curvature form changes during "homotopy deformations" of the graph but
the Euler characteristic does not. The discrete version of Euler form might be more accessible to geometric
interpretations since (to cite\cite{leeriemannian}) {\it "The only problem with the Gauss-Bonnet-Chern 
result is that the relationship between the Pfaffian and sectional curvatures is obscure in higher 
dimensions, so no one seems to have any idea how to interpret the theorem geometrically!"}
As in classical differential geometry, the geometry of spheres near a point 
$p$ determines the Euler curvature form. In classical differential geometry, the curvature tensor is
accessible through sectional curvatures which can be measured using Bertrand-Puiseux formulas from lengths
of circles on two dimensional geodesic sheets. 
There are Puisaux curvature formulas which do not refer to an underlying flat 
situation like $2 |S_1(p)|-|S_2(p)|$, where $|S_r(p)|$ is the length of the sphere
of radius $r$ in the sheet $\exp_p(\Sigma)$ where $\Sigma$ is a two dimensional plane in the tangent
space $T_pM$.  The classical curvature form $\kappa(v,w)$ of a manifold is 
determined from the lengths of circles of radius $2r$ and $r$ in a circle 
obtained by intersecting the sphere with the plane spanned by $v$ and $w$.  \\
{\bf Remark 7.} Even so the following quote of \cite{Chung97} appeared in the context of 
spectral graph theory, it hits the heart of matter:
{\it "Although differential geometry and spectral graph theory share a great deal 
in common, there is no question that significant differences exist. Obviously a graph is not "differentiable"
and many geometrical techniques involving high-order
derivatives could be very difficult, if not impossible to utilize for graphs. There are substantial 
obstacles for deriving the discrete analogues of many of the known results in the continuous case. 
Nevertheless, there are many successful examples developing the discrete parallels, and 
this process sometimes leads to improvement and strengthening of the original results from the continuous case."}
If there is a discrete Gauss-Bonnet-Chern theorem with second order notions, then this would be closer to the 
continuum, but this is more subtle and applies only to a much more restricted set of graphs. 
Even for the two dimensional Gauss-Bonnet theorem, there are challenges to identify the set 
of graphs in which second order curvatures like $2|S_1|-|S_2|$ work.  \\
{\bf Remark 8.} The study of dimension of graphs is also interesting in random graph theory. 
We have explicit formulas for the average dimension of a graph with $n$ vertices where each edge is
switched on with probability $p$. But the statistics of dimension is largely unexplored. We see experimentally
for example that both in random and concrete networks, points with integer dimension have higher probability. \\
{\bf Remark 9.} Since \cite{elemente11} was written, we also found out more about the origin of combinatorial curvature:
Higuchi \cite{Higuchi} uses a definition of an unpublished talk by Ishida from 1990, the
combinatorial curvature $H(p) = 1-|S_1(p)|/2 + \sum_{q \in S_1(p)}  1/|S_1(q)|$
which is dual to a curvature defined by Gromov \cite{Gromov87} in 1987. Higuchi 
poses a question whether the positivity of $H(p)$ 
for a planar graph $G$ in which $S_1$ is a closed circuit everywhere implies that the graph is finite.
He also mentions $\sum_p H(p) = 2$.  This is $\sum_{p \in V} (1-|S_1(p)|/6)=2$ in our notation,
for two-dimensional graphs, positive curvature is rather restrictive: the degree
of the graph has to be $4$ or $5$, the octahedron and icosahedron are the extreme cases with constant 
curvature $2$ and $1$ respectively. In the example section we will give
a list of all 2 dimensional graphs with positive curvature.\\
{\bf Remark 10.} Some results about $d$-dimensional graphs can be derived from the continuum because a $d$-dimensional 
graph can be embedded into a $d$-dimensional manifold producing a triangularization of the manifold. Still, 
it would be nice to have direct graph theoretical proofs of results like: any connected 2-dimensional graph of 
Euler characteristic $2$ is planar. The later is graph theoretical notion because by Kuratowski it is equivalent
to have no subgraph $K_5$ or $K_{3,3}$. While $K_5$ is 4 dimensional and excluded, the utility graph
$K_{3,3}$ is one dimensional and not ruled out by dimension alone. \\
{\bf Remark 11.} In the definition for a d-dimensional graph we include all the properties we need.
They could be replaced by topological requirements.
We could ask that for $d=1$ we have $|S_1(p)|=2$, for $d=2$ that the unit sphere is a connected
boundary-less graph and that for $d \geq 3$ the unit sphere is a simply connected graph.
This is equivalent if the unit sphere $S_1(p)$ of a graph can be embedded into a 
$d$-dimensional manifold with the same topological properties. Since this topological assumptions
need intrinsic justifications and the just mentioned terms depend on deep results in
differential geometry, we chose the stronger but simpler assumptions.

\section{Proof}

By assumption, the unit sphere $S_1(p)$ at every point is a $(d-1)$-dimensional graph without boundary. The
Euler characteristic assumption $\chi(S_1(p)) = 2 e_{d-1}$ assures for example that
the unit sphere is not a union of different spheres as in Hessel type examples. By definition,
$$ v_0-v_1+v_2-v_3+v_4-v_5 +\cdots +v_{d-1} = \chi    \; ,              $$ 
and by assumption, at every point $p$, the quantities $V_i=V_i(p)$ satisfy
$$ V_0-V_1+V_2-V_3+V_4-V_5 + \cdots  +V_{d-2} = \chi_1 = 2 e_{d-1}  \; .   $$ 
For two-dimensional graphs, for example, $v-e+f=\chi$ and $V-E=0$. The following 
equations are well known in graph theory.

\begin{lemma}[ Transfer equations]
$$ \sum_p V_{k-1}(p) = (k+1) v_k  \;  $$
for $k=1,\dots ,d$. 
\label{transferrelations}
\end{lemma}

\begin{proof} 
The equation $\sum_p V_0(p)   = 2 v_1$ means $\sum_p V(p) = 2 e$. It
holds because to every edge belong two vertices and $2e$ is the sum of all degrees of the graph
and because$V(p)$ is the degree of $p$. The second equation $\sum_p V_1(p)   = 3 v_2$
means $\sum_p E(p) = 3 f$, the sum of all face degrees of a point $p$ is three times the number
of faces and since faces touching $p$ are in one to one correspondence to edges of 
$S_1(p)$. In general, define the $k$-degree $d_k(p)$ of $p$ as the number of $k$-dimensional 
simplices which contain the point $p$. The usual degree is $d_0(p)$.
We have $d_k(p) = V_{k-1}(p)$ because $k$-dimensional simplices containing $p$ are
in one to one correspondence to $(k-1)$-dimensional faces of $S_1(p)$. 
There are $v_{k}$ simplices and $(k+1) v_{k}$ is the sum of $k$ degrees of the point
because each simplex adds $k+1$ vertices. Therefore, $\sum_p V_{k-1}(p) = (k+1) v_k$.
\end{proof} 

Additionally to the transfer equations, we have assumed a relation
between the number of $d$-dimensional pieces with $d-1$-dimensional pieces in the graph.
These hyper relations $(d+1) v_d  = 2 v_{d-1}$ are
inductively required to hold also for every unit sphere $S_1(p)$.
They reflect the fact that the graph $G$ has no boundary. 
In $2$ dimensions for example, the hyper relation is $3 f = 2 e$. 
Two $2$-dimensional adjacent faces intersect in an edge. Similarly, two $3D$-adjacent chambers 
intersect in a face showing the relation in $3D$. In general, we know that two 
$n$-dimensional adjacent pieces intersect in a hyper surface. We use these relations in a sum
form:

\begin{lemma}[Hyper relations]
$$ \sum_p V_{d-1}  = \frac{2}{d} \sum_p V_{d-2}   \; . $$
\label{hyperrelations}
\end{lemma}
\begin{proof}
Because a sphere is a $(d-1)$-dimensional graph without boundary too, the hyper relations also 
apply to the unit spheres which means
$$  d V_{d-1}  = 2 V_{d-2}  \; .  $$        
This iimples for example $2 E = 2 V$ for three dimensional graphs, where the unit spheres are two dimensional graphs
which are polyhedra. Summing up the second gives the sum form.
\end{proof}

{\bf Remark 1.} For $d=2$, the hyper relations follow from the dimensionality assumption. 
To violate the hyperrelations would mean that two adjacent spaces intersect in more than one face.  \\
{\bf Remark 2.} In all dimensions $d$, Lemma~(\ref{hyperrelations}) could be proven from some topological 
assumptions on the unit spheres as in the case $d=2$. The relations mean that two adjacent 
$d$-dimensional simplices always intersect in a single $(d-1)$-dimensional face. 
Instead of trying to prove the relations from more topological assumptions on the unit sphere 
like from a discrete version of simple connectivity, we assume in this paper the hyper relation.  \\
{\bf Remark 3.} Because of the hyper relations, there are only $d-2$ terms in the Euler form. 
In two dimensions, the transfer and hyper relation on the sphere are the same 
so that we still have one term in the curvature $K(p)$: the number of edges in the sphere $S(p)$. \\

All sums $\sum$ which appear in the following proof are understood over the vertex set $V$ and abbreviate $\sum_{p \in V}$.
We write $\chi=\chi(G)$ for the Euler characteristic of $G$ and $\chi_1$ for the Euler characteristic of the unit sphere 
$S_1(p)$. By assumption, we know that $\chi_1=2 e_{d-1}$ is a constant taking either the values $0$ or $2$.

Here is the proof of the theorem: 

\begin{proof}
By definition, the Euler characteristic of the graph satisfies
$$ v_0 - v_1    + v_2    - v_3    + \cdots   + (-1)^d  v_d  = \chi \; .  $$
Using Lemma~(\ref{transferrelations}), we get from this
$$ v_0 - \sum \frac{V_0}{2} + \sum \frac{V_1}{3} + \sum \frac{V_2}{4} + \cdots \pm \sum \frac{V_{d-2}}{d}  \mp \sum \frac{V_{d-1}}{(d+1)} = \chi  \; . $$
Because $\sum 1 = v_0$, this is
\begin{equation}
\sum 1 - \sum \frac{V_0}{2} + \sum \frac{V_1}{3} + \sum \frac{V_2}{4} + \cdots \pm \sum \frac{V_{d-2}}{d}  \mp \sum \frac{V_{d-1}}{(d+1)} = \chi  \; .
\label{proof0}
\end{equation}

Together with Lemma~(\ref{hyperrelations}) , we get 
\begin{equation}
\label{proof1}
 \sum 1 - \sum \frac{V_0}{2} + \sum \frac{V_1}{3} - \sum \frac{V_2}{4} + \cdots \pm \sum V_{d-2} (\frac{1}{d} - \frac{2}{d(d+1)})  = \chi \; .   
\end{equation}
The assumption 
$$   V_0-V_1+V_2- \dots \pm V_{d-1}=\chi_1 $$ 
for the Euler characteristic $\chi_1$ of each sphere can be written as
$$  \sum \frac{V_0}{2} - \sum \frac{\chi_1}{2} - \sum \frac{V_1}{2}+\sum \frac{V_2}{2}-\sum \frac{V_3}{2} \dots \mp \sum \frac{V_{d-1}}{2}  \;  $$
which is
\begin{equation}
    \sum \frac{V_0}{2} - \sum \frac{\chi_1}{2} - \sum \frac{V_1}{2}+\sum \frac{V_2}{2}-\sum \frac{V_3}{2} \dots \mp \sum (\frac{1}{2} - \frac{2}{d}) V_{d-2}  \; . 
\label{proof2}
\end{equation}
Adding equations~(\ref{proof1}) and ~(\ref{proof2})  gives
$$ \sum (1 - \frac{\chi_1}{2}) + \sum V_1 (\frac{1}{3}-\frac{1}{2}) - \sum V_2 (\frac{1}{4}-\frac{1}{2}) 
        + \dots  \pm  \sum V_{d-2} (  \frac{1}{d} - \frac{2}{d(d+1)} - \frac{1}{2} + \frac{1}{d} )  = \chi  \; .  $$
Now use the assumption $1 - \chi_1/2 = e_{d-1}$ and the identity $a_{d-2}+(2/d) a_{d-1}=-(-1)^d ((2/d) - 2/(d (d+1)) - 1/2)$
to see that the left hand side is the total curvature.
\end{proof}

{\bf Remark 1.} As mentioned in the introduction,
Equation~(\ref{proof0}) is already a Gauss-Bonnet type formula relating the global 
Euler characteristic with a local property. It works for all
graphs without assumptions on Euler characteristic of the unit sphere nor hyper relations. It only uses
the transfer equations and is very general. Too general for geometry. \\
{\bf Remark 2.} As in the continuum, there is a generalization of Gauss-Bonnet-Chern
to $d$-dimensional graphs $G$ with boundary $\delta G$ by defining the curvature at the boundary.
One can obtain the boundary case by gluing two copies of $G$ together at $\delta G$. This gives a
$d$-dimensional graph without boundary. \\
{\bf Remark 3.} We have defined in \cite{elemente11} a polyhedron as a graph which can be made a
$2$-dimensional boundary less graph by snubbing some vertices and then triangularizing finitely many
one dimensional subgraphs. A $d$-polytope can be defined as a graph which can be made a $d$-dimensional
graph without boundary by replacing some vertices with $d$-dimensional simplices and then stellating
finitely many $(d-1)$-dimensional boundary subgraphs which are polytopes. \\
{\bf Remark 4.} There are curvatures on graphs which involve larger neighborhoods of a point in the graph like the second
order formula $K=2 |S_1|-|S_2|$. Already for two dimensional surfaces, the Gauss-Bonnet formula is then 
more subtle. Cylindrical graphs with $7$ fold symmetry are already counter examples. 
It is necessary to restrict the set of graphs, similarly as differential geometry 
restricts to manifolds with a differentiable structure. 

\section{Examples}

In the rest of the article we illustrate the theorem with examples 
and repeat the computation of the proof for smaller dimensions. 

\subsection{General graphs} 

\begin{figure}
\begin{tabular}{ll|ll}
\parbox{3cm}{This tree is a $1-$dimensional graph with boundary.                                  }& \scalebox{0.10}{\includegraphics{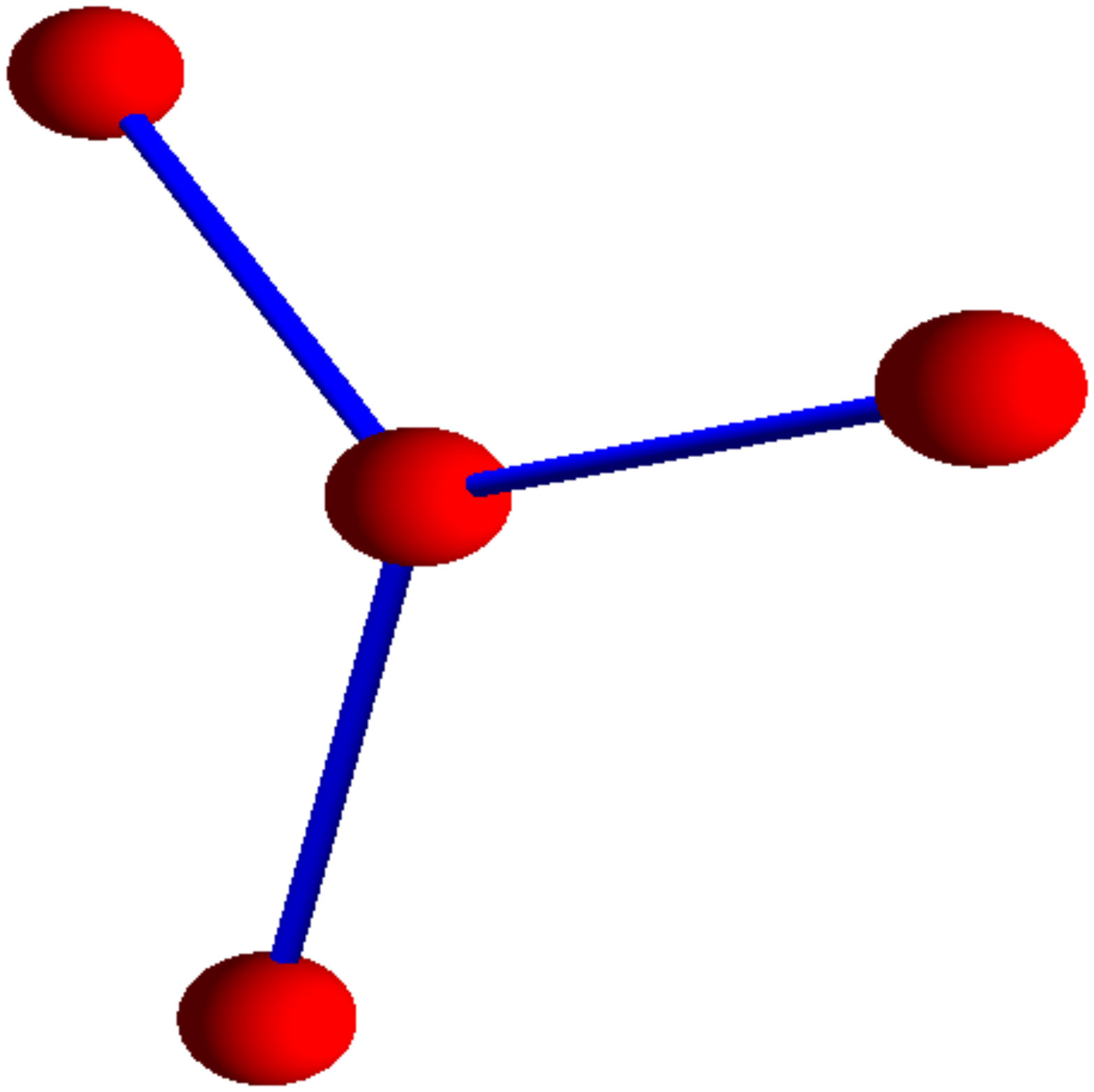}} &
\parbox{3cm}{$K_3$ is a $2$-dim graph with boundary. Each point is a boundary.            }& \scalebox{0.10}{\includegraphics{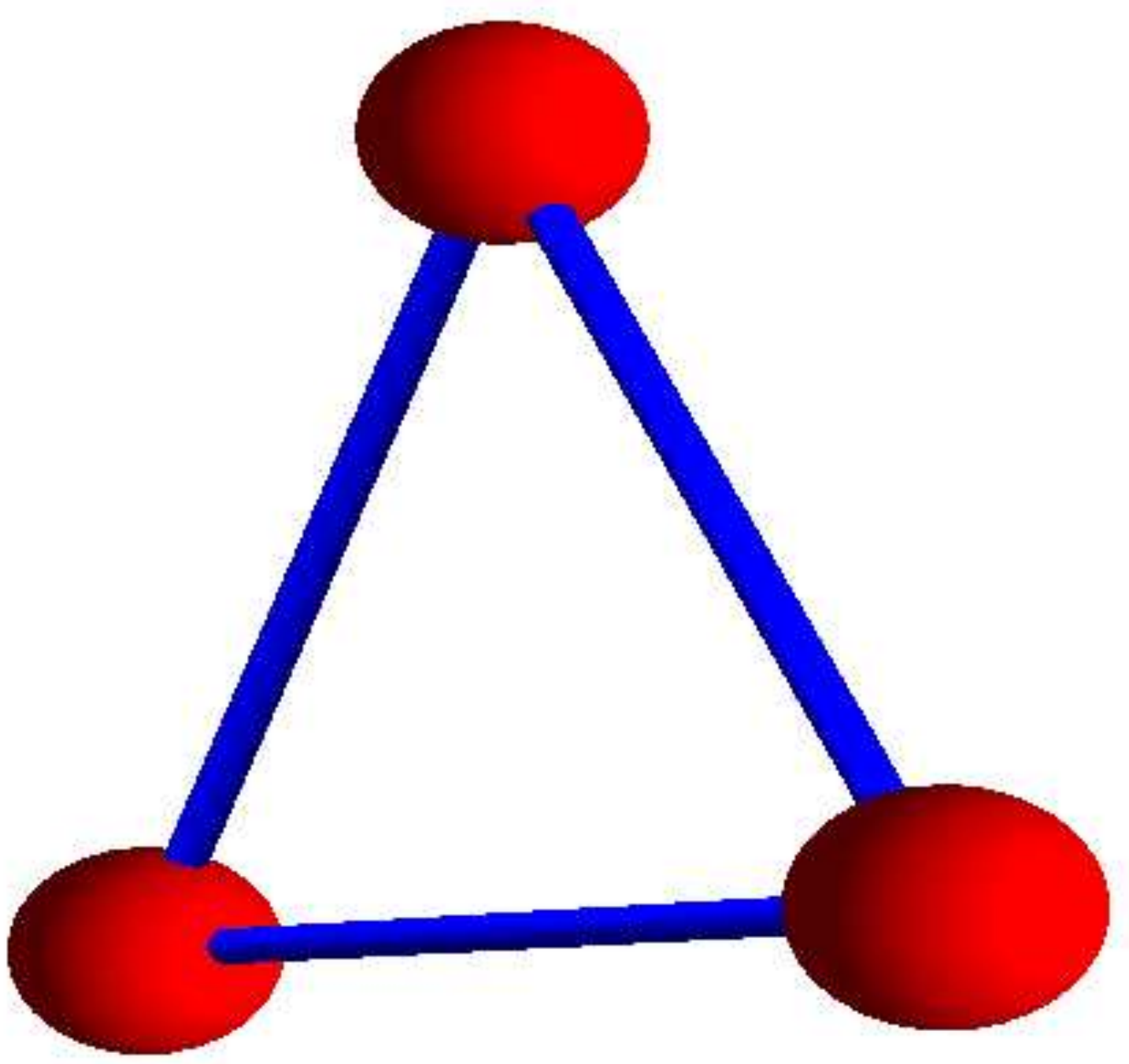}} \\
\parbox{3cm}{$W_6$ is a $2$-dimensional wheel graph with boundary.                                       }& \scalebox{0.10}{\includegraphics{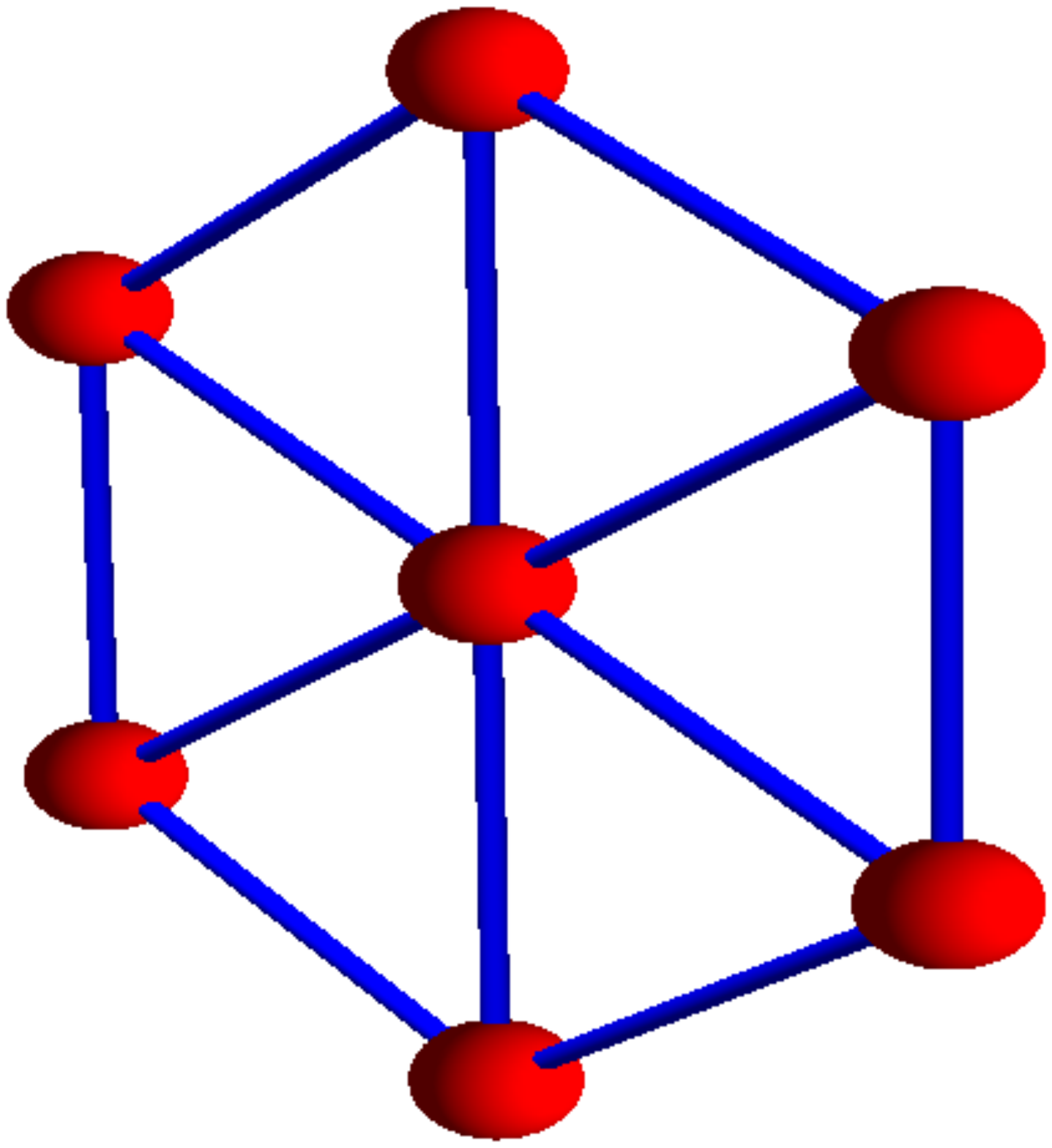}} &
\parbox{3cm}{$C_7$ is a connected $1$-dimensional graph without boundary.                                }& \scalebox{0.10}{\includegraphics{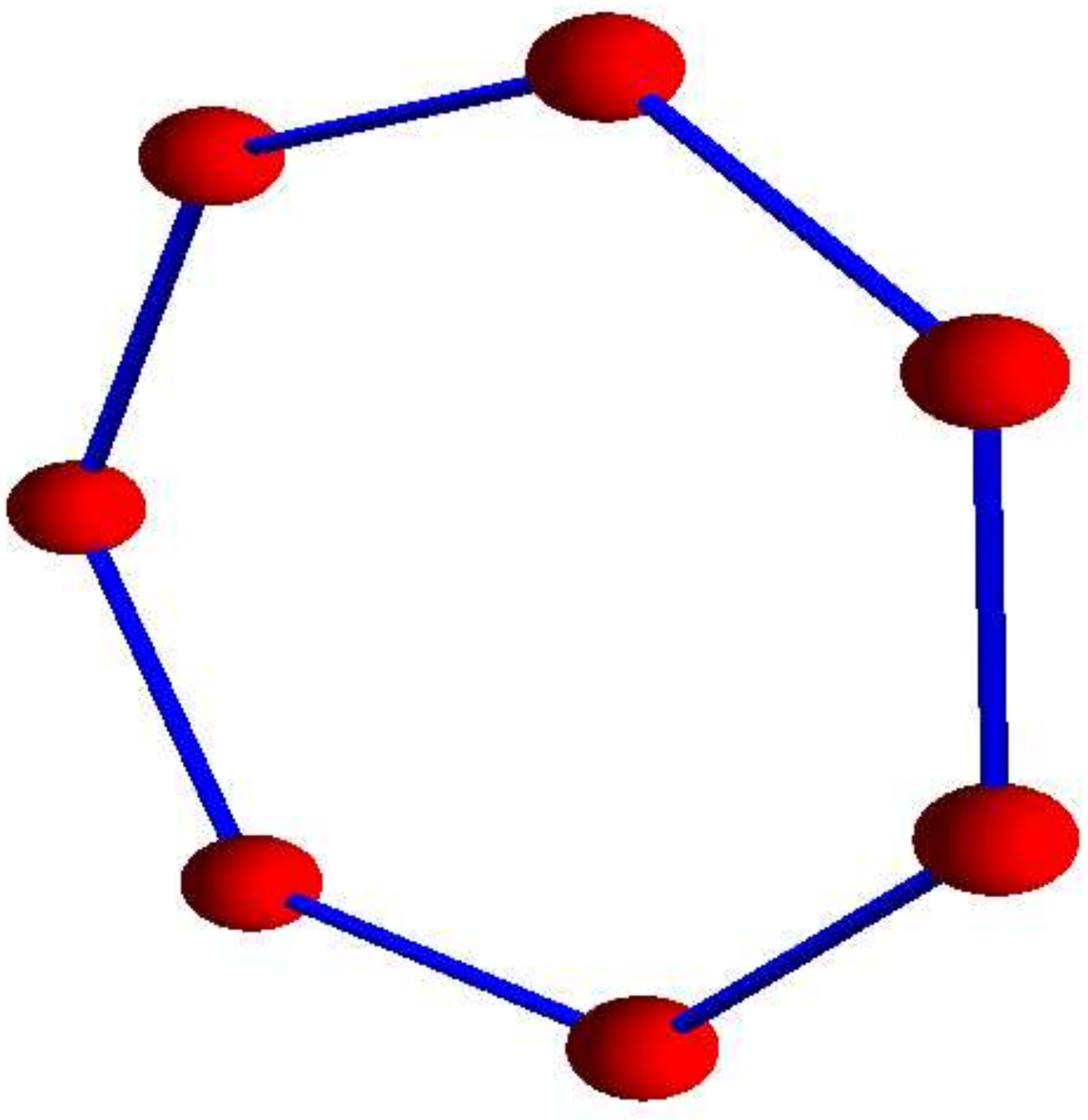}} \\
\parbox{3cm}{Hessel's example. $S(p)$ is a union of two cyclic graphs.                          }& \scalebox{0.10}{\includegraphics{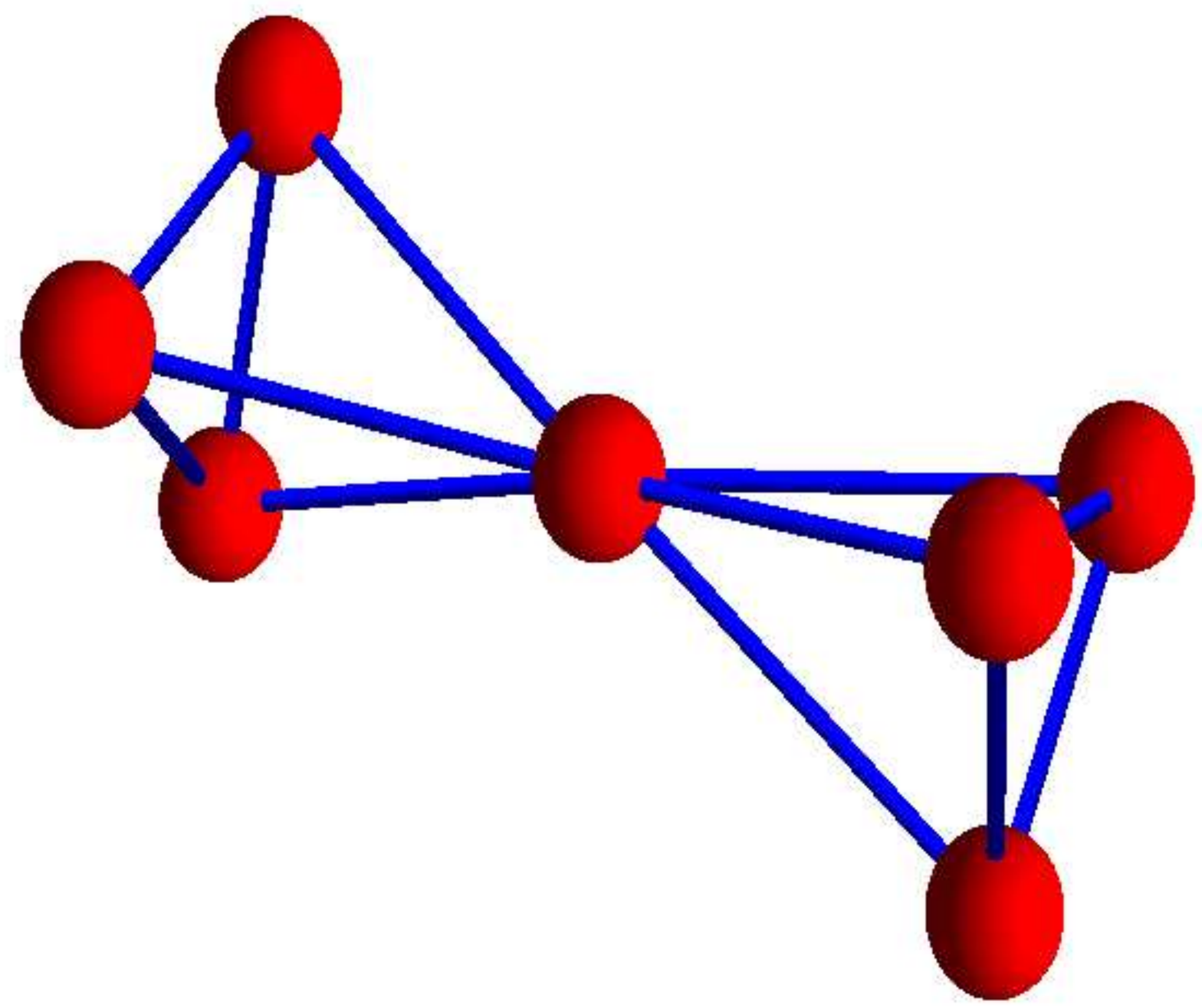}} &
\parbox{3cm}{$K_4$, the tetrahedron has dimension $3$ but has a boundary.                       }& \scalebox{0.10}{\includegraphics{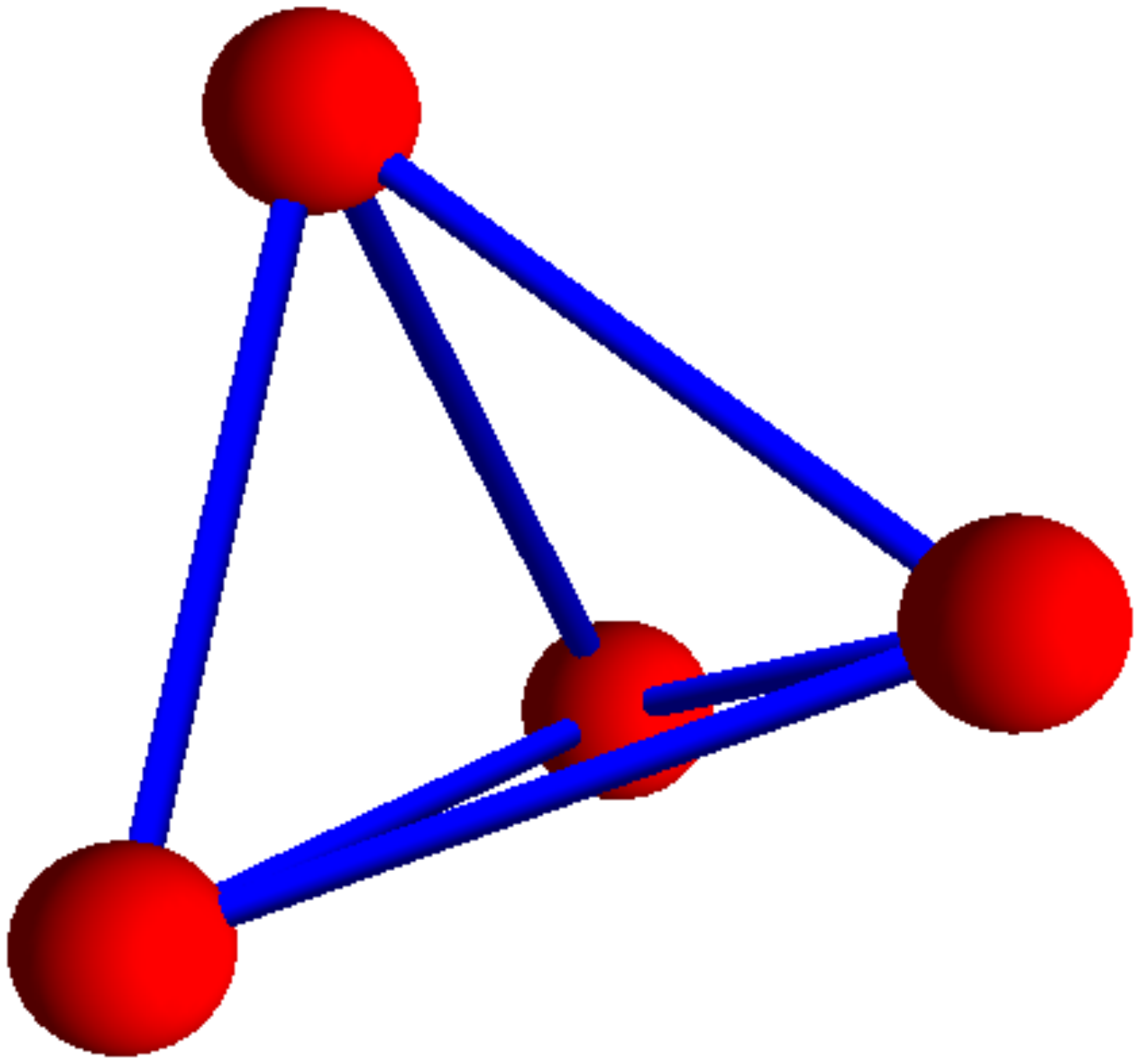}} \\
\parbox{3cm}{The utility graph $K(3,3)$ is one dimensional                                      }& \scalebox{0.10}{\includegraphics{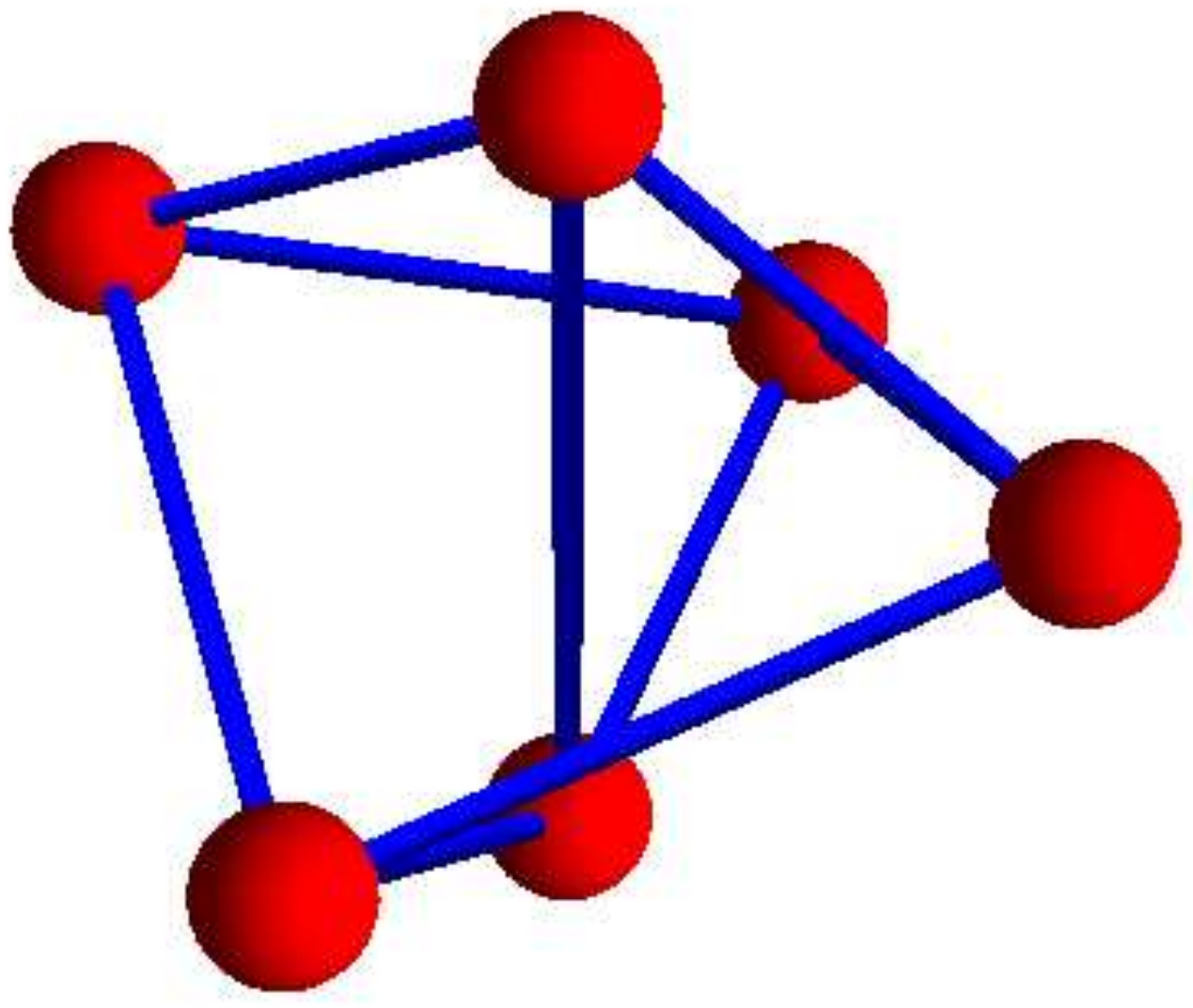}} &
\parbox{3cm}{A tri-pyramid construction of dimension 3 which contains $K(3,3)$                  }& \scalebox{0.10}{\includegraphics{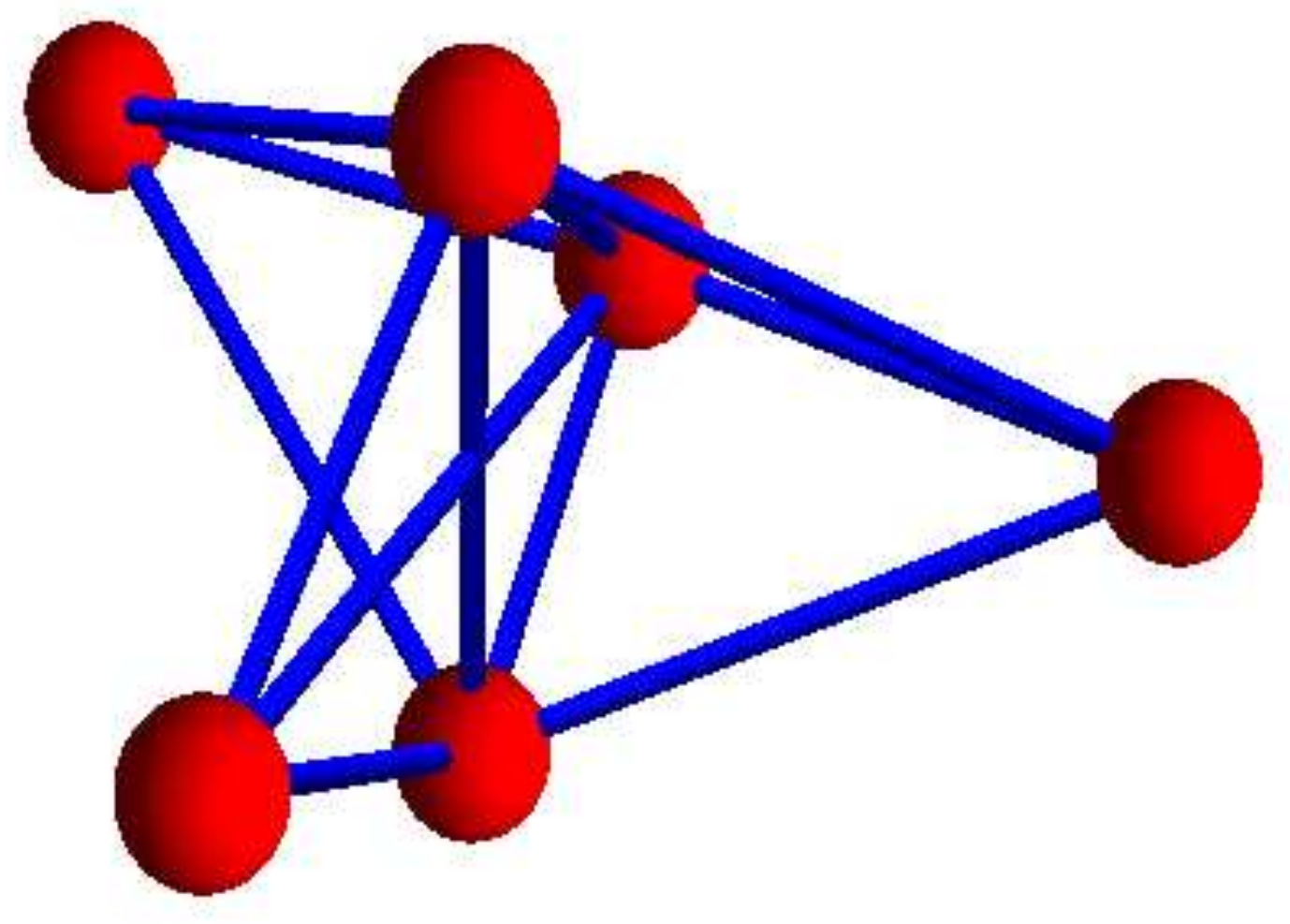}}  \\
\end{tabular}
\caption{Examples of graphs}
\end{figure}

\begin{figure}
\begin{tabular}{ll|ll}
\parbox{3cm}{A figure 8 type graph is 1-dimensional but not smooth.                            }& \scalebox{0.10}{\includegraphics{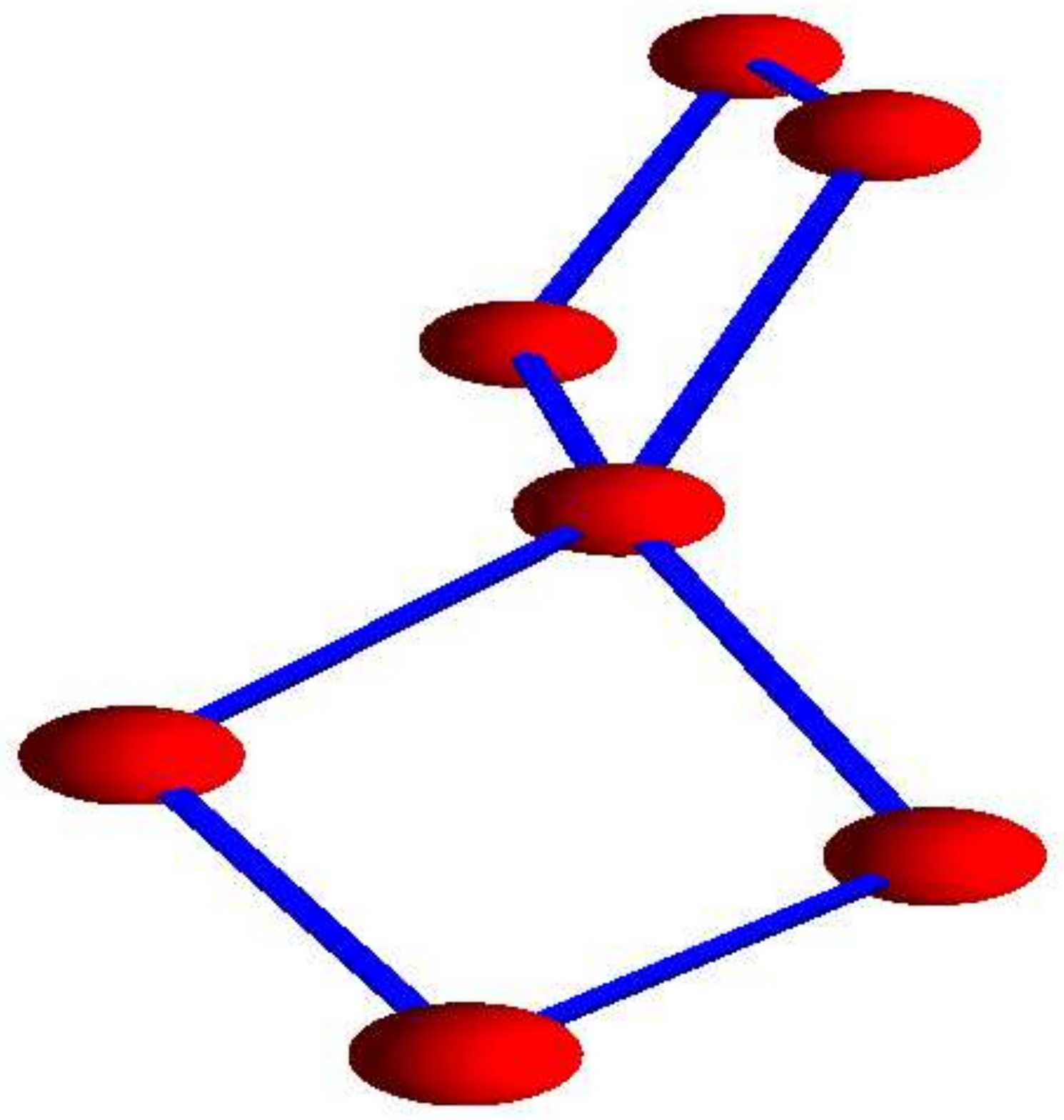}} &
\parbox{3cm}{A cube is a $1$-dim graph. It is a graph theoretical polyhedron.                  }& \scalebox{0.10}{\includegraphics{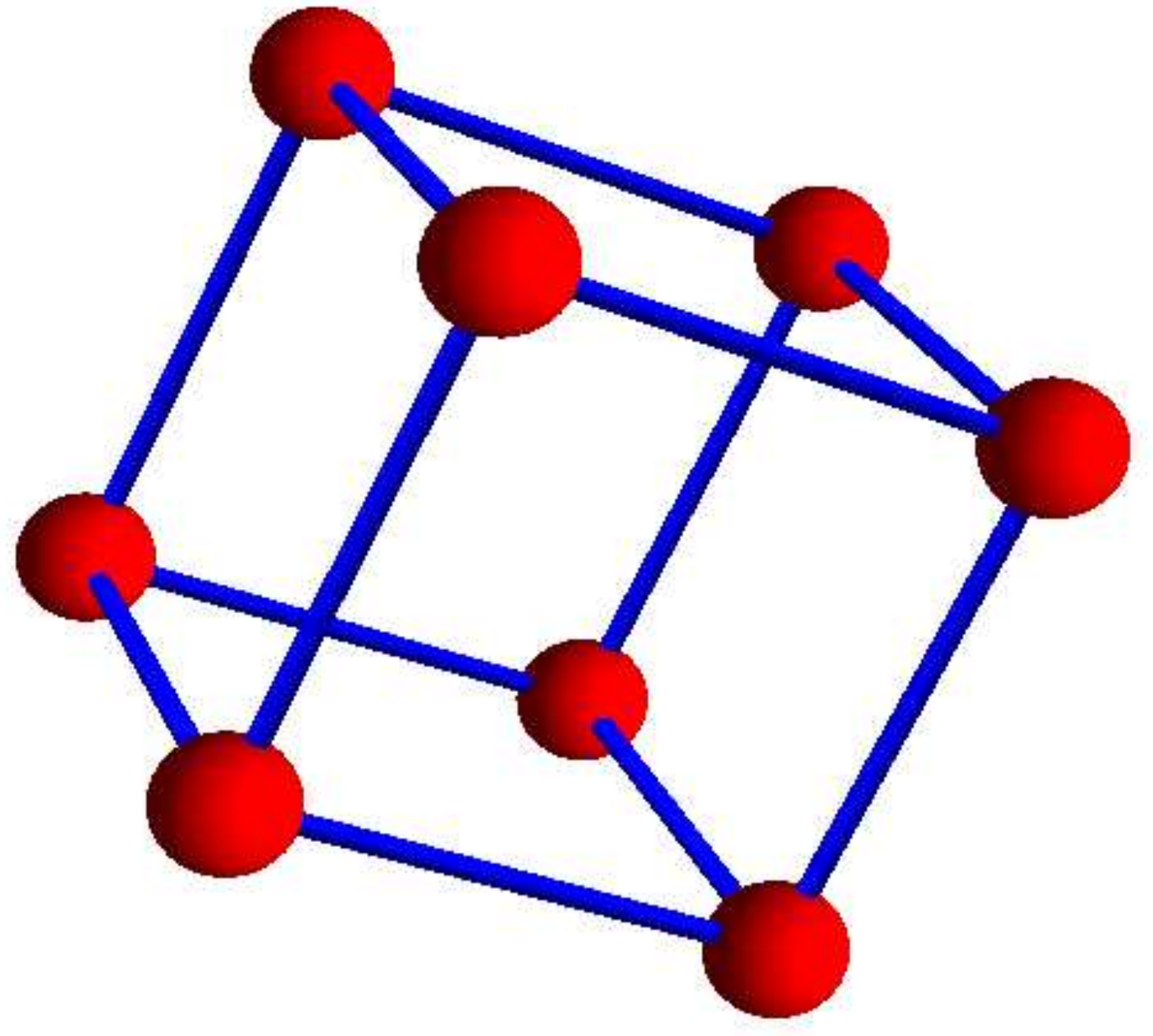}} \\
\parbox{3cm}{A stellated cube $d=2$ is a Catalan graph.                                        }& \scalebox{0.10}{\includegraphics{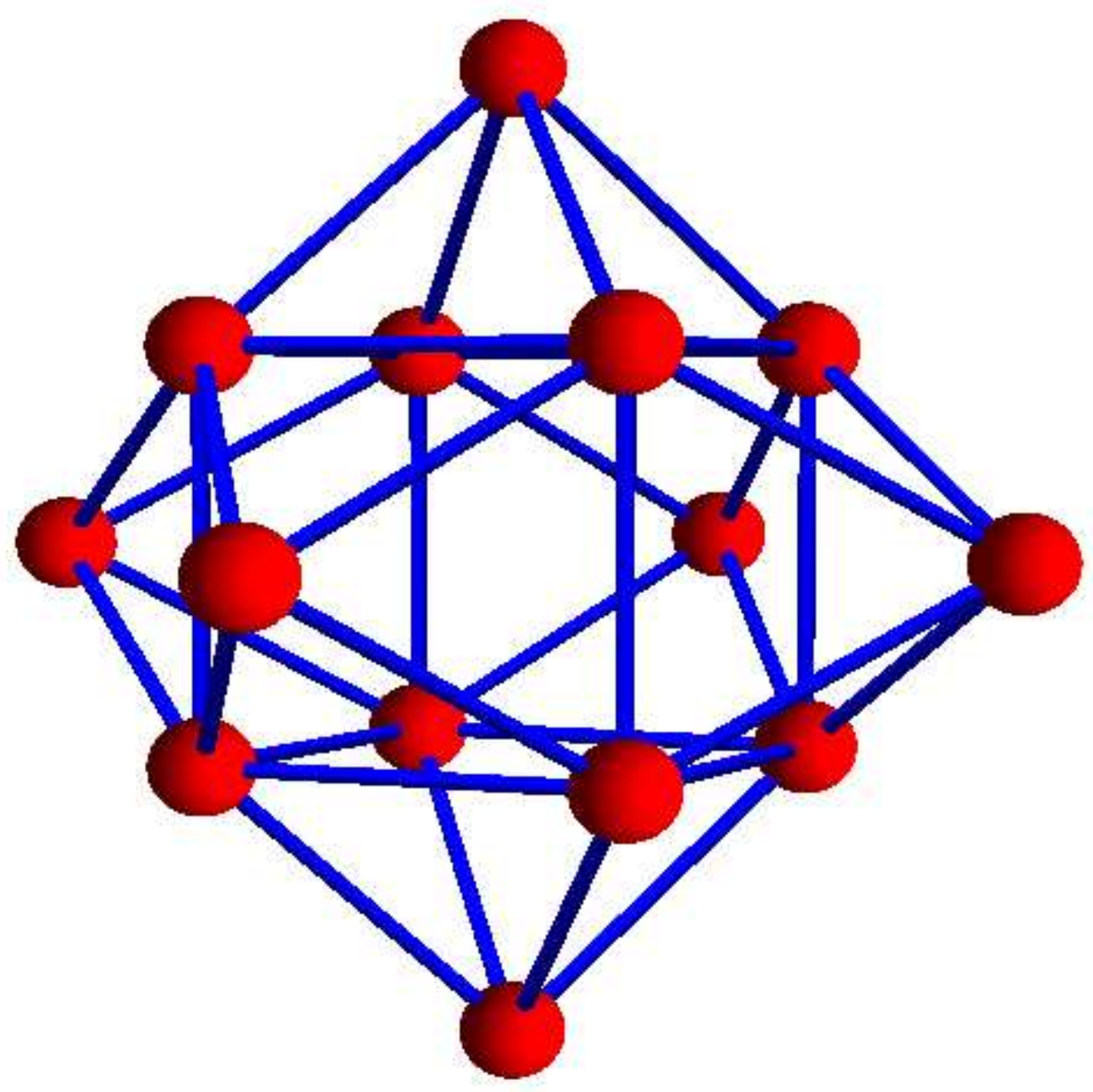}} &
\parbox{3cm}{The 3-cross-polytope = octahedron is the smallest 2-dim graph without boundary.   }& \scalebox{0.10}{\includegraphics{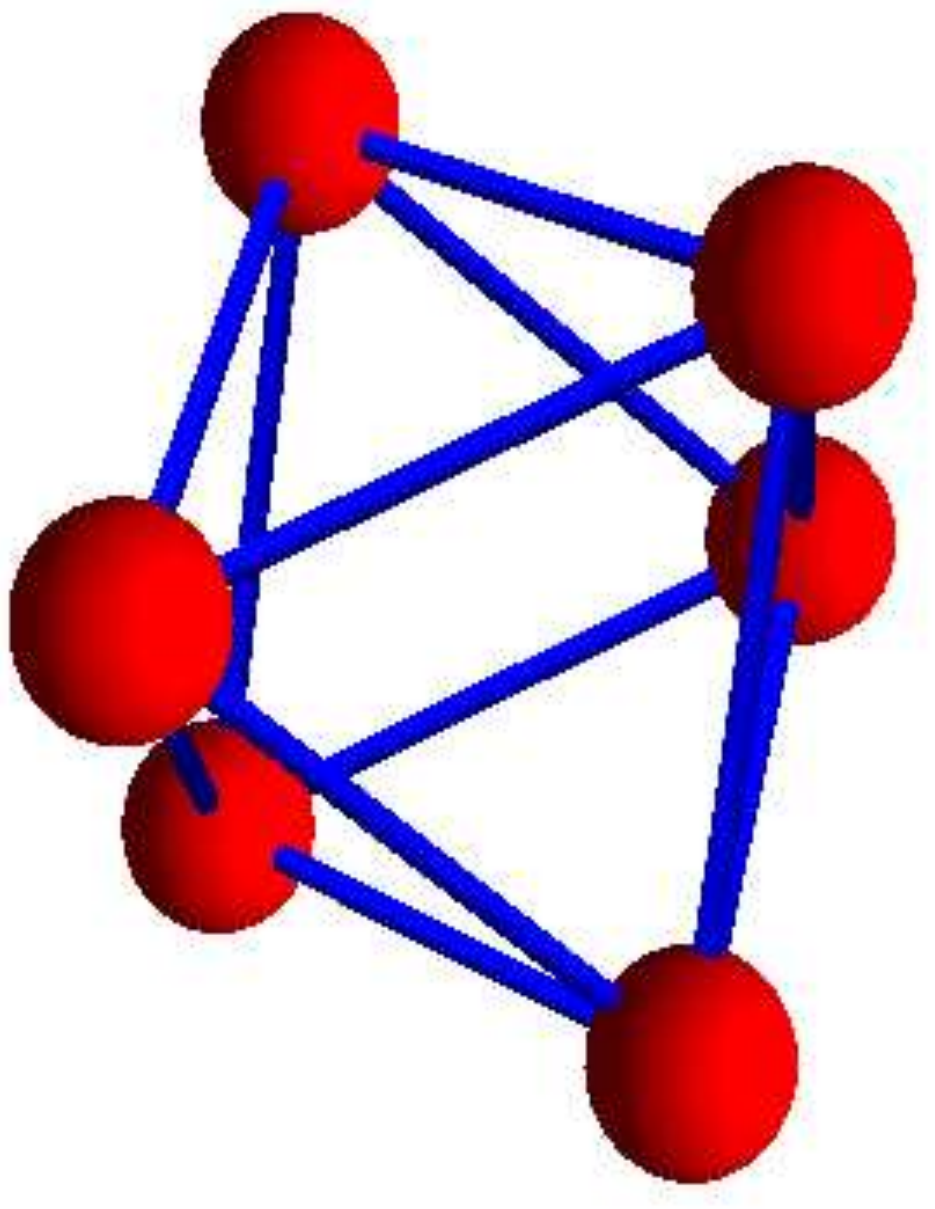}} \\
\parbox{3cm}{The 4-cross-polytope, the smallest 3 dimensional graph without boundary.          }& \scalebox{0.10}{\includegraphics{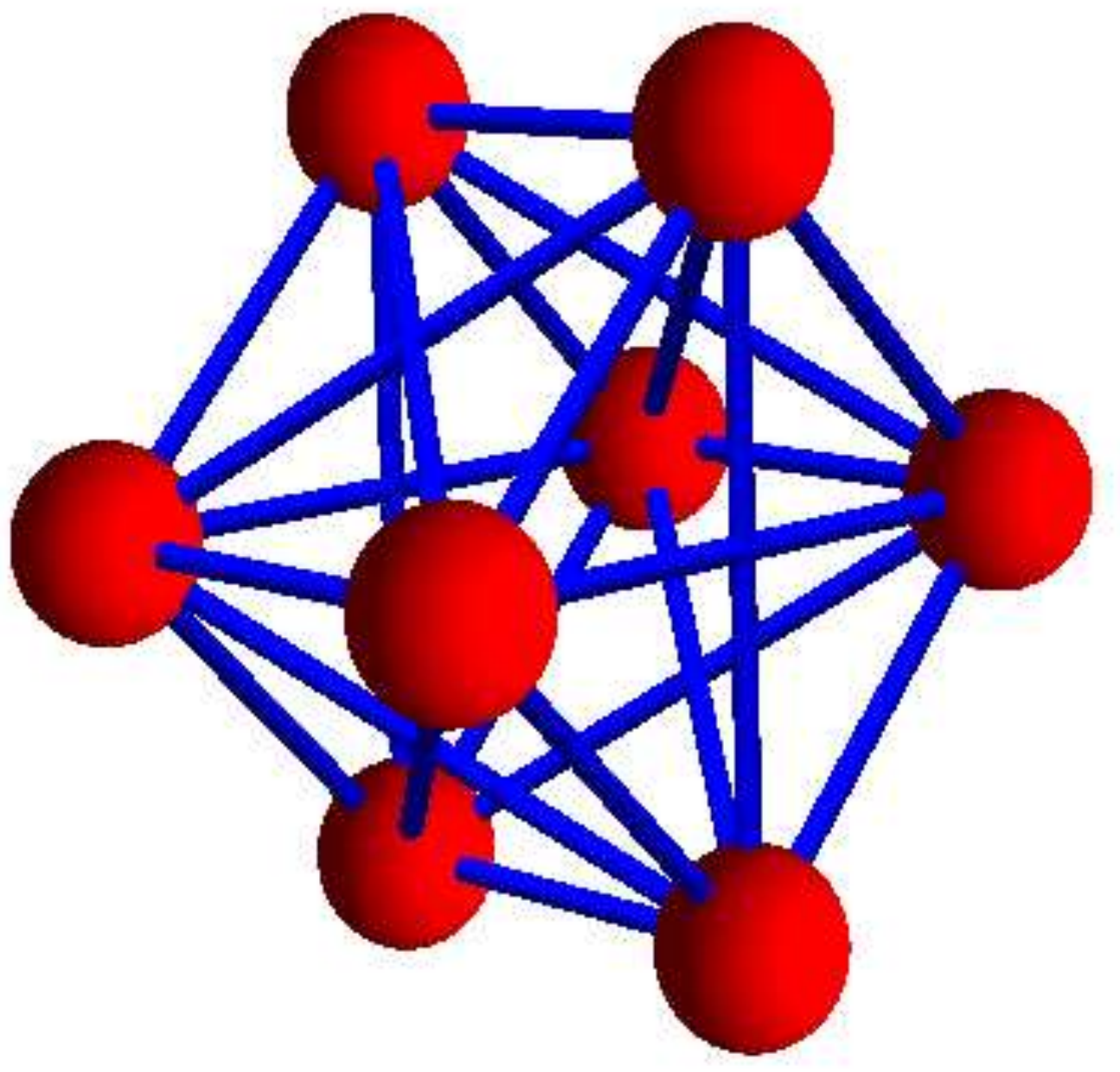}} &
\parbox{3cm}{The 5-cross-polytope is the smallest 4 dimensional graph without boundary.        }& \scalebox{0.10}{\includegraphics{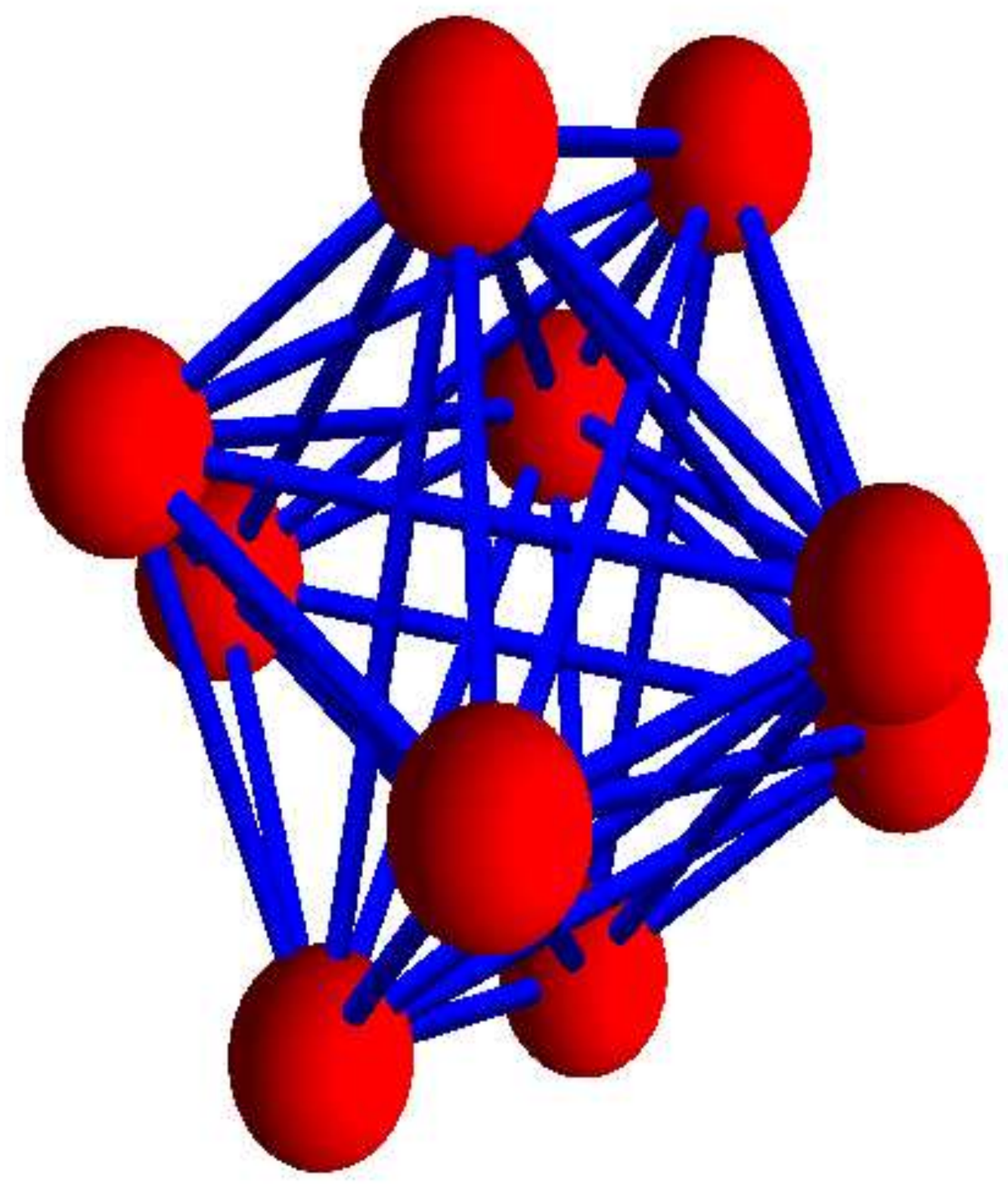}} 
\end{tabular} 
\caption{Examples of graphs. The last 4 examples are graphs we consider to be smooth. The unit 
spheres have the right Euler characteristic. }
\end{figure}

Many of the following examples are regular or semi-regular polyhedra in two and 
higher dimension \cite{Schlafli,coxeter}. Some of whom are Cayley graphs of finitely presented groups. 
More examples can be obtained from constructions. We can borrow from any
construction known to manifolds. Here are examples of constructions which can be used to generate
new graphs from other graphs or from manifolds:  \\

\begin{center}
\begin{tabular}{ll}
Pyramid    & add a vertex point and connect to all previous vertices \\
Bipyramid  & add two vertices to get a $d+1$ dimensional graph \\
Joining   & join d-dimensional graphs along $d-1$ dimensional subgraphs. \\
Identification   & identify along isomorphic $d-1$ dimensional subgraphs.  \\
Product             & produces a $d_1+d_2$ dimensional graph after  filling boundaries  \\
Triangularization      & produces d-dimensional graphs from $d$-dimensional manifolds \\
Conway            & kis, snub, truncate, gyro, propel etc to generate new graphs. \\
\end{tabular}
\end{center}

\vspace{5mm}

The classical pyramid, bipyramid and prism constructions discussed for example in \cite{gruenbaum} 
in the case of convex polytopes lead to higher dimensional regular polytopes. 
The pyramid construction does not produce $(d+1)$-dimensional 
graphs unless the graph $H$ is a simplex.  \\

If $H=(W,F)$ is a $d$-dimensional graph with Euler characteristic $\chi(H)$, 
then the bi-pyramid construction $G=(V \cup \{p,q \;\},F \cup \{ (w,p),(w,q) \; \}_{w \in W} )$ 
is a $(d+1)$-dimensional graph in the wider sense with Euler characteristic $\chi(G) = 2-\chi(H)$.
\begin{proof}
The original graph $H$ is the unit sphere of both $p$ and $q$. If $w \in W$ is a vertex in the old graph $H$,
then its unit sphere is $S_1(w) \cup \{p,q \; \}$, where $S_1(w)$ is the unit sphere of $w$ in the old graph.
The new unit sphere is $d$-dimensional. Since the unit spheres are either double pyramid constructions
of old unit spheres or coincide with $H$, the first claim follows by induction.
If $H$ is two-dimensional with $v'$ vertices, $e'$ edges and $f'$ faces, then $G$ has $s=2f'$ spaces
and $f=f'+2e'$ faces, $e=e'+2v'$ edges and $v=v'+2$ vertices. If the Euler characteristic of $H$
is $\chi'=v'-e'+f'$, then the Euler characteristic of the bipyramid extension $G$ is
$\chi = v-e+f-s=(v'+2)-(e'+2v') + (f'+2e')-2f' = 2-v'+e'-f' = 2-\chi'$.
\end{proof}

It follows that for every $d$-dimensional graph $G$ there exists a $d+1$-dimensional graph (in the wider sense without
any assumptions on the unit sphere) such that the unit sphere is $G$.
The topology of the unit sphere is needed if we want to use only $d-2$ sphere quantities.
For $d$-dimensional graphs without any assumptions on the unit spheres,
the unit spheres can be quite arbitrary. Classical differential geometry
motivates assumptions on the unit spheres for graphs. For radii smaller than
the injectivity radius of a Riemannian manifold $M$, the geodesic sphere $S_r(p)$ is a
topological sphere with Euler characteristic $2 e_{d-1}$ which 
is $0$ for even $d$ and $2$ for odd $d>1$. Its natural to assume this for graphs too.  \\

\begin{figure}
\scalebox{0.20}{\includegraphics{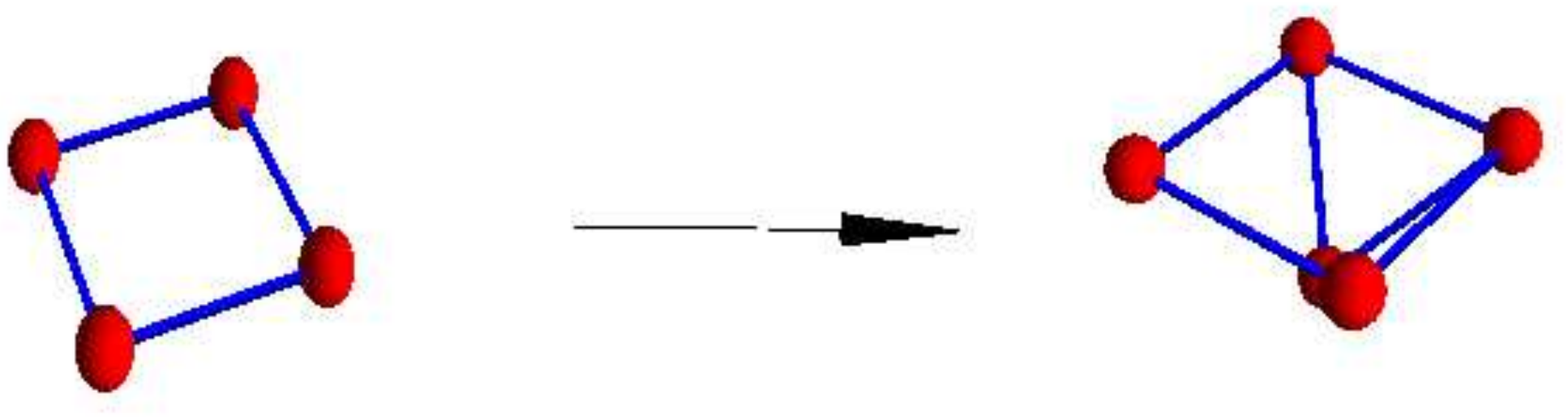}}
\scalebox{0.20}{\includegraphics{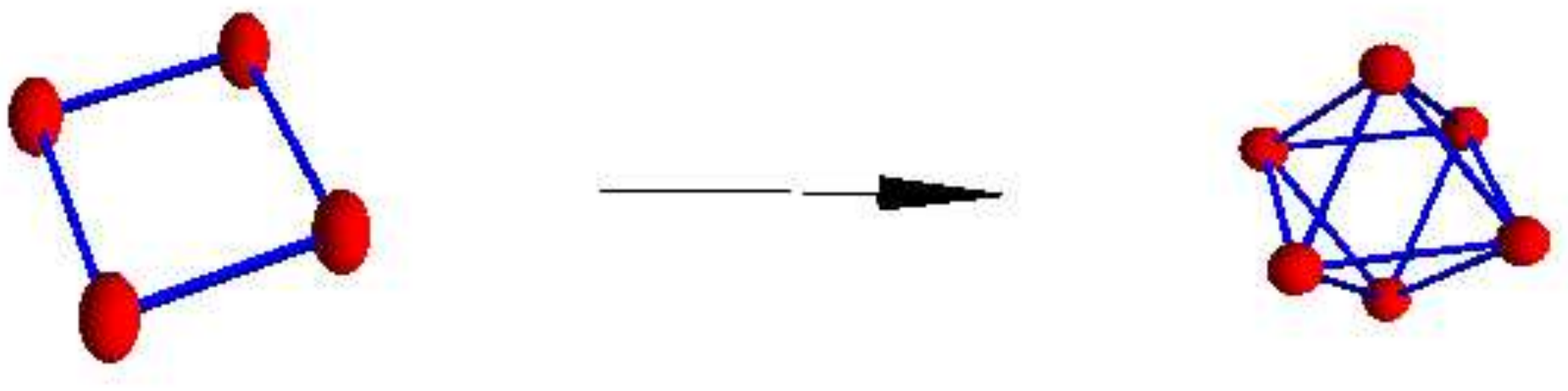}}
\caption{
The pyramid and bipyramid constructions obtained by a graph of dimension $d$ 
produce a bipartite graphs of dimension $d+1$. The can be generalized to multipyramid
constructions.  }
\label{4crosspolytope}
\end{figure}

\noindent
{\bf Remark 1.} The cross polytope in any dimension is obtained from the bipyramid construction 
starting with the zero dimensional point. We can take any two dimensional graph $H=(W,F)$,
add two points $\{ p,q \; \}$ and get a three dimensional graph
$G=(V \cup \{p,q \; \},F \cup \{ (w,p),(w,q) \; \}_{w \in W})$ etc. \\

{\bf Remark 2.} After repeating such a construction, we have a $(d+2)$-dimensional 
graph with the same Euler characteristic. It can be repeated. To get a $100$-dimensional
connected graph with Euler characteristic $-24$, start with a two dimensional graph with
this Euler characteristic, then do the bi-pyramid construction 98 times. \\

Given two $d$ dimensional graphs $G,H$ and assume that there are $(d-1)$ dimensional 
spheres $S_r(p) \subset G, S_r(q) \subset H$ with a graph isomorphisms 
$\phi: S_r(p) \to S_r(q)$. The joined graph $G \cup_{\phi} H$ by identifying points along $\phi$
is a $d$-dimensional graph of Euler characteristic $\chi(G) +\chi(H) -2 - \chi(S_r)$.
\begin{proof}
We have to show that (i) each unit sphere in the new graph is a $(d-1)$-dimensional graph, that (ii) each 
unit sphere has  Euler characteristic $1+(-1)^{d-1}$ and (iii) that $(d+1) v_d  = 2 v_{d-1}$ for the new graph. 
The statement is true for one dimensional graphs. We use induction to verify it in higher dimensions. 
Statements (i),(ii) need only to be proven for points on $S_1(p)$ or $S_1(q)$. 
The surgery has induced a smaller dimensional surgery for the unit spheres along $d-2$ 
dimensional graphs. By induction, the dimension is correct.
The Euler characteristic of each unit sphere stays because the $\chi(S_1)$ is the same 
than the Euler characteristic of union of the complement with the boundary. 
The joining does not change the property that every face of a $d$-dimensional simplex is connected to a 
$d$-dimensional simplex. Therefore $(d+1) v_d  = 2 v_{d-1}$ stays true. 
Finally, the Euler characteristic formula for the entire follows because the simplices 
add up, we remove two balls with Euler characteristic $1$ and double-count the intersection $S_r$. 
\end{proof}

For example, joining two simply connected two-dimensional graphs of Euler characteristic $2$ 
along a one-dimensional graph has Euler characteristic $2+2-2-0 = 2$. The new graph is still simply 
connected. If we join two $2$-dimensional doughnut graphs of Euler characteristic $0$
along a circle, we get a graph with Euler characteristic $0+0-2-0 = -2$.  \\

\begin{figure}
\parbox{9.0cm}{ \scalebox{0.34}{\includegraphics{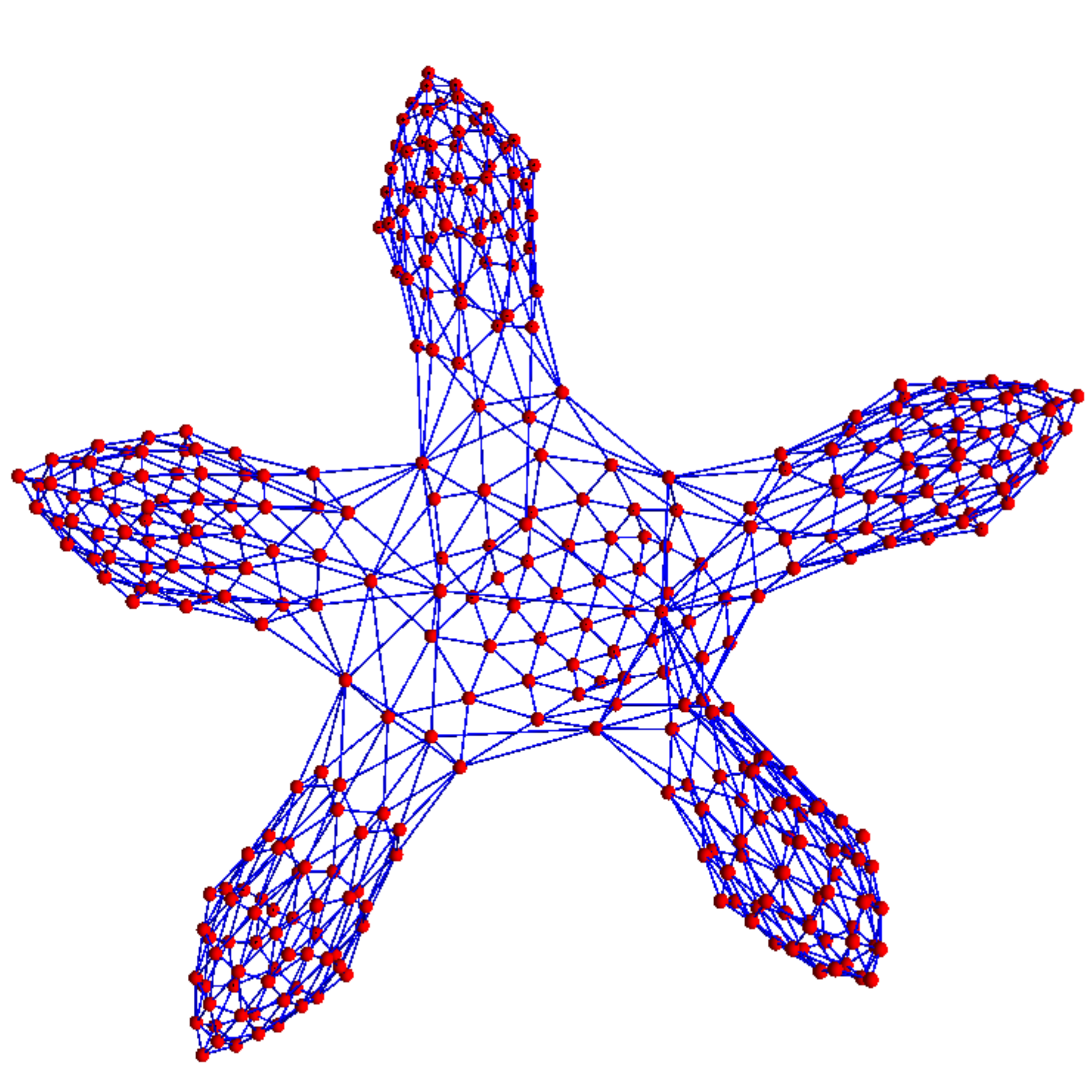}} }
\caption{
An example of a graph obtained by joining 6 fullerene type graphs along circles. 
The resulting graph is two dimensional but the curvature takes now both positive and 
negative signs. The Euler characteristic is still $2$ in this case. 
}
\label{graphjoinings}
\end{figure}

Also when identifying points of the graph, we get a new graph of the same dimension in general:
given a graph $G$ and assume that there are $(d-1)$-dimensional spheres $S_r(p) \subset G, S_r(q) \subset G$
for which $d(p,q)>r$ and such that there is a graph isomorphism: $\phi: S_r(p) \to  S_r(q)$. 
By identifying points along $\phi$, we get a new $d$-dimensional graph of Euler 
characteristic $\chi(G) - 2 - \chi(S_r(p))$.
For two-dimensional graphs, this identification corresponds to adding "handles" 
so that the new graph has Euler characteristic $\chi(G)-2$. \\

Finally, one could ask about using random processes to generate d-dimensional graphs. Among 
all $2^{n(n-1)/2}$ undirected graphs with $n$ vertices the $d$-dimensional ones of course have probability 
which goes to zero for $n \to \infty$. But we can generate graphs by "aggrevation" starting with a ball
of radius one, adding more balls. For two-dimensional graphs we will end up with a planar graph which 
when finite, has Euler characteristic $2$. Picture~(\ref{graphjoinings}) gives an example.

\begin{center}
\begin{tabular}{llllll} 
polyhedron      &  vertices  &  edges &   faces &  curvatures & $\chi$ \\ \hline
octahedron      &  6         &  12    &   8     &  2/6        & $2$ \\
icosahedron     &  12        &  30    &   20    &  1/6        & $2$ \\
stellated cube  &  14        &  36    &   24    &  0,1/3      & $2$ \\
drilled cube    &  16        &  48    &   32    &  0          & $0$ \\ 
drilled stellated cube &  16     &  96    &   64    &  1/3,-1/3   & $0$ \\  
\end{tabular}
\end{center} 

\begin{figure}
\parbox{14.8cm}{
\parbox{7.1cm}{
\begin{tabular}{lll}
cube         &  1  & \scalebox{0.08}{\includegraphics{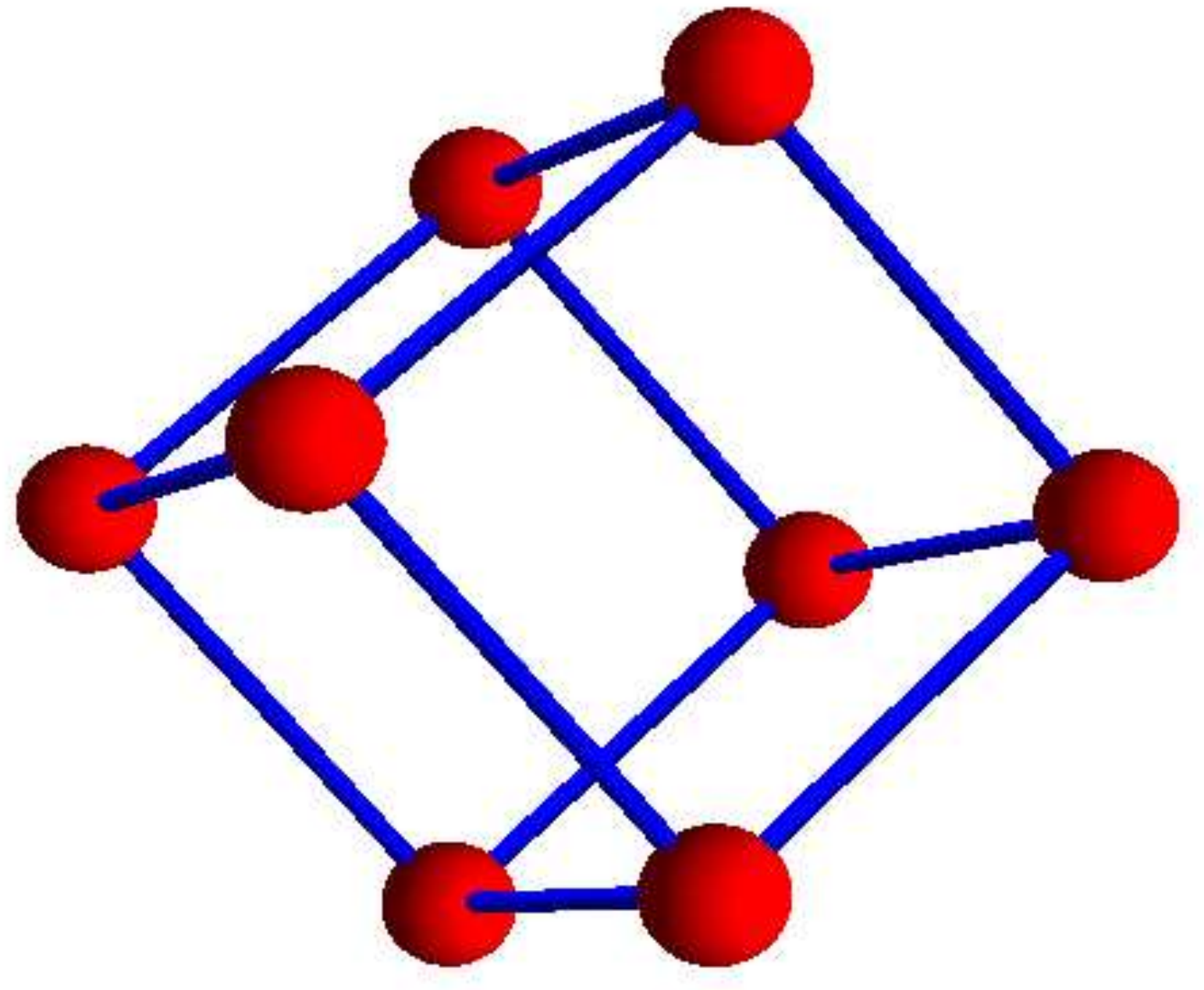}}  \\
tetrahedron  &  3  & \scalebox{0.08}{\includegraphics{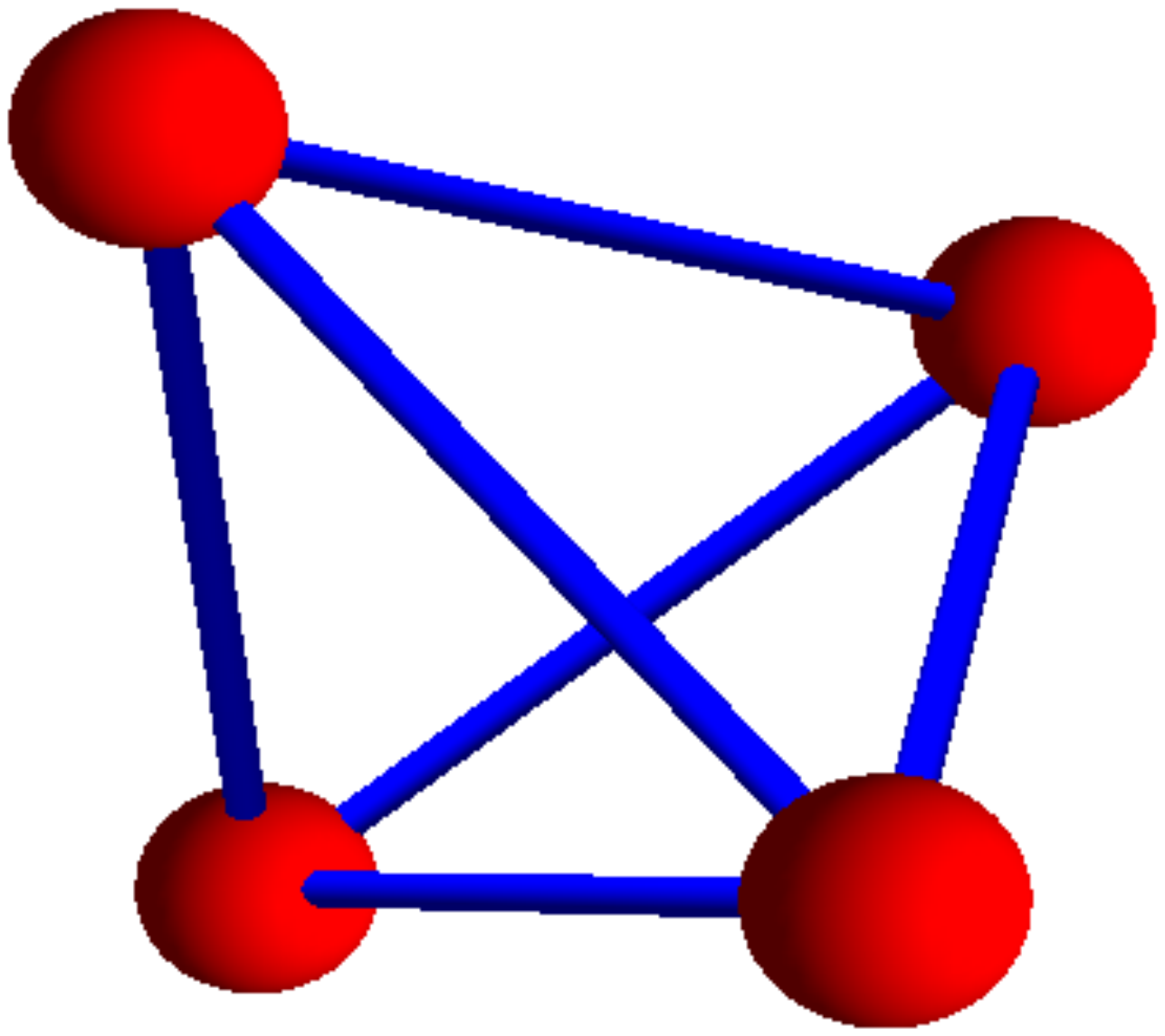}}  \\
octahedron   &  2  & \scalebox{0.08}{\includegraphics{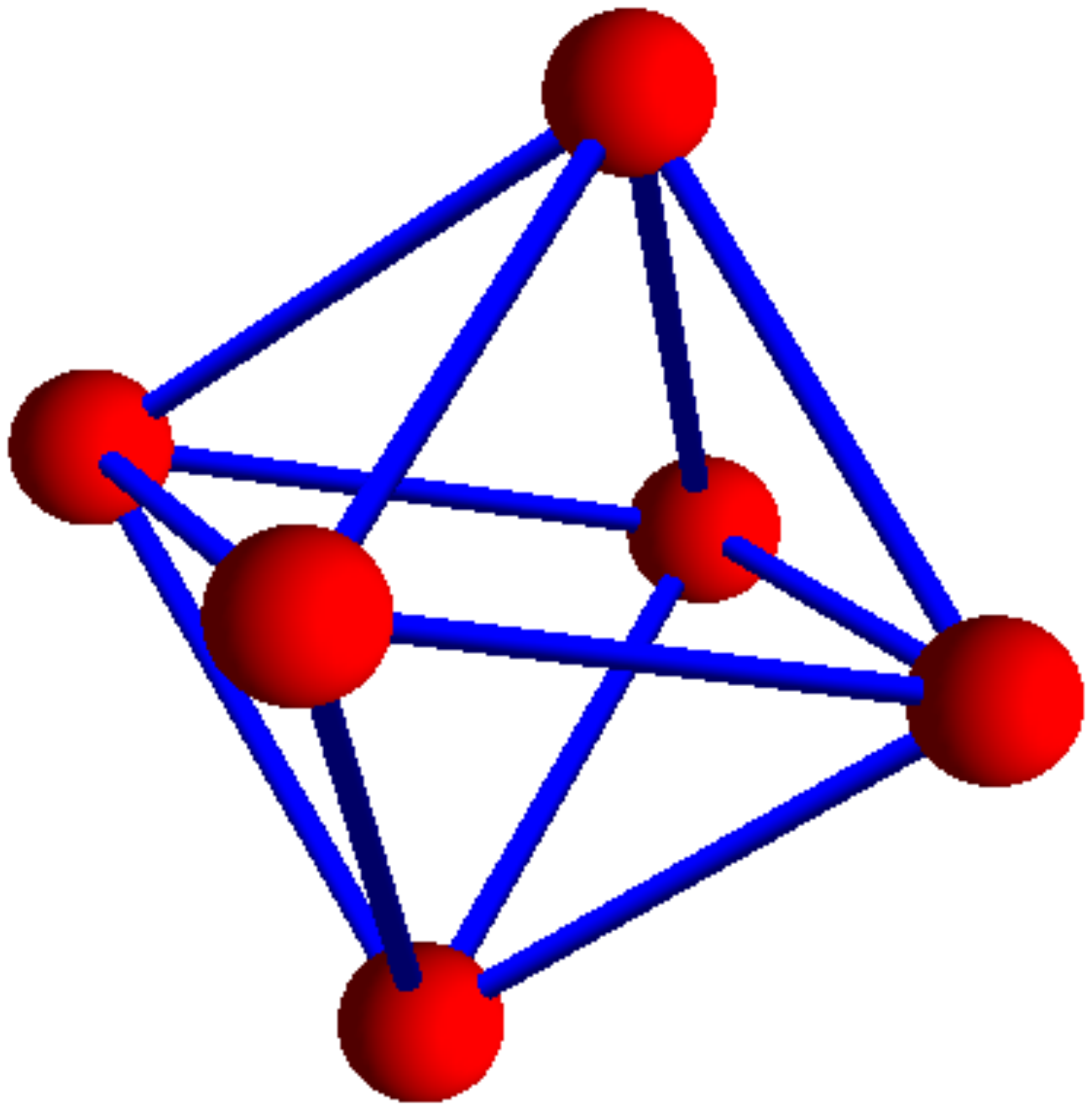}}  \\
icosahedron  &  2  & \scalebox{0.08}{\includegraphics{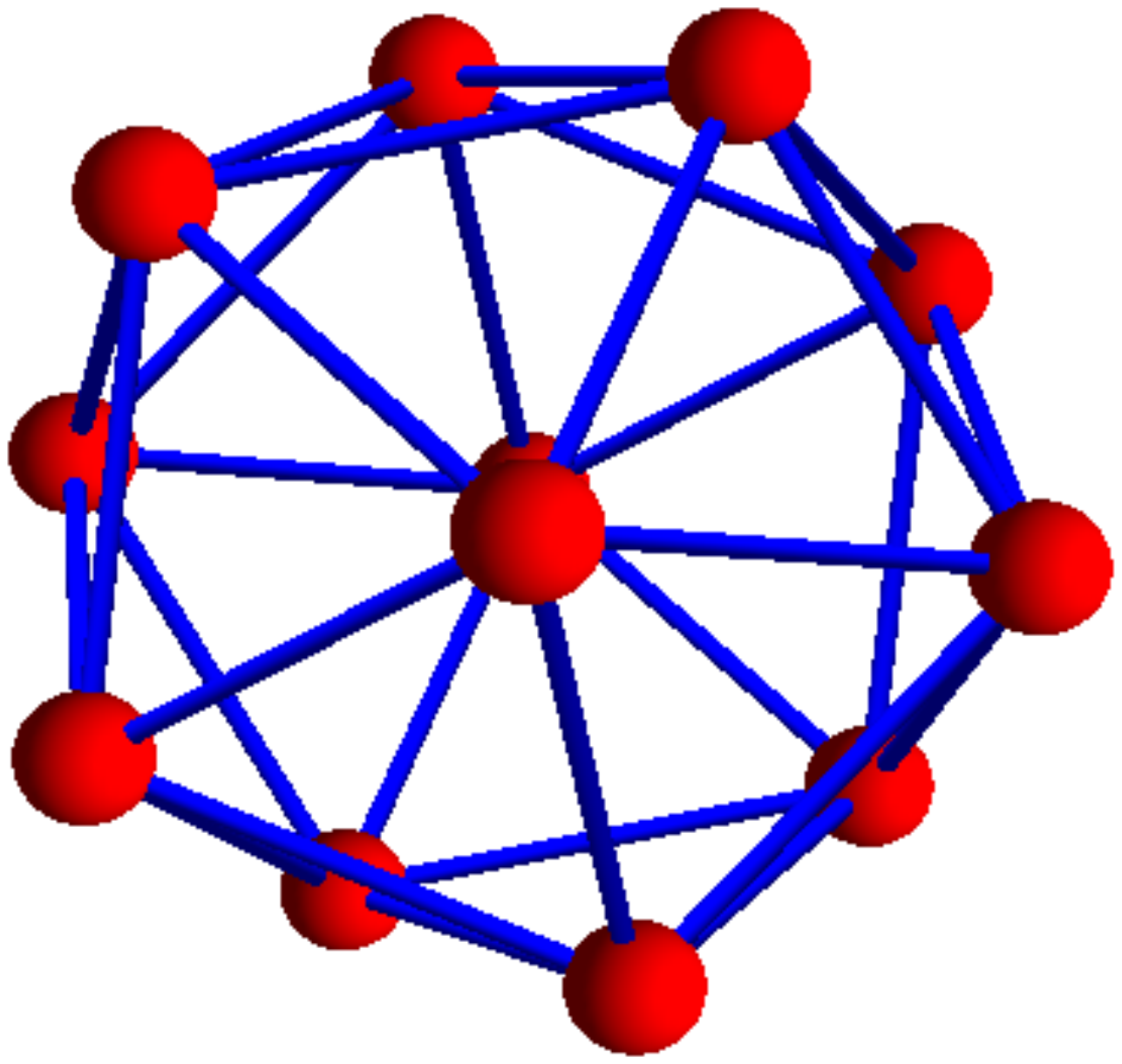}}  \\
dodecahedron &  1  & \scalebox{0.08}{\includegraphics{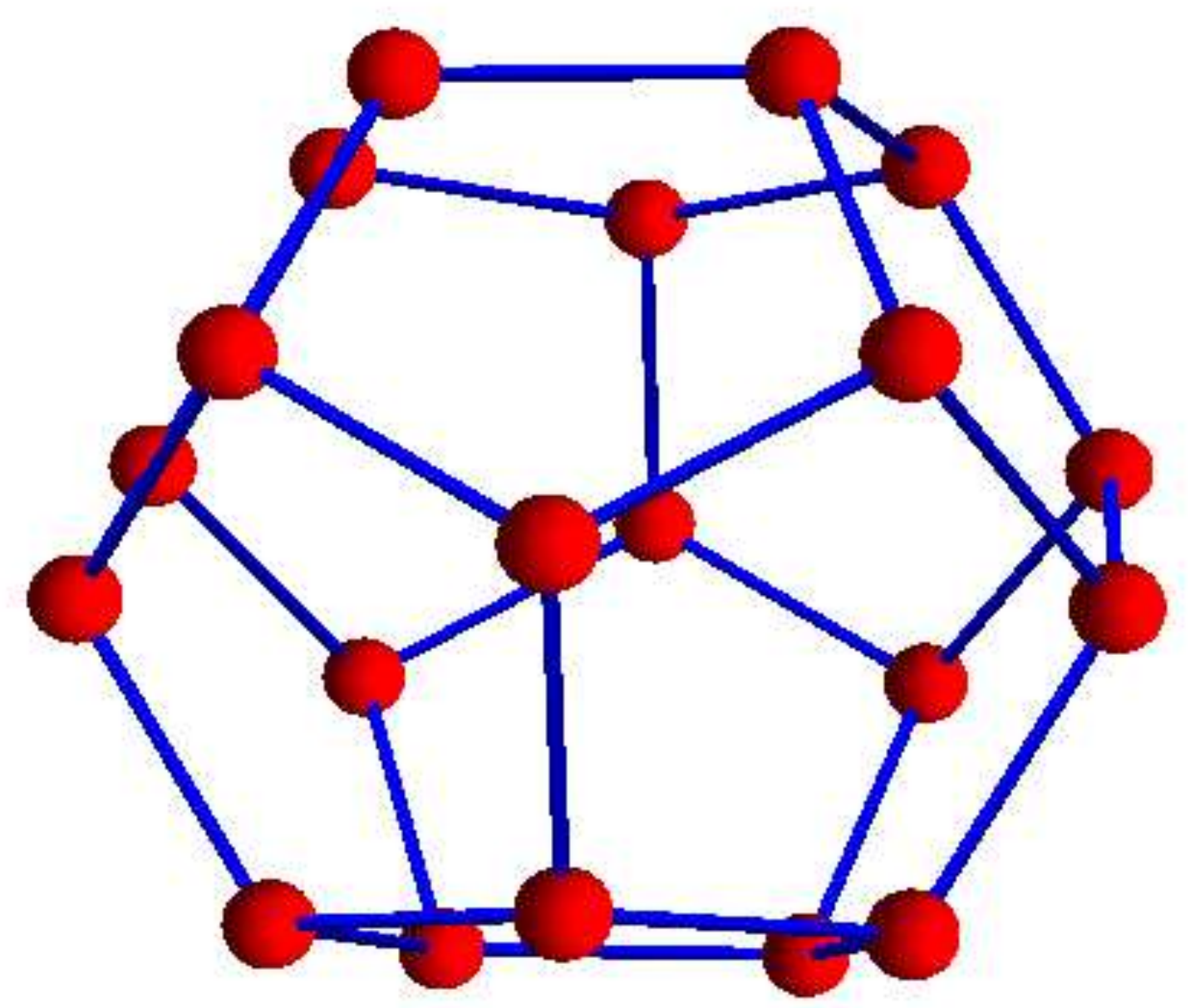}} 
\end{tabular} }
\parbox{7.1cm}{
\begin{tabular}{lll}
stellated cube &  1  & \scalebox{0.08}{\includegraphics{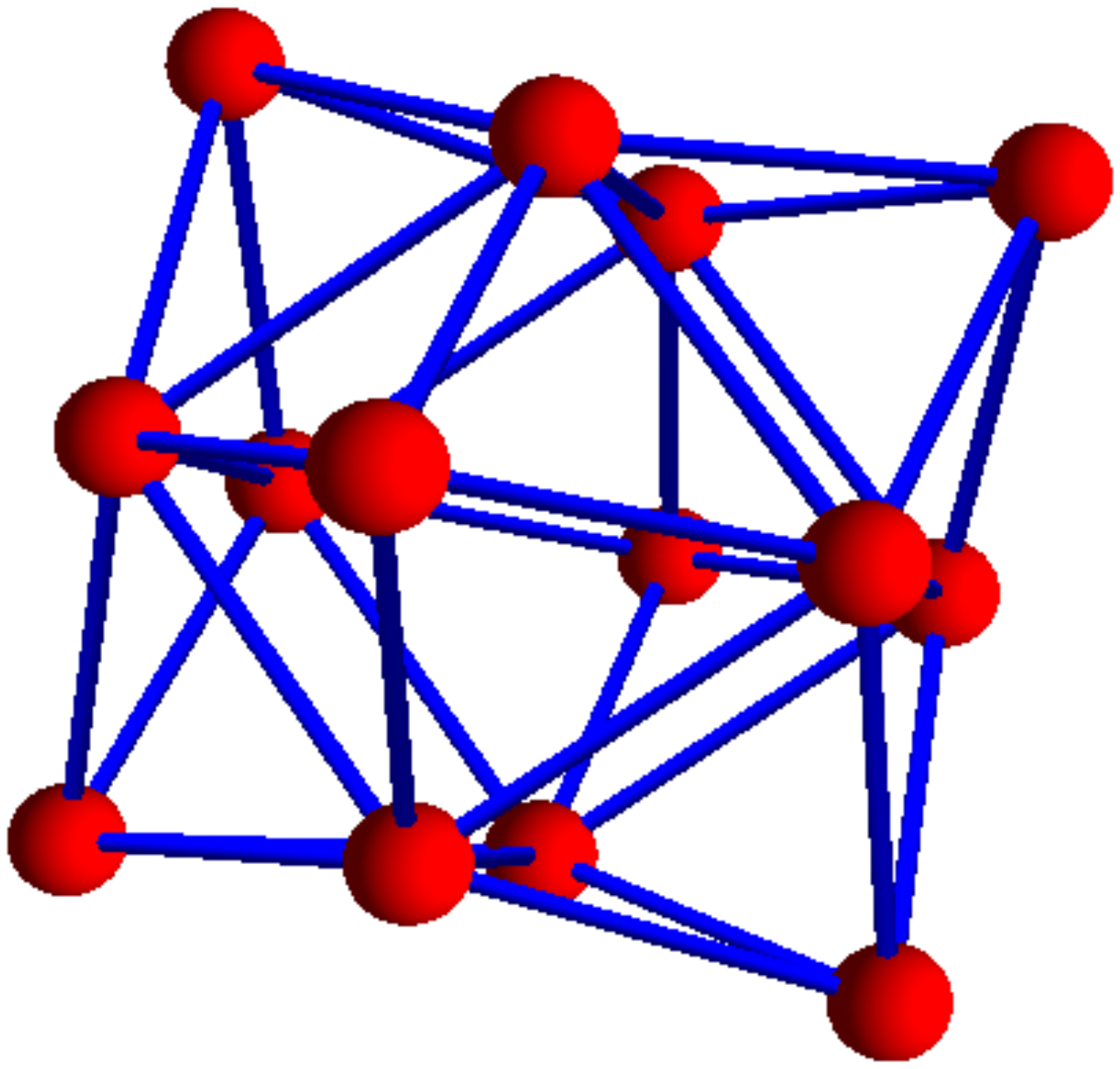}}  \\
stellated dodecahedron &  2  & \scalebox{0.08}{\includegraphics{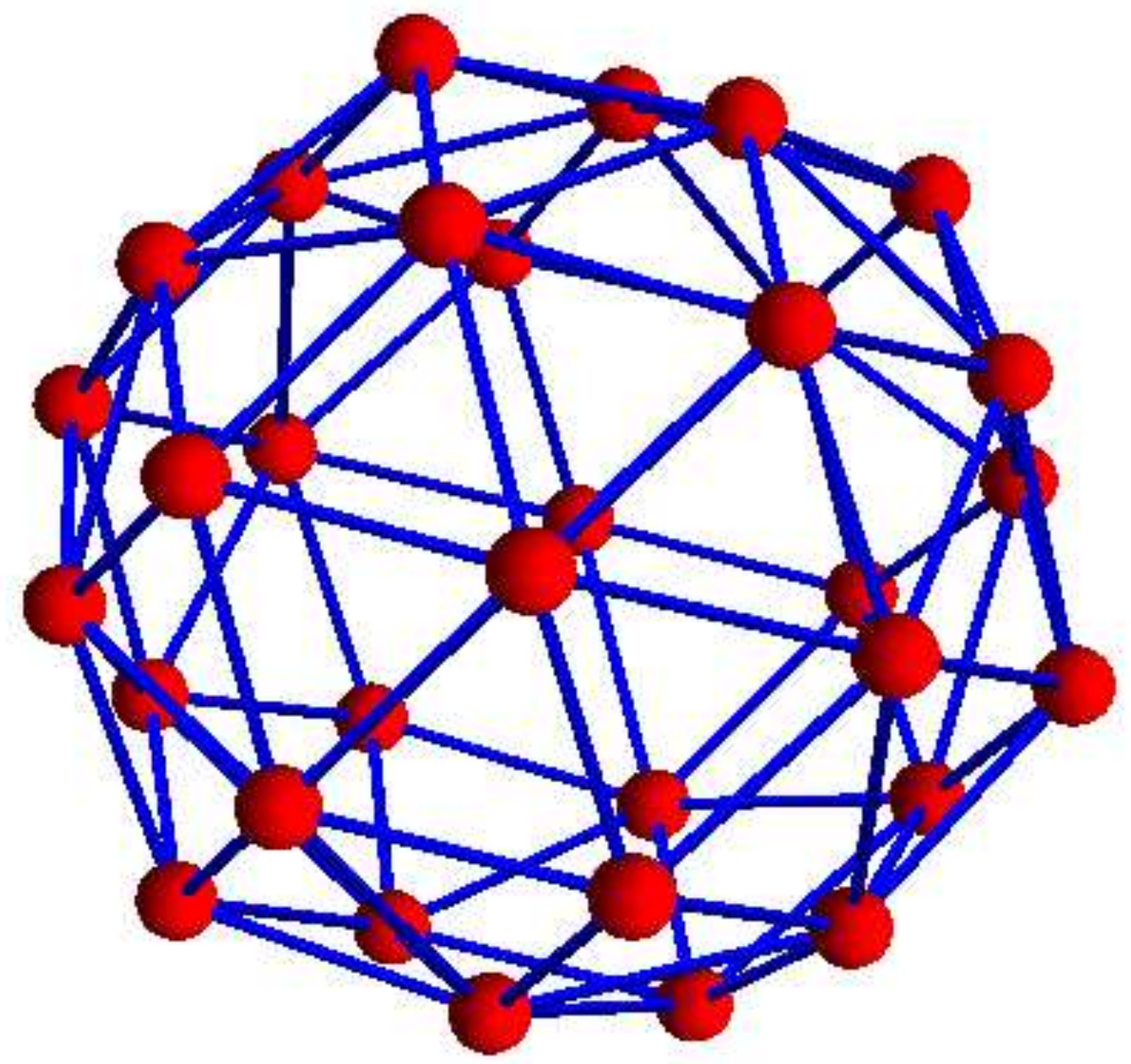}}  \\
drilled cube  &  2  & \scalebox{0.08}{\includegraphics{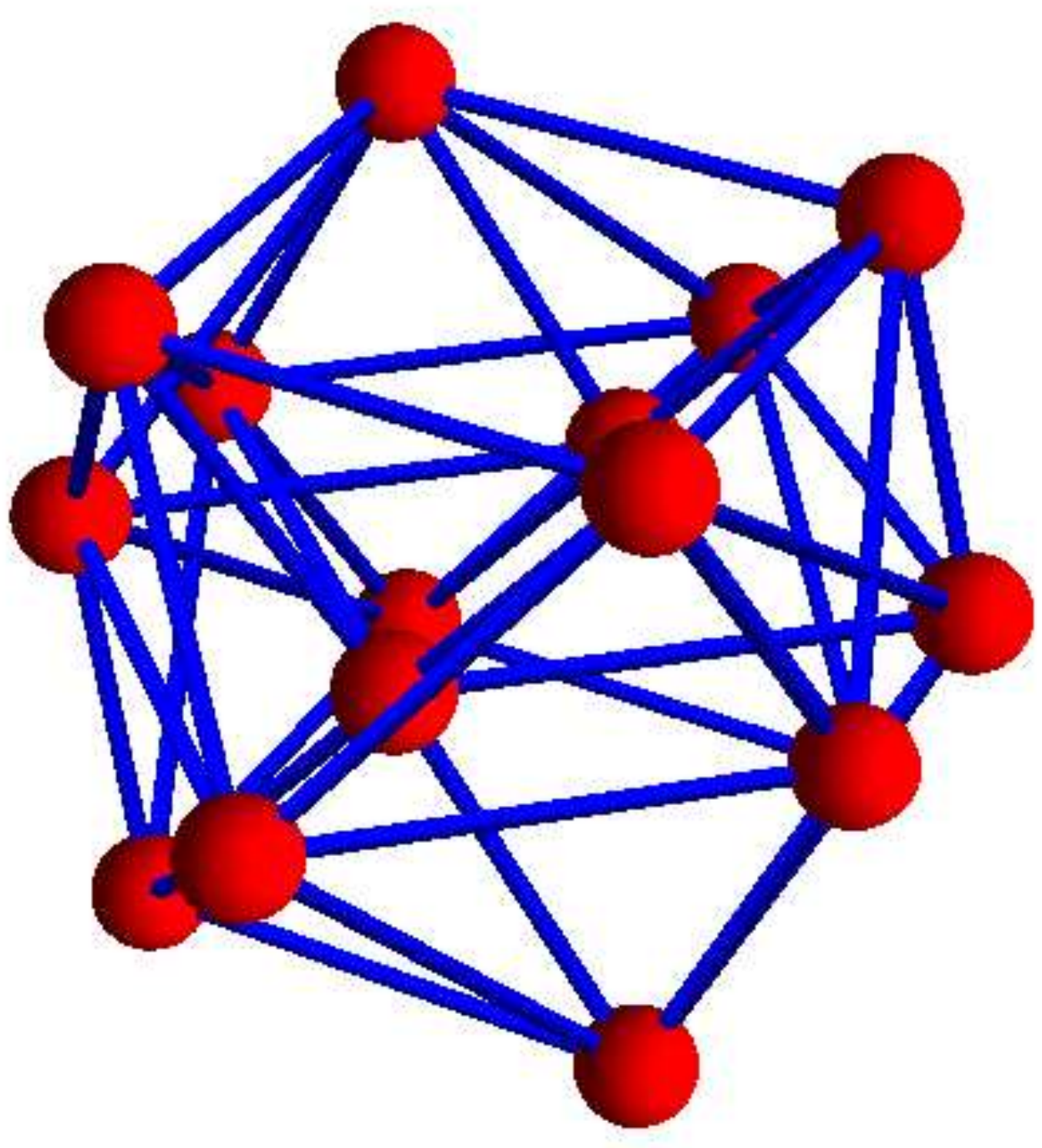}}  \\
drilled stellated cube  &  2  & \scalebox{0.08}{\includegraphics{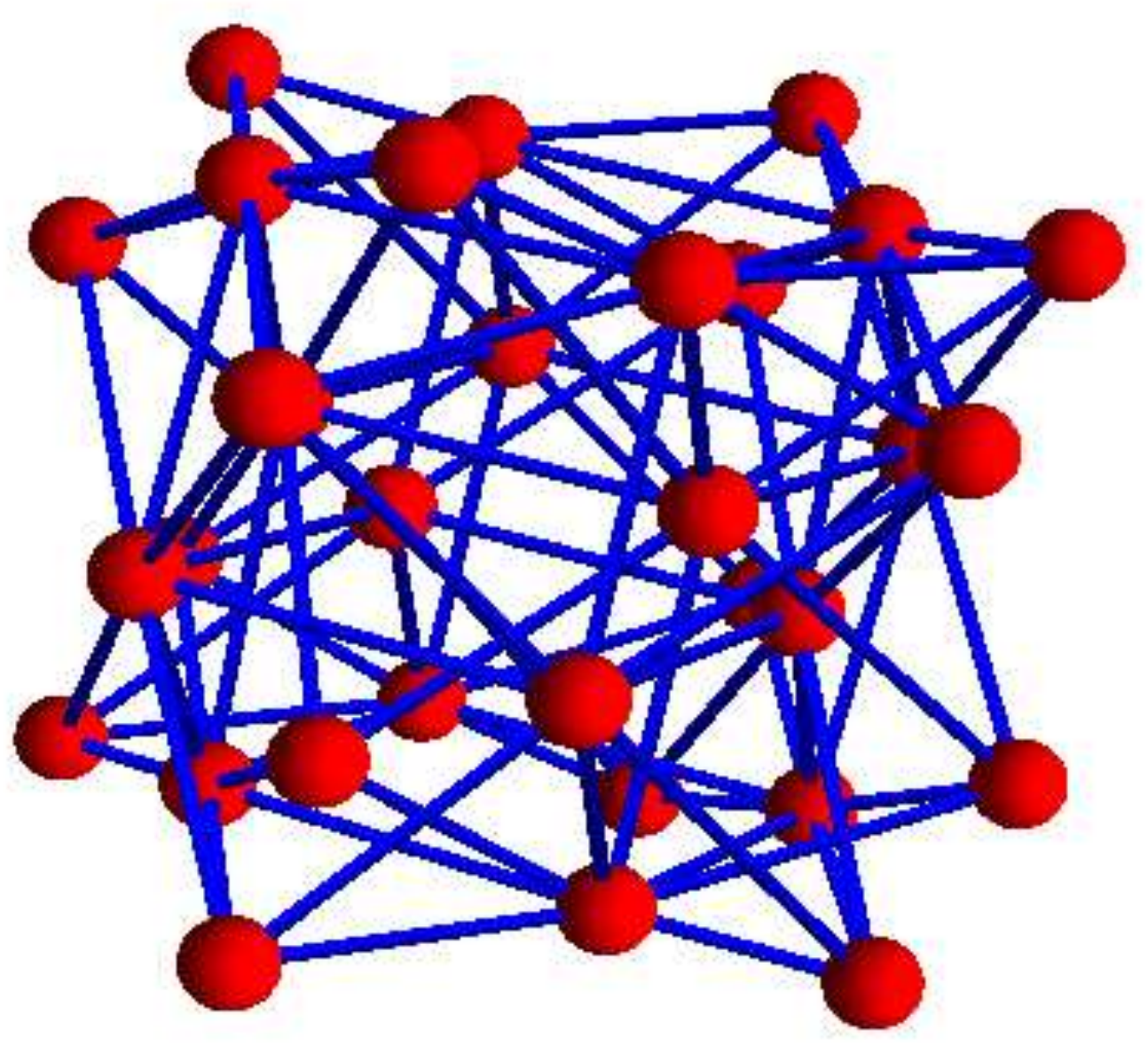}}  \\
stellated icosahedron &  2  & \scalebox{0.08}{\includegraphics{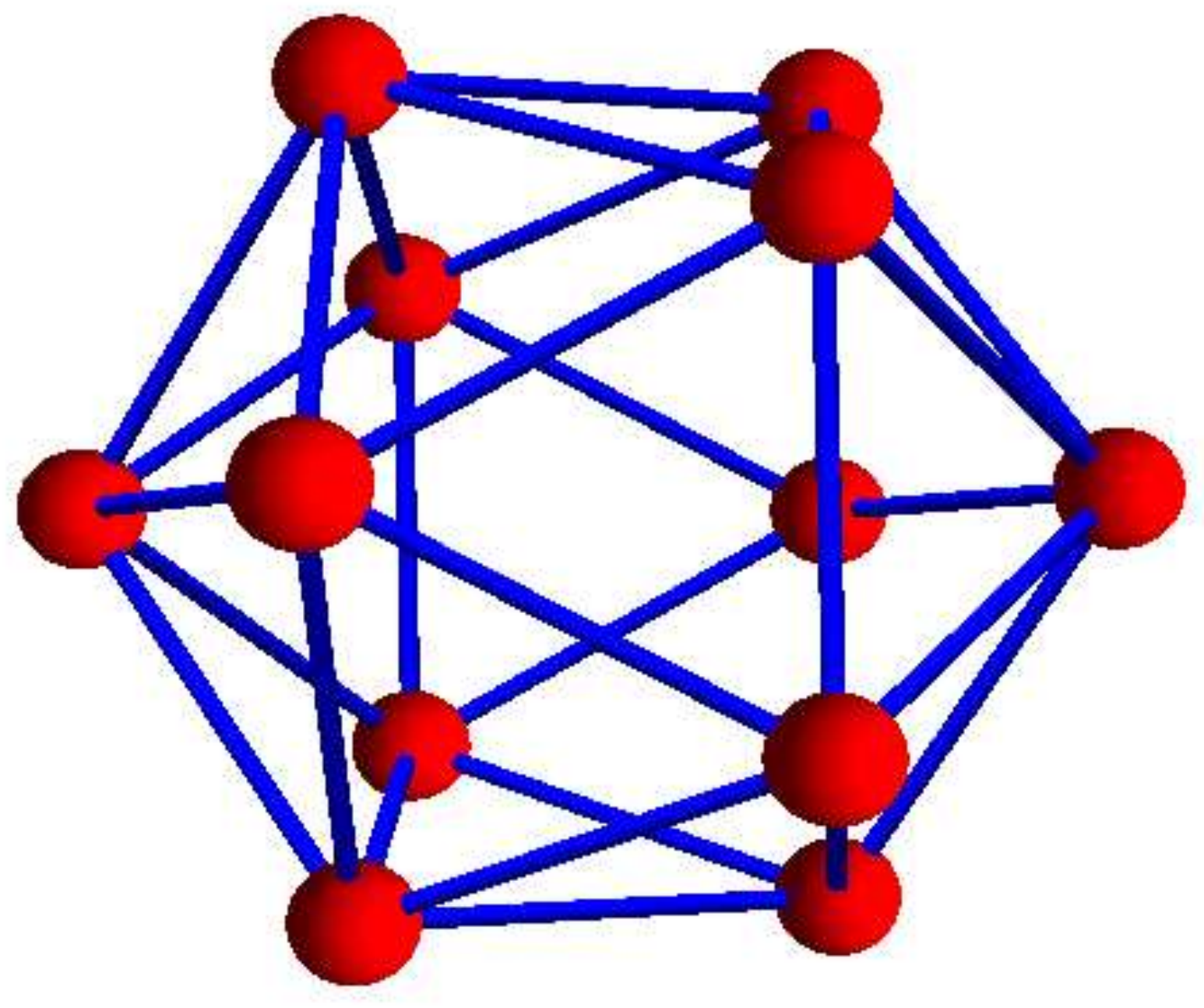}}  \\
\end{tabular} 
}}
\caption{
{\bf Example 2.1.}
Platonic solids and variants. The number displayed left to the graph is the dimension.
The only platonic solids which are $2$-dimensional are the octahedron and icosahedron. }
\end{figure}

The dimension of a vertex $p$ is the average dimension of its unit sphere $S_1(p)$ plus 1. 
The dimension of the graph can then be defined as the average of the dimensions of its points.
Figure~\ref{archimedean} shows the dimensions for the 13 Archimedean solids and their 
duals, the Catalan solids. \\

\begin{figure}
\parbox{14.8cm}{
\parbox{7.2cm}{
\begin{tabular}{ccc}
\scalebox{0.08}{\includegraphics{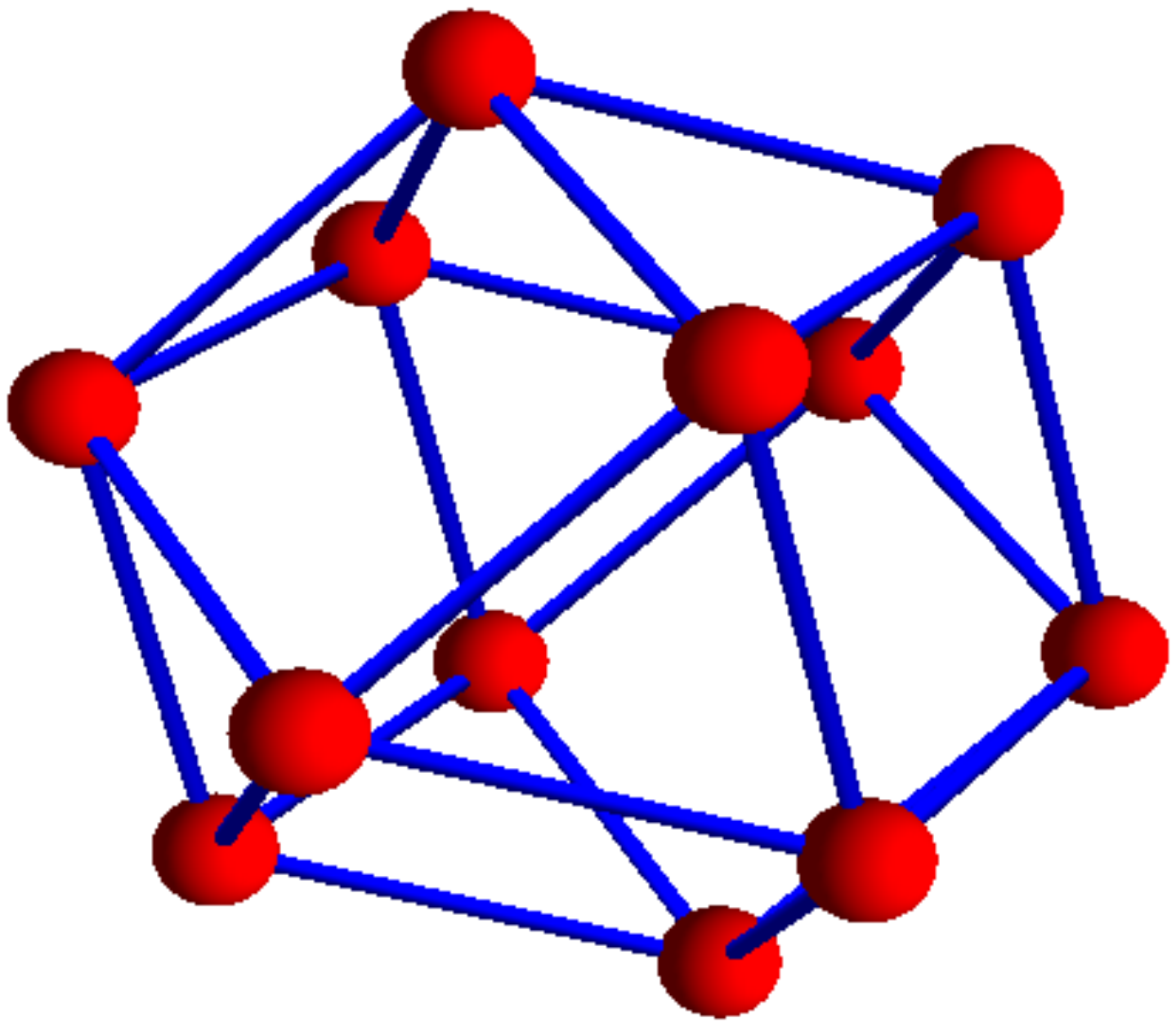}}  &
\scalebox{0.08}{\includegraphics{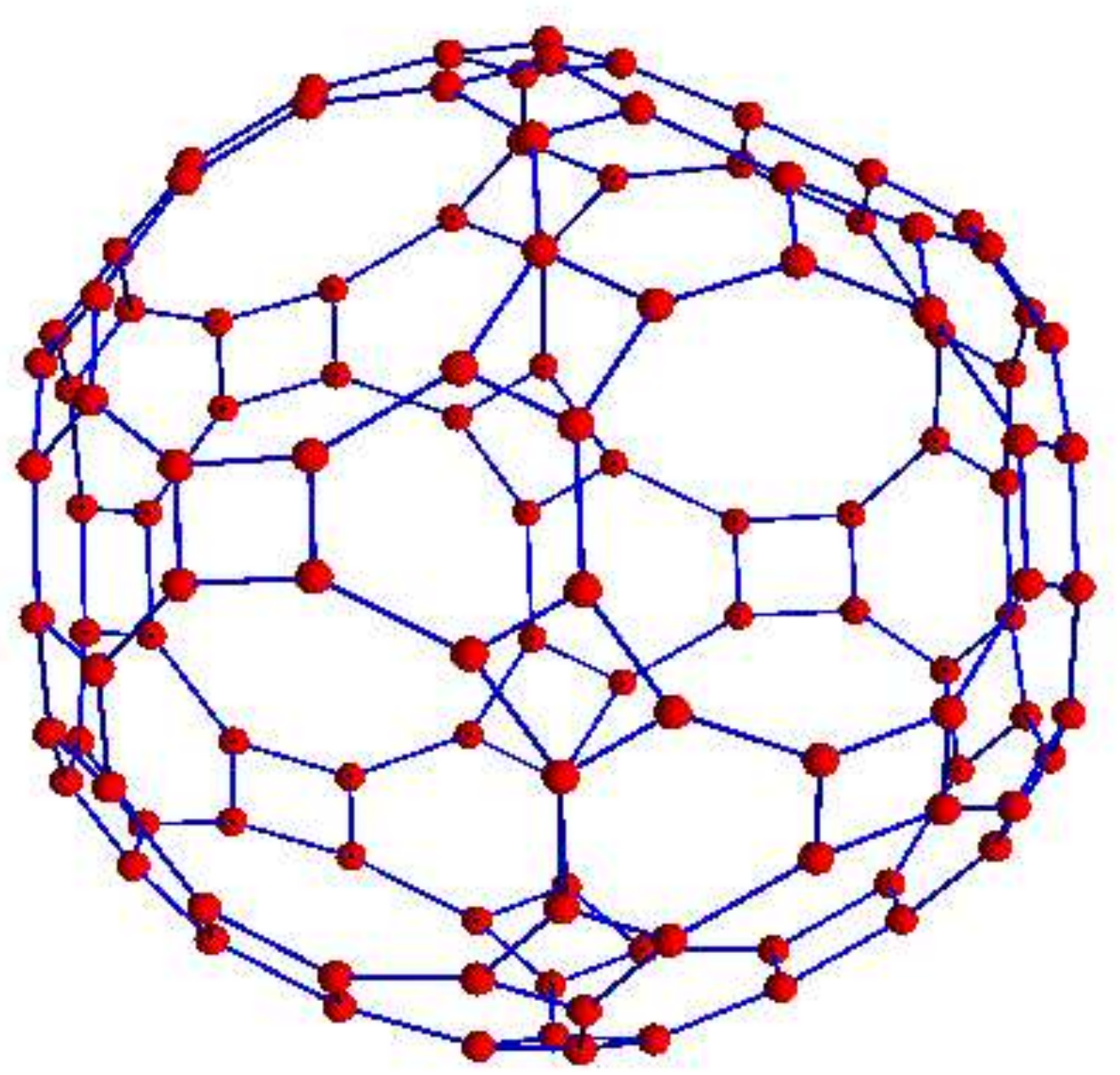}}  &
\scalebox{0.08}{\includegraphics{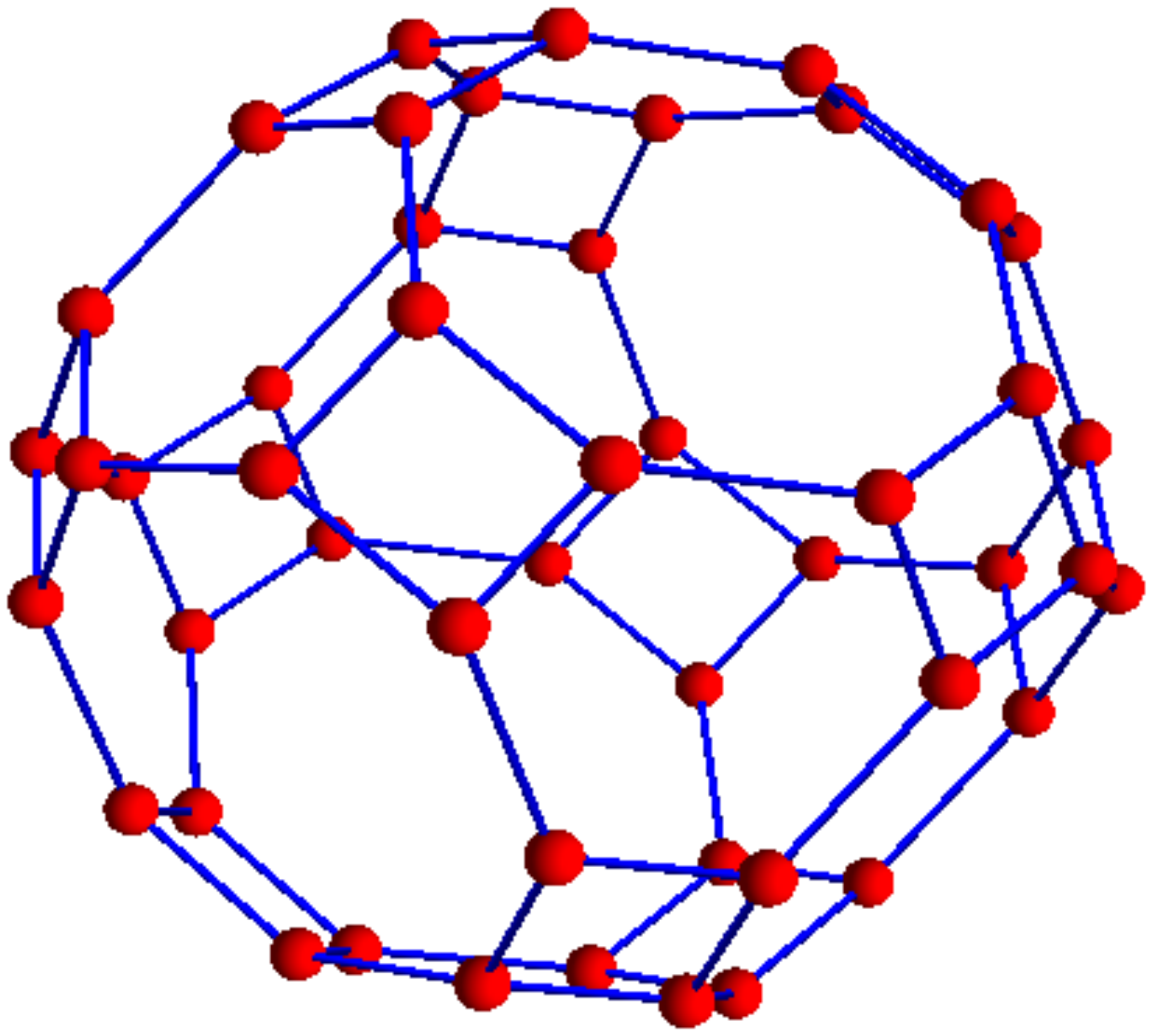}}  \\
(2,-4)                                                            &
(1,-60)                                                           &
(1,-24)                                                           \\
\scalebox{0.08}{\includegraphics{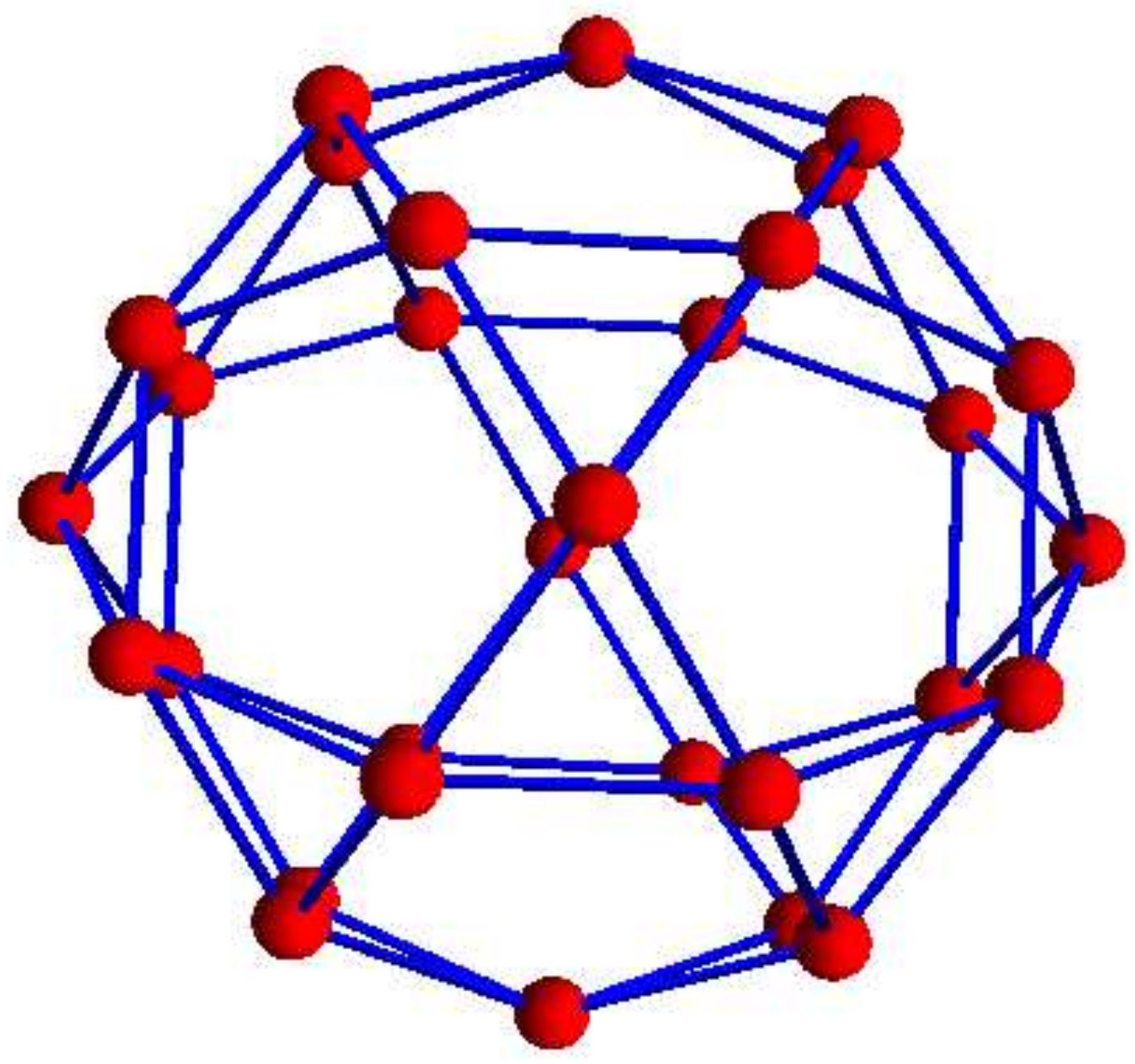}}  &
\scalebox{0.08}{\includegraphics{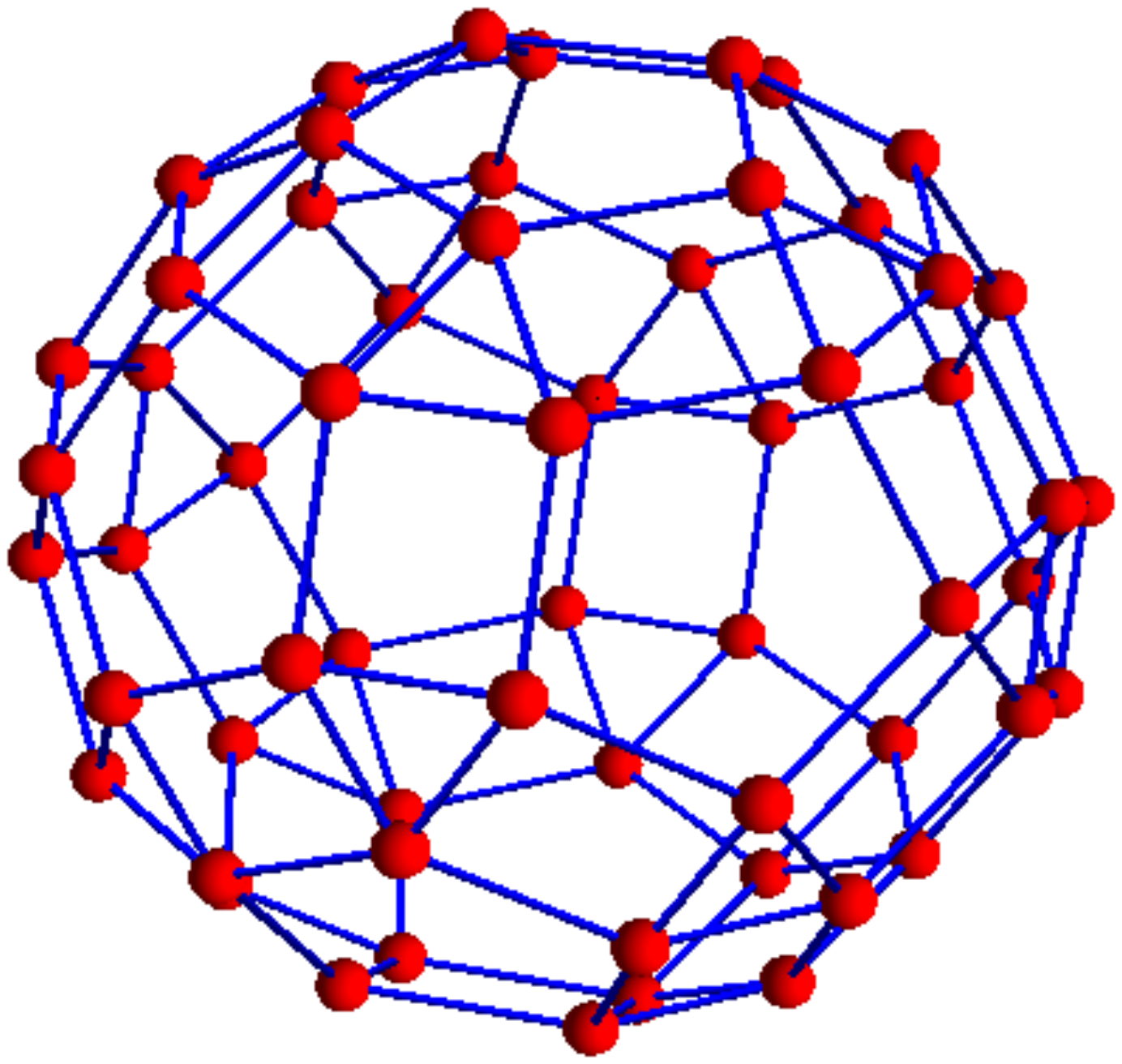}}  &
\scalebox{0.08}{\includegraphics{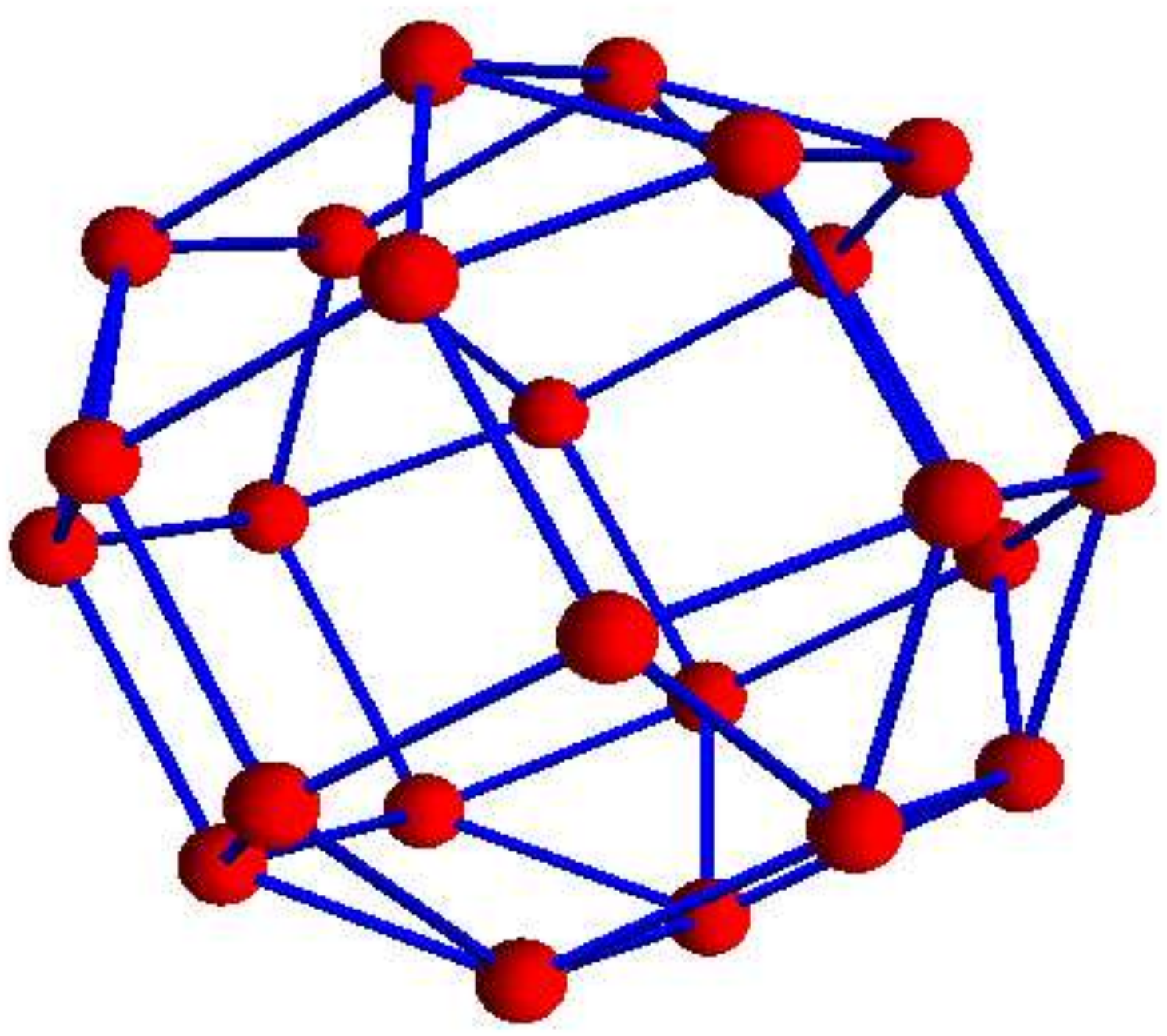}}  \\
(2,-10)                                                           &
(3/2,-40)                                                         &
(3/2,-16)                                                        \\
\scalebox{0.08}{\includegraphics{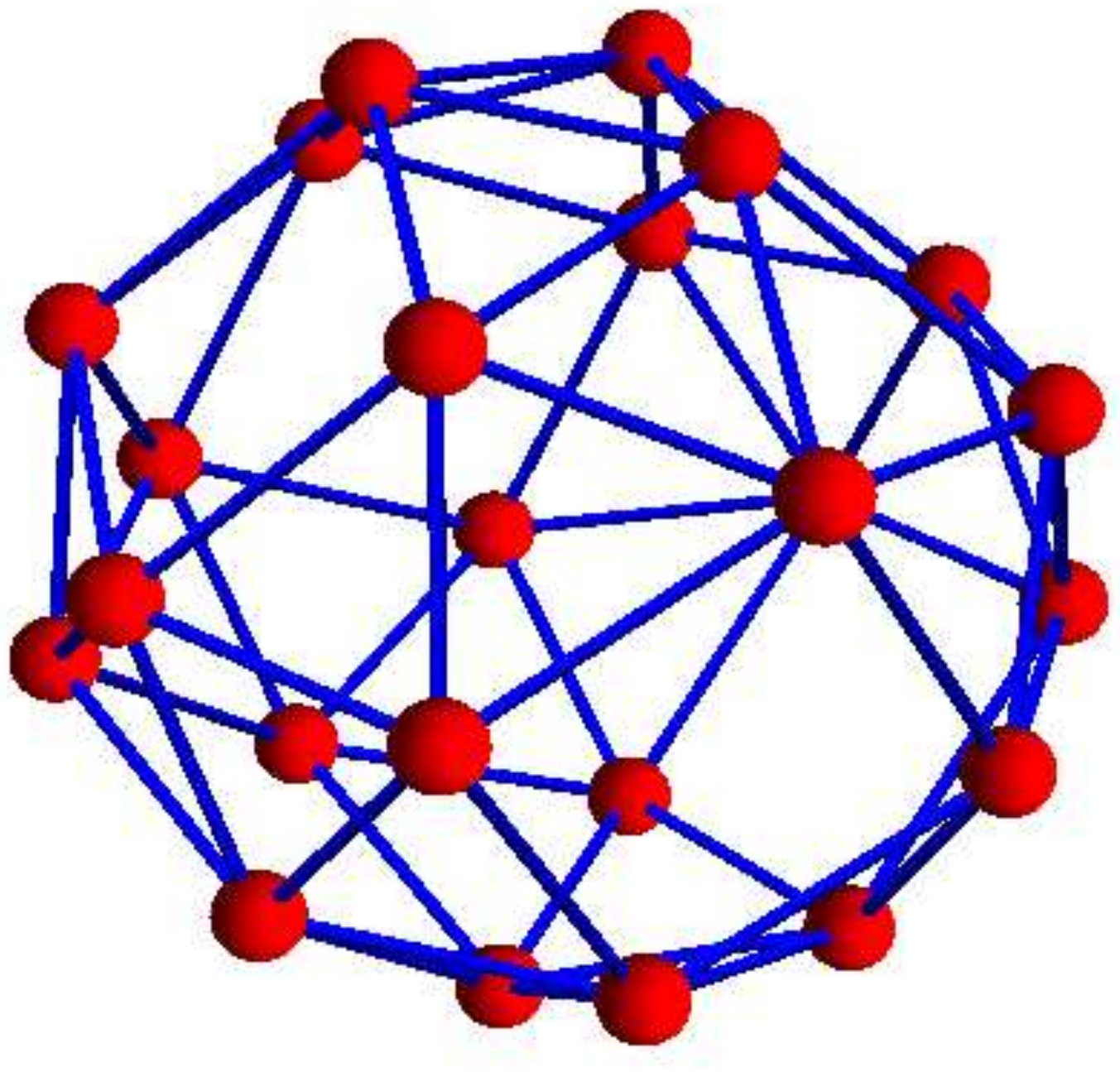}}  &
\scalebox{0.08}{\includegraphics{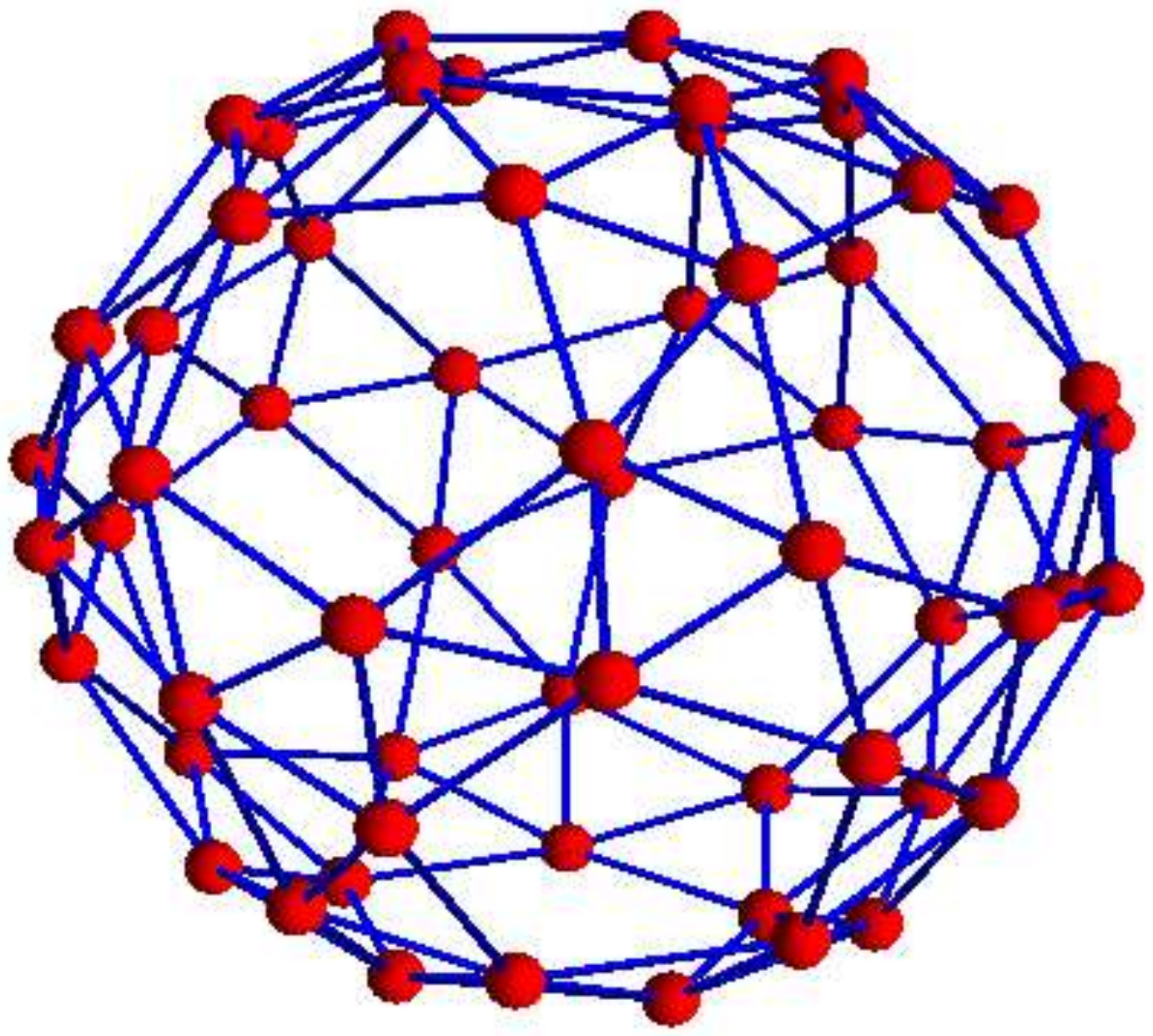}}  &
\scalebox{0.08}{\includegraphics{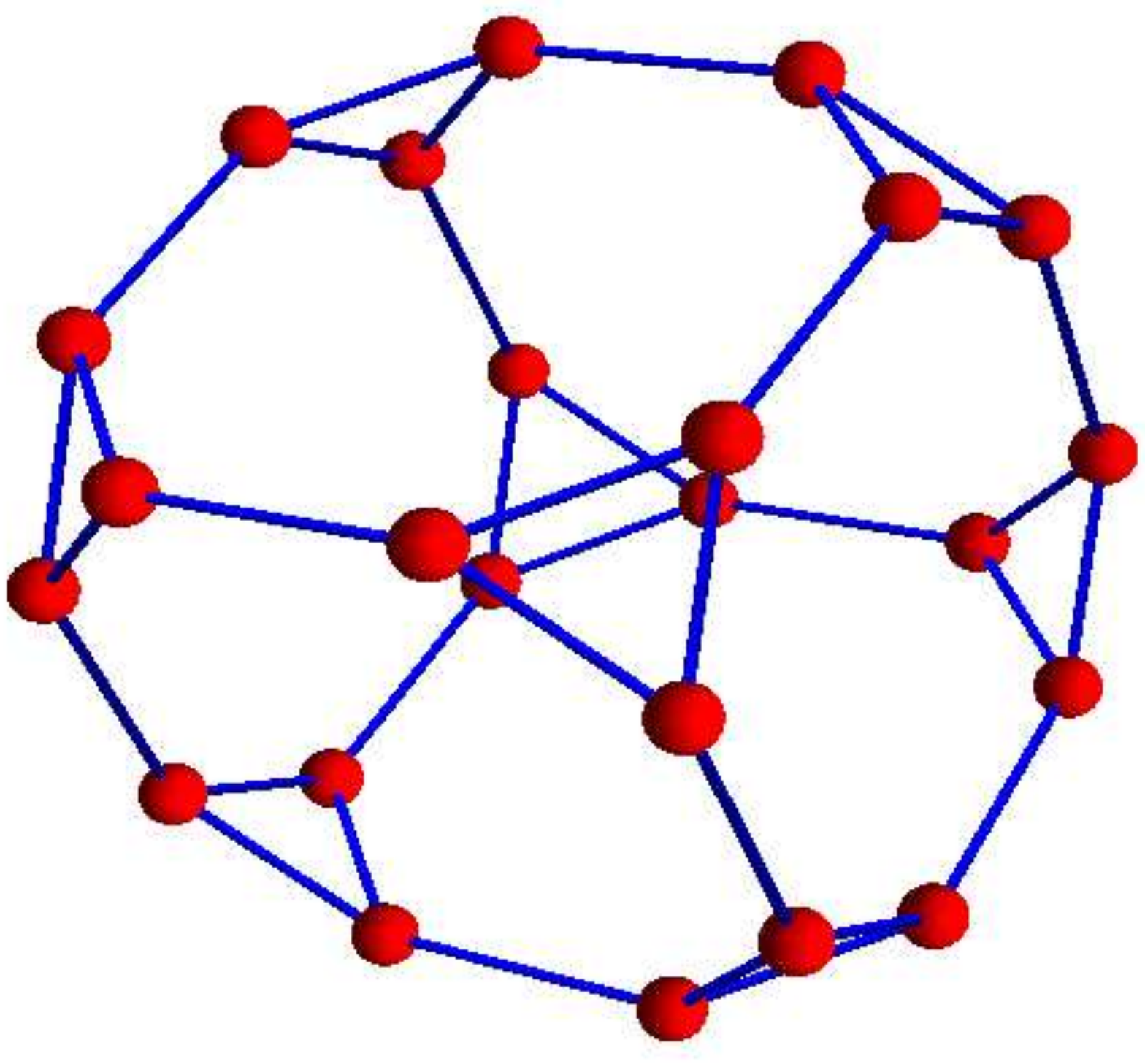}}  \\
(2,-4)                                                            &
(2,-10)                                                           &
(5/3,-4)                                                          \\
\scalebox{0.08}{\includegraphics{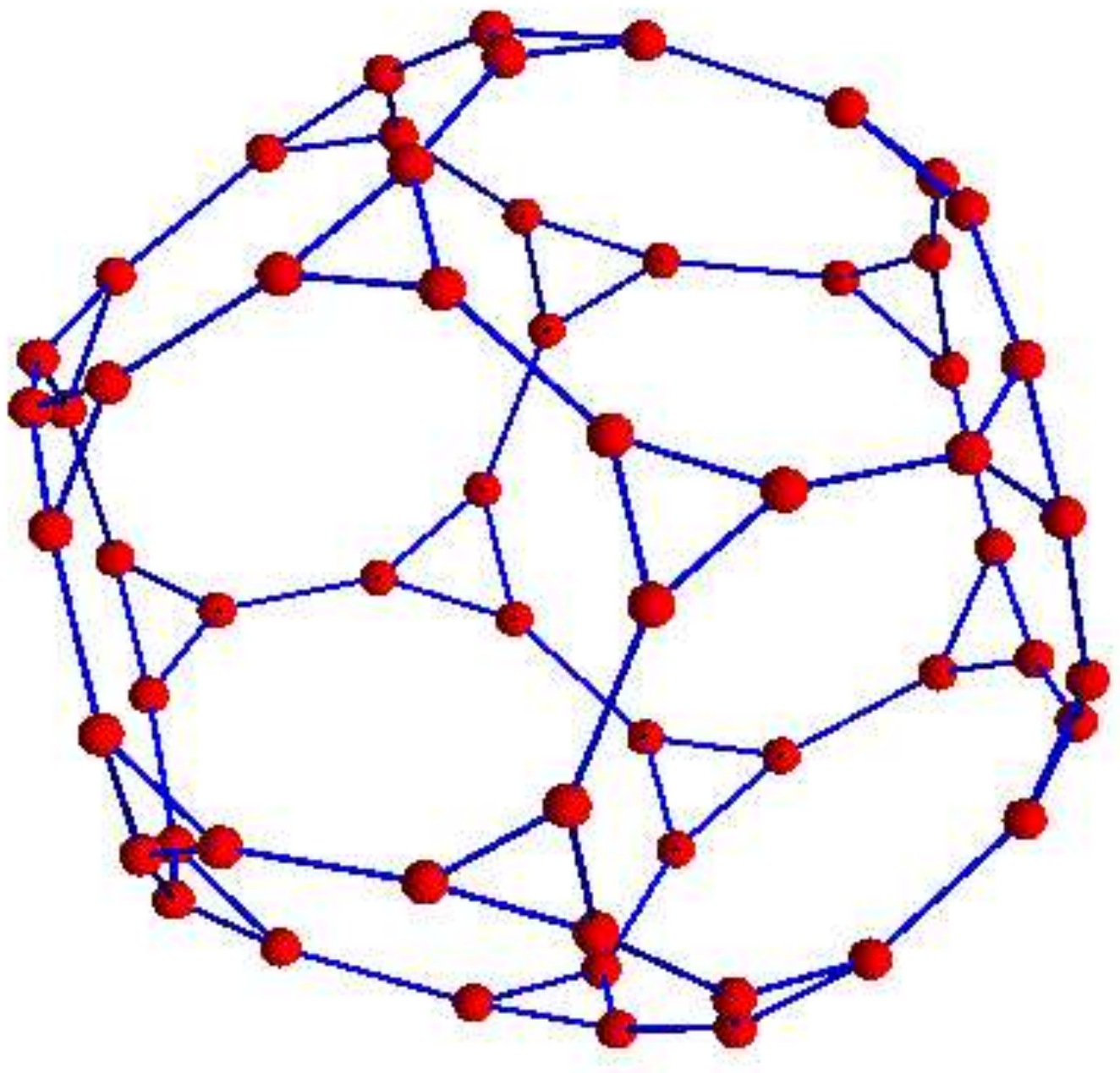}}  &
\scalebox{0.08}{\includegraphics{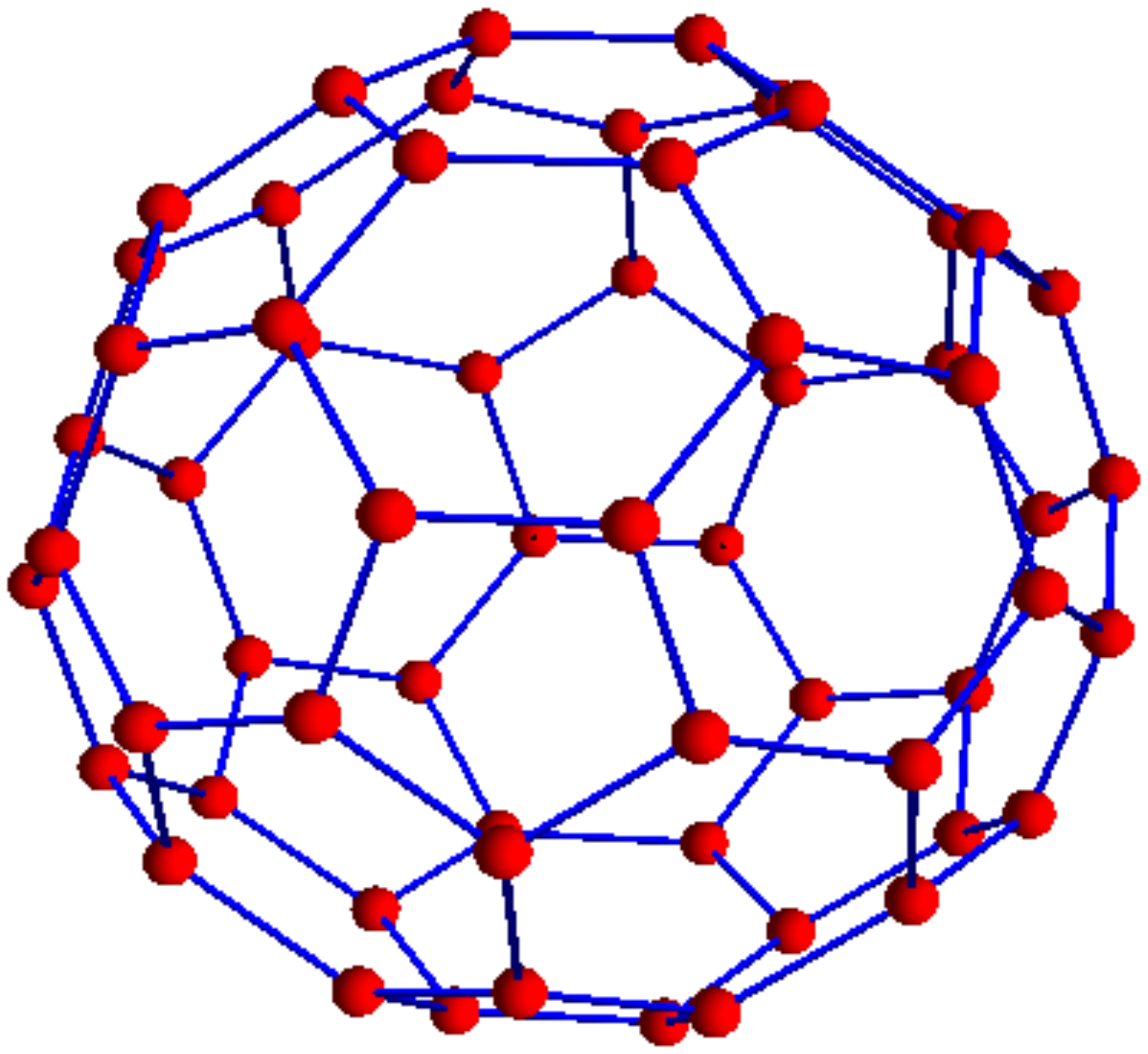}}  &
\scalebox{0.08}{\includegraphics{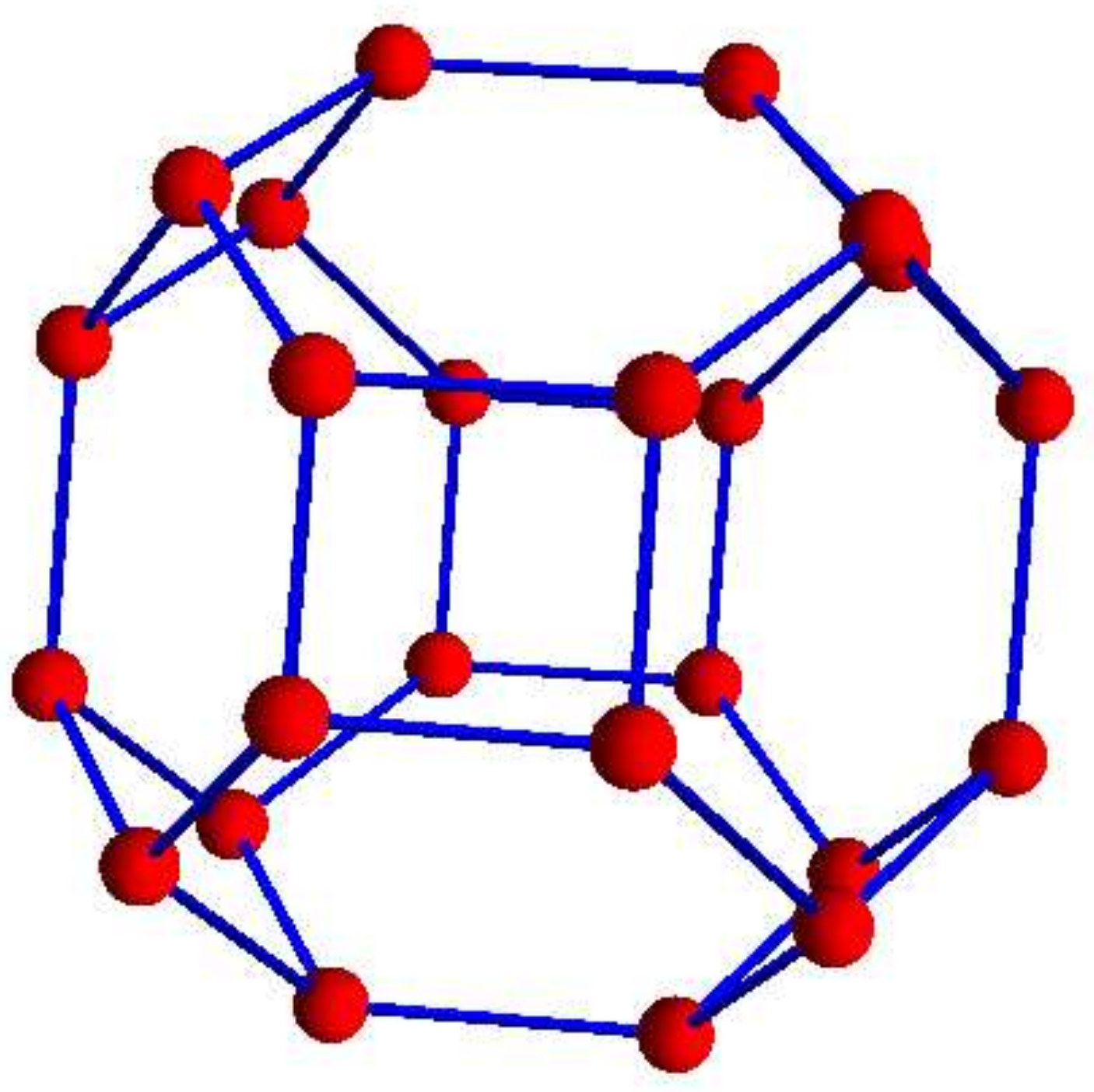}}  \\
(5/3,-10)                                                         &
(1,-30)                                                           &
(1,-12)                                                           \\
\scalebox{0.08}{\includegraphics{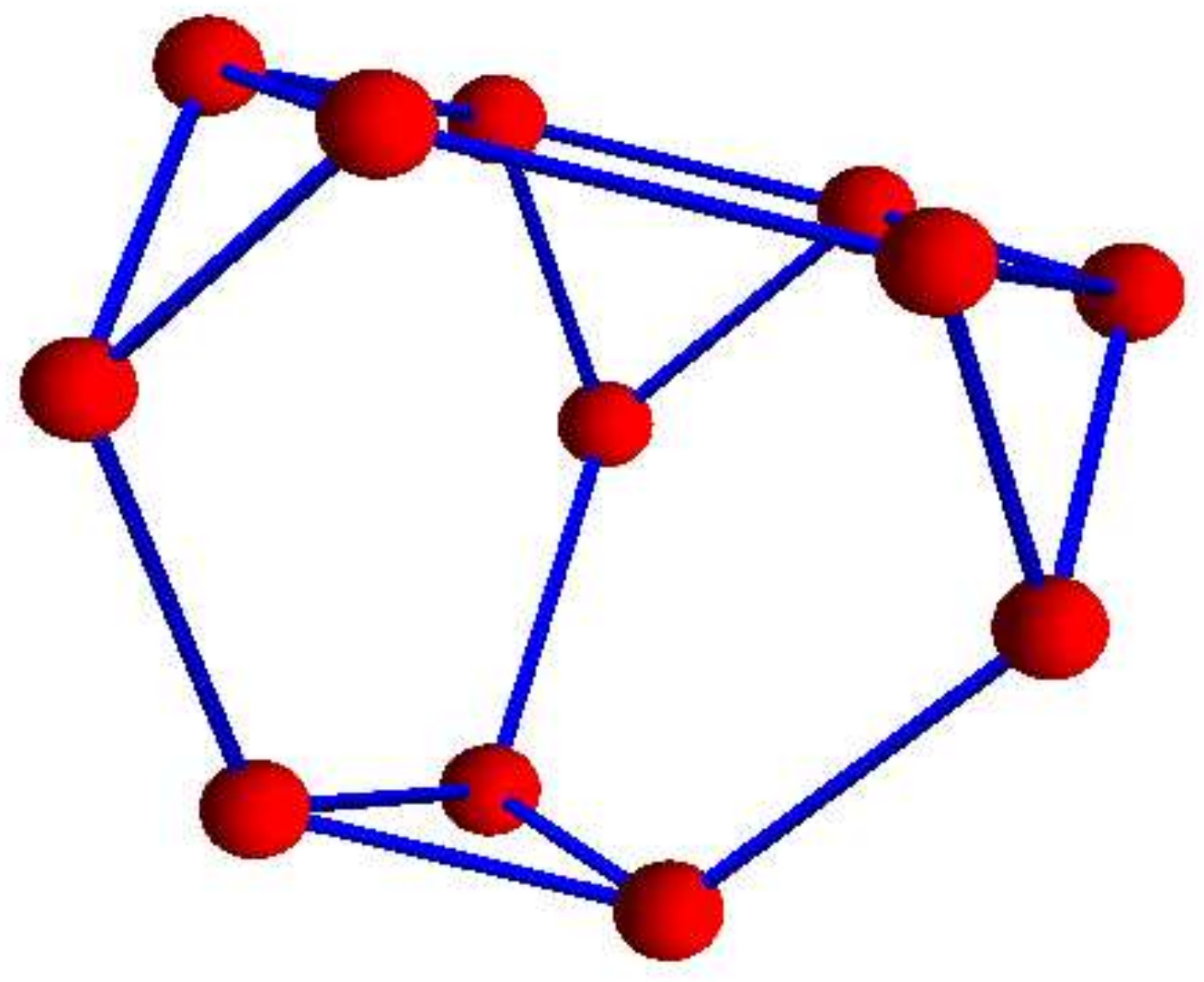}}  & &  \\
(5/3,-2)                                                         & &  
\end{tabular}
}
\parbox{7.2cm}{
\begin{tabular}{ccc}
\scalebox{0.08}{\includegraphics{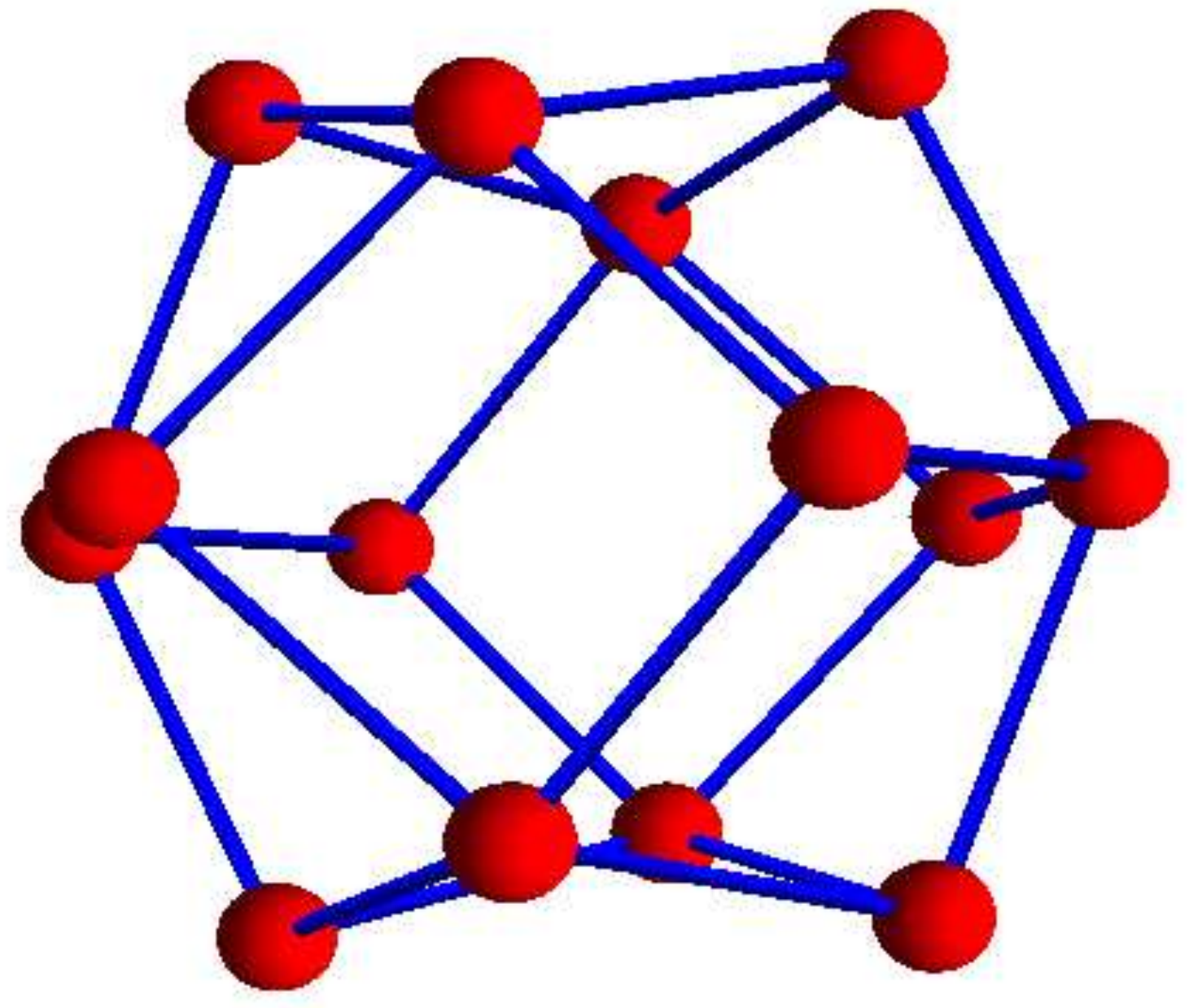}}  &
\scalebox{0.08}{\includegraphics{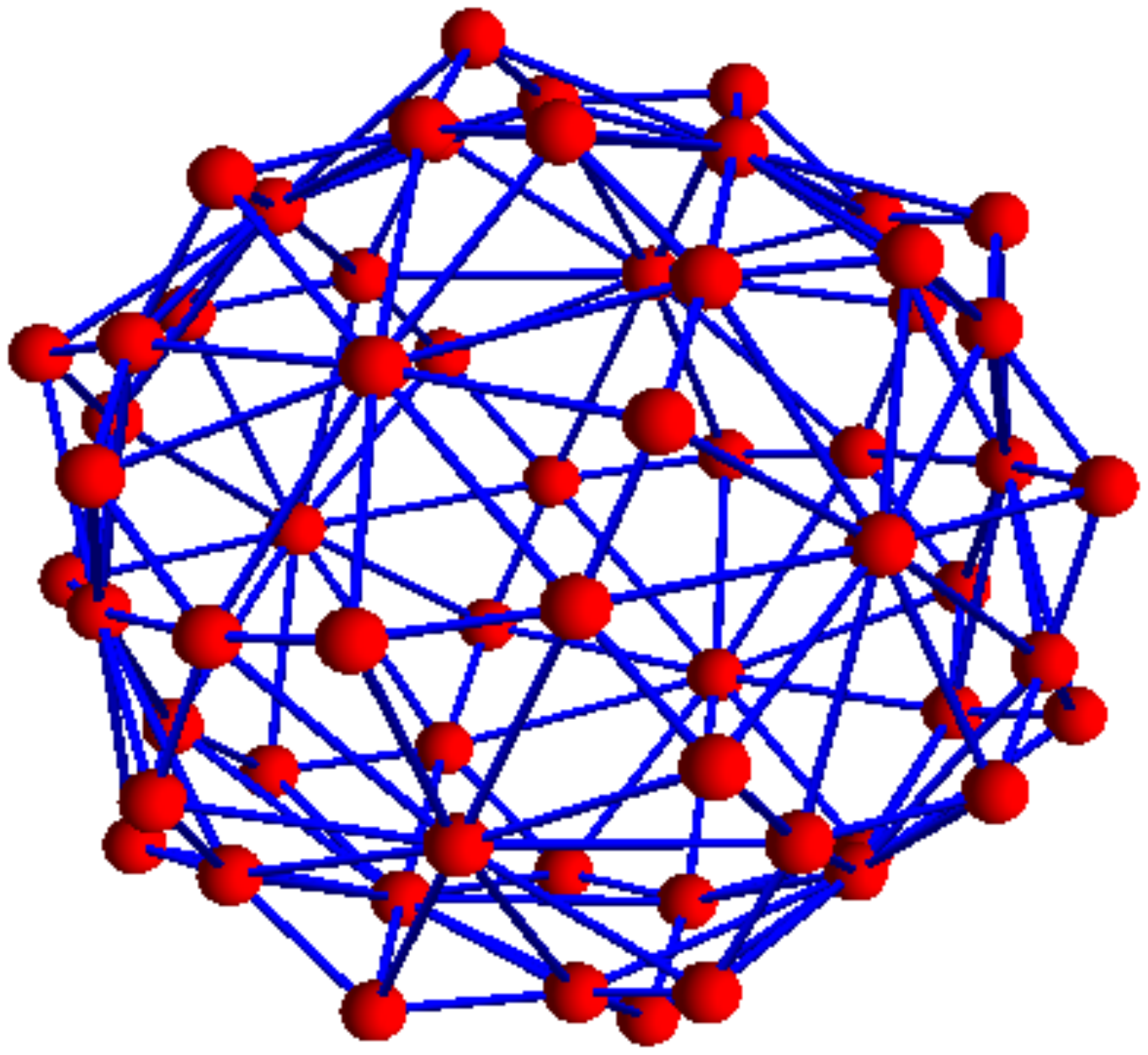}}  &
\scalebox{0.08}{\includegraphics{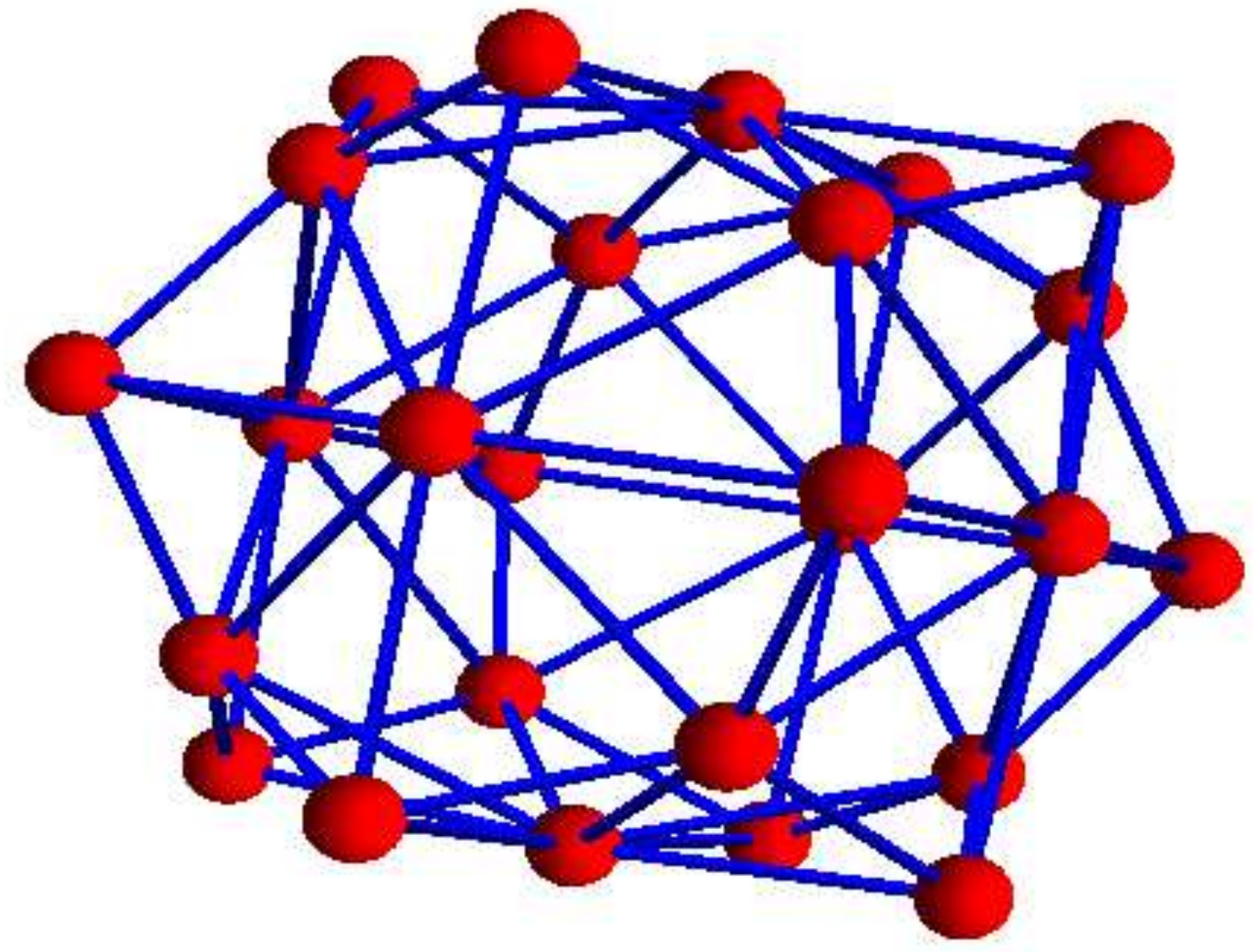}} \\
(1,-10)                                                       &
(2,2)                                                         &
(2,2)                                                         \\
\scalebox{0.08}{\includegraphics{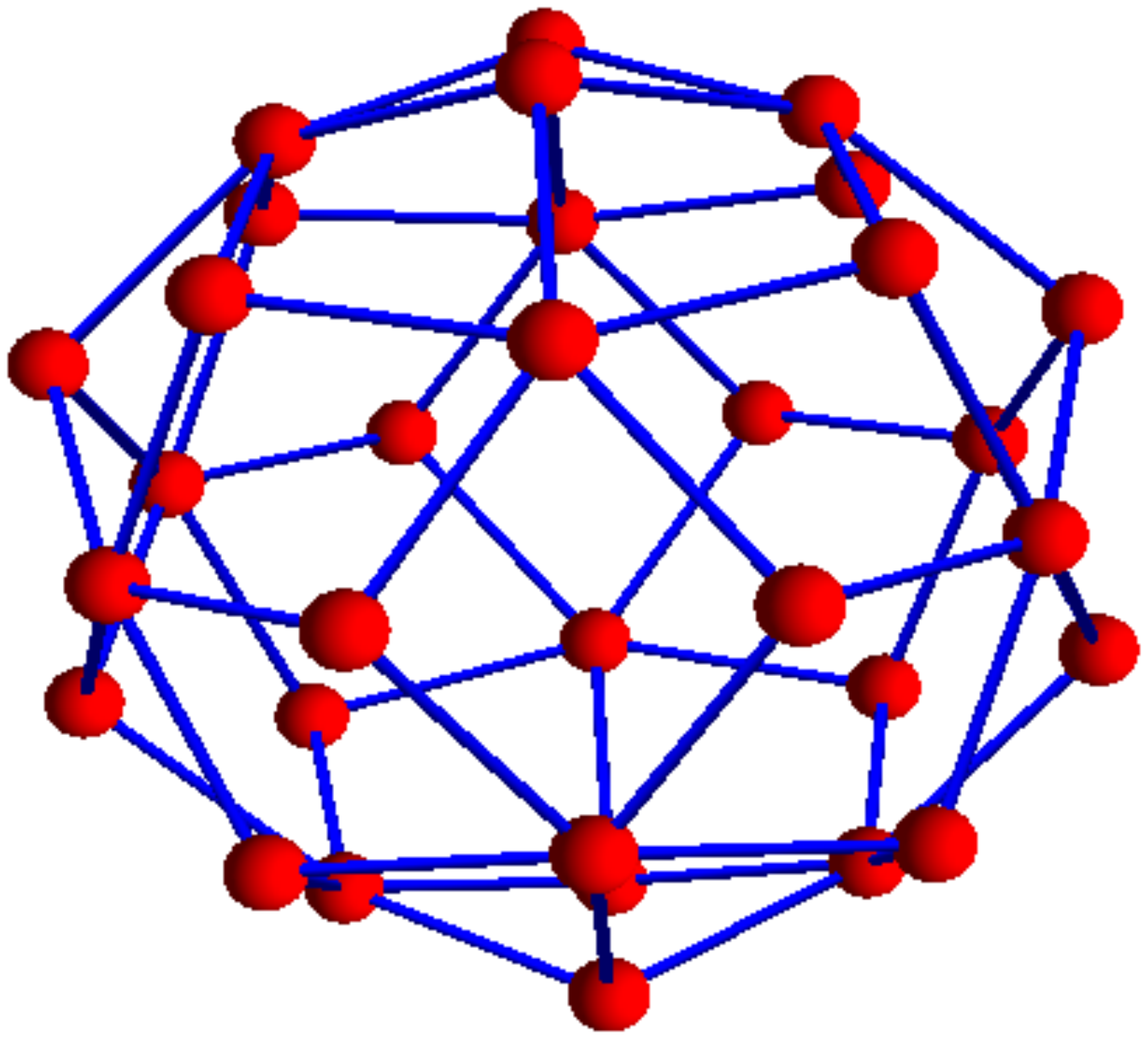}}  &
\scalebox{0.08}{\includegraphics{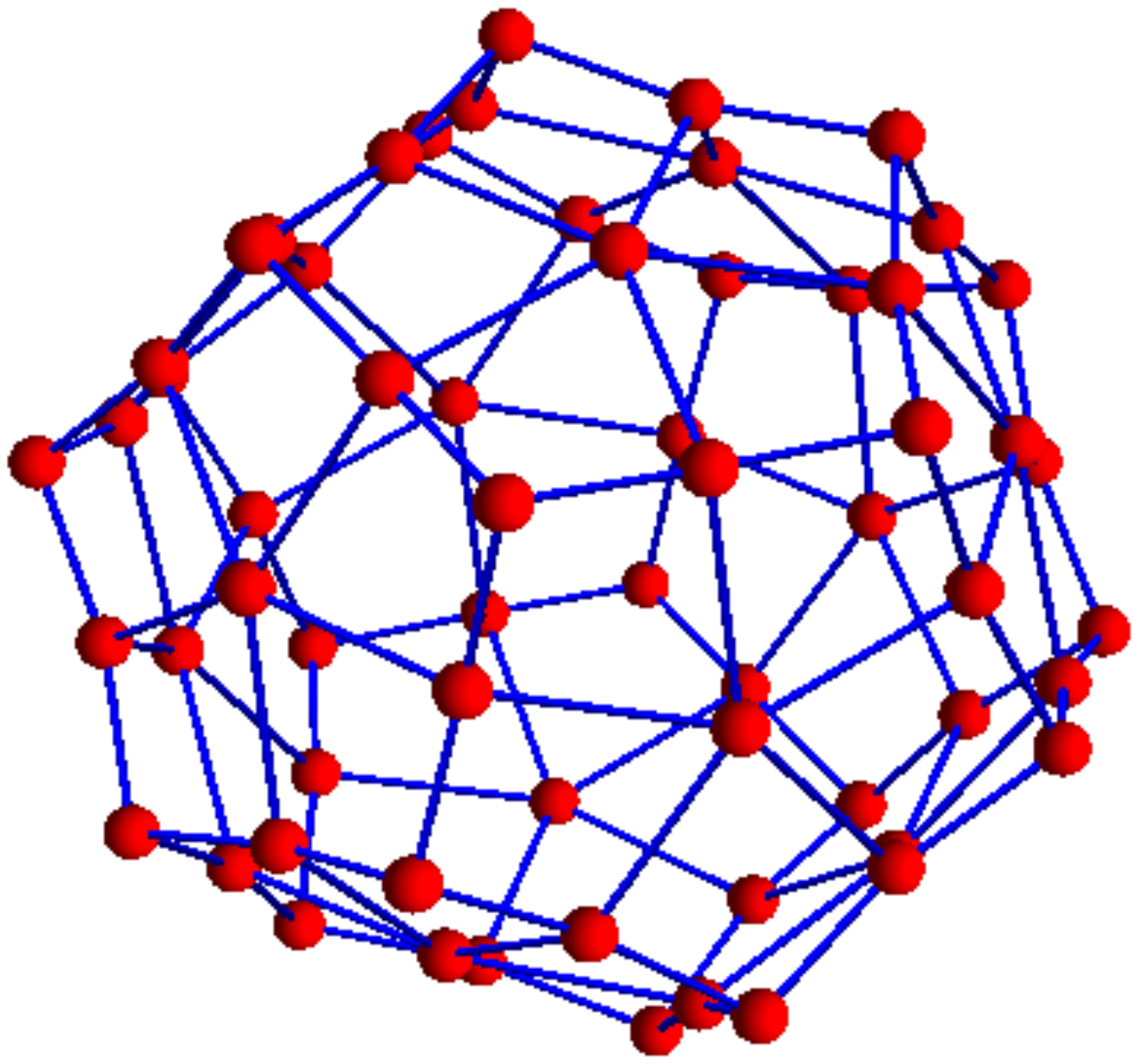}}  &
\scalebox{0.08}{\includegraphics{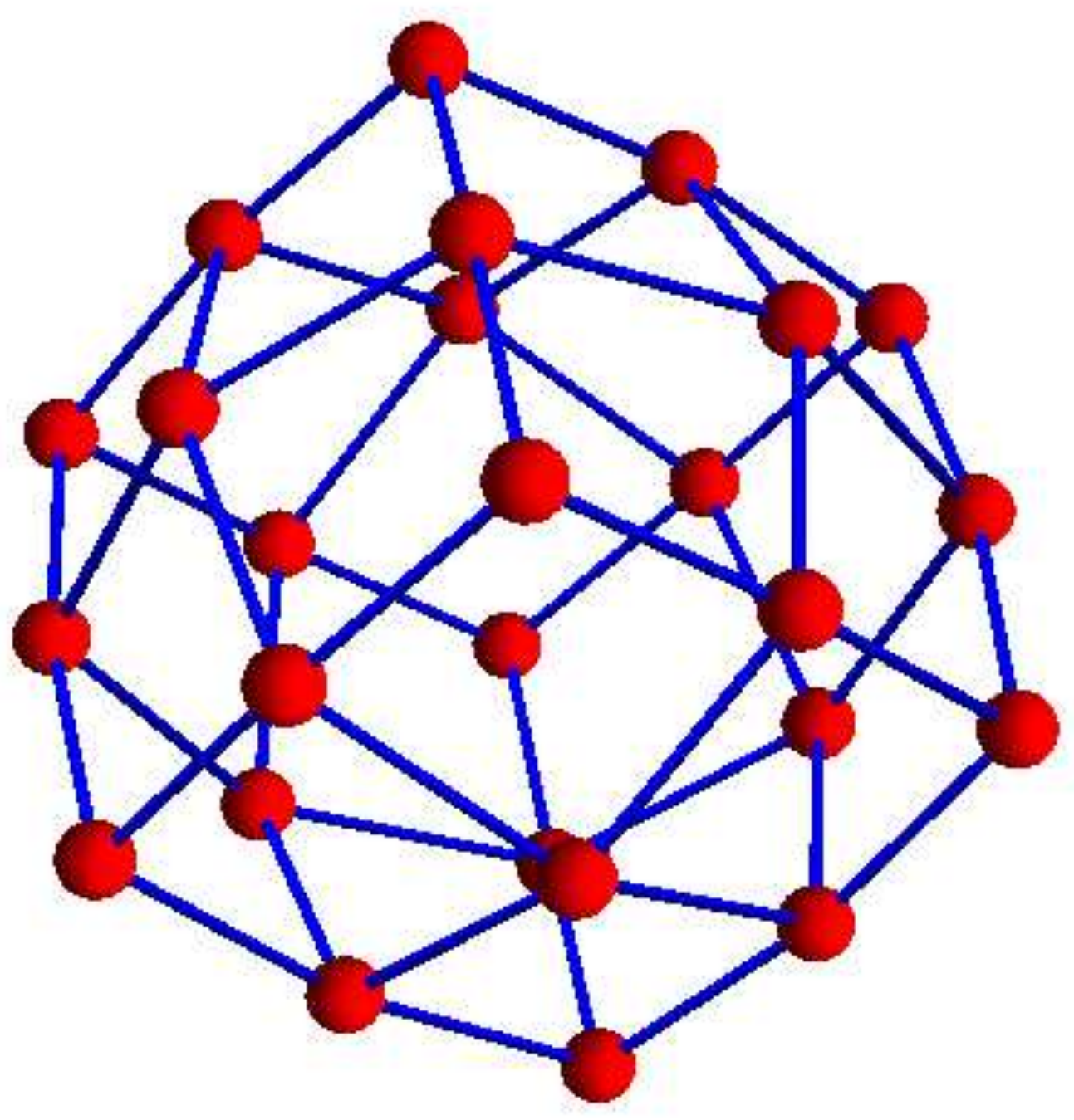}} \\
(1,-28)                                                       &
(1,-58)                                                       &
(1,-22)                                                      \\
\scalebox{0.08}{\includegraphics{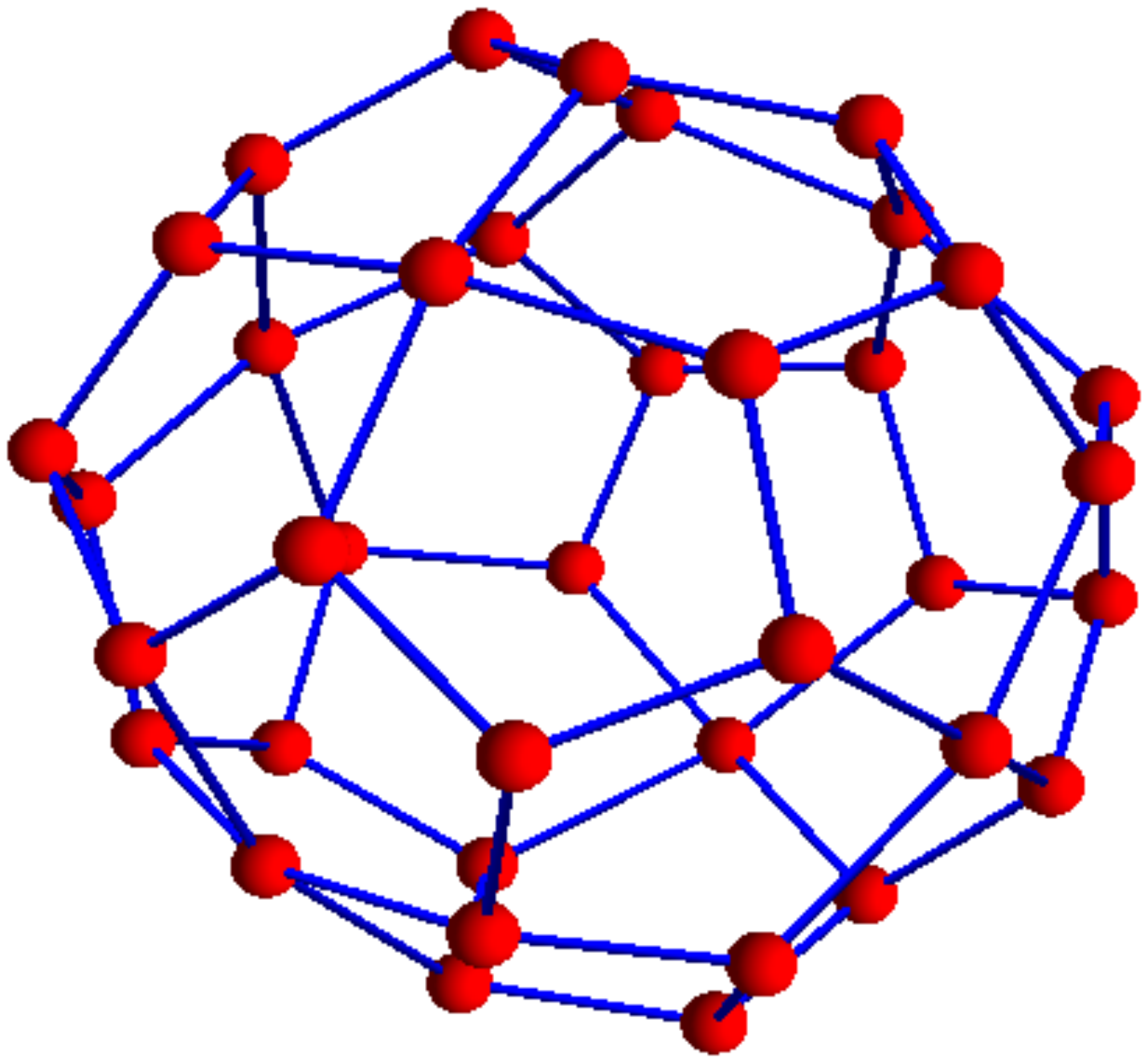}}  &
\scalebox{0.08}{\includegraphics{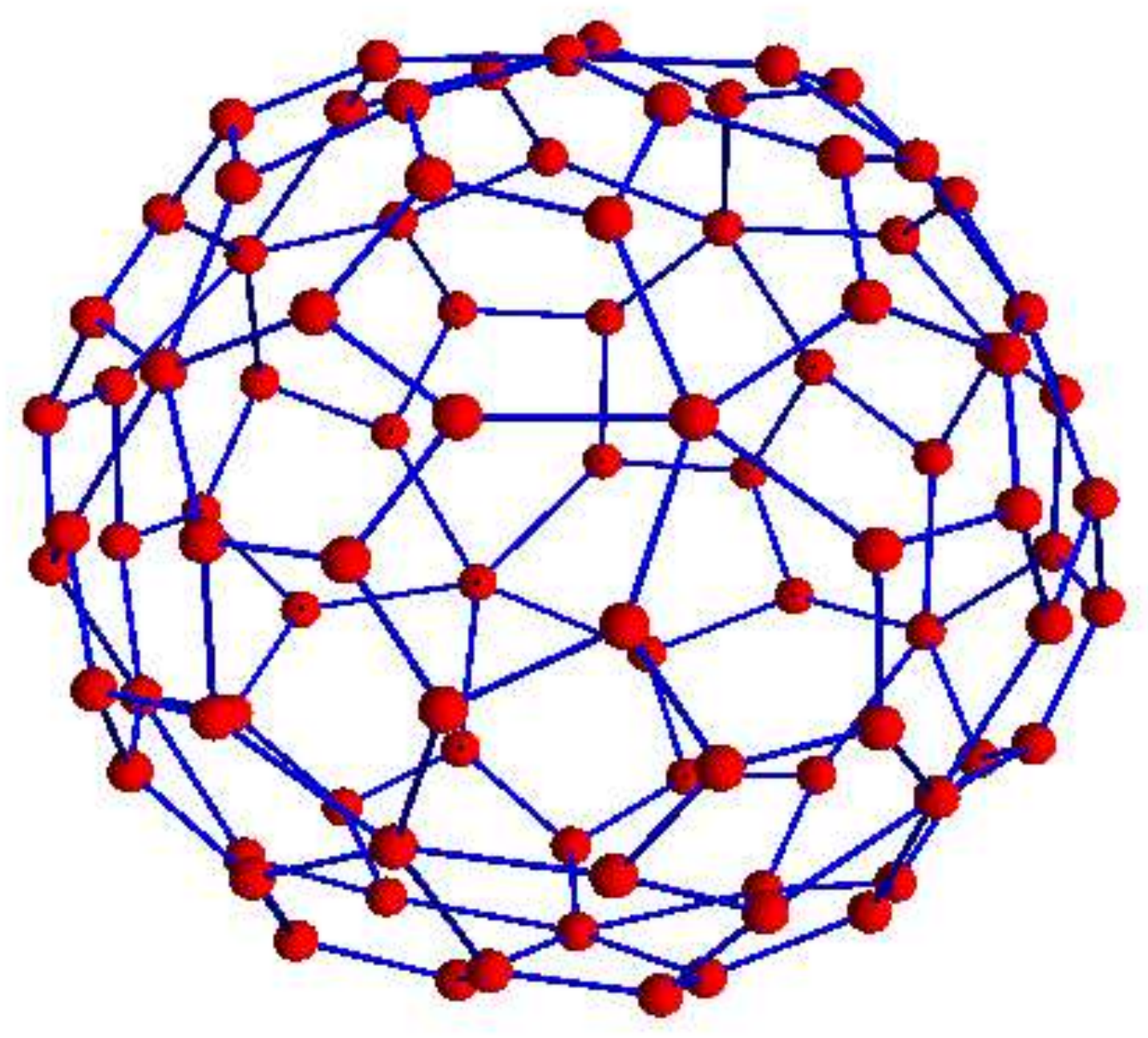}}  &
\scalebox{0.08}{\includegraphics{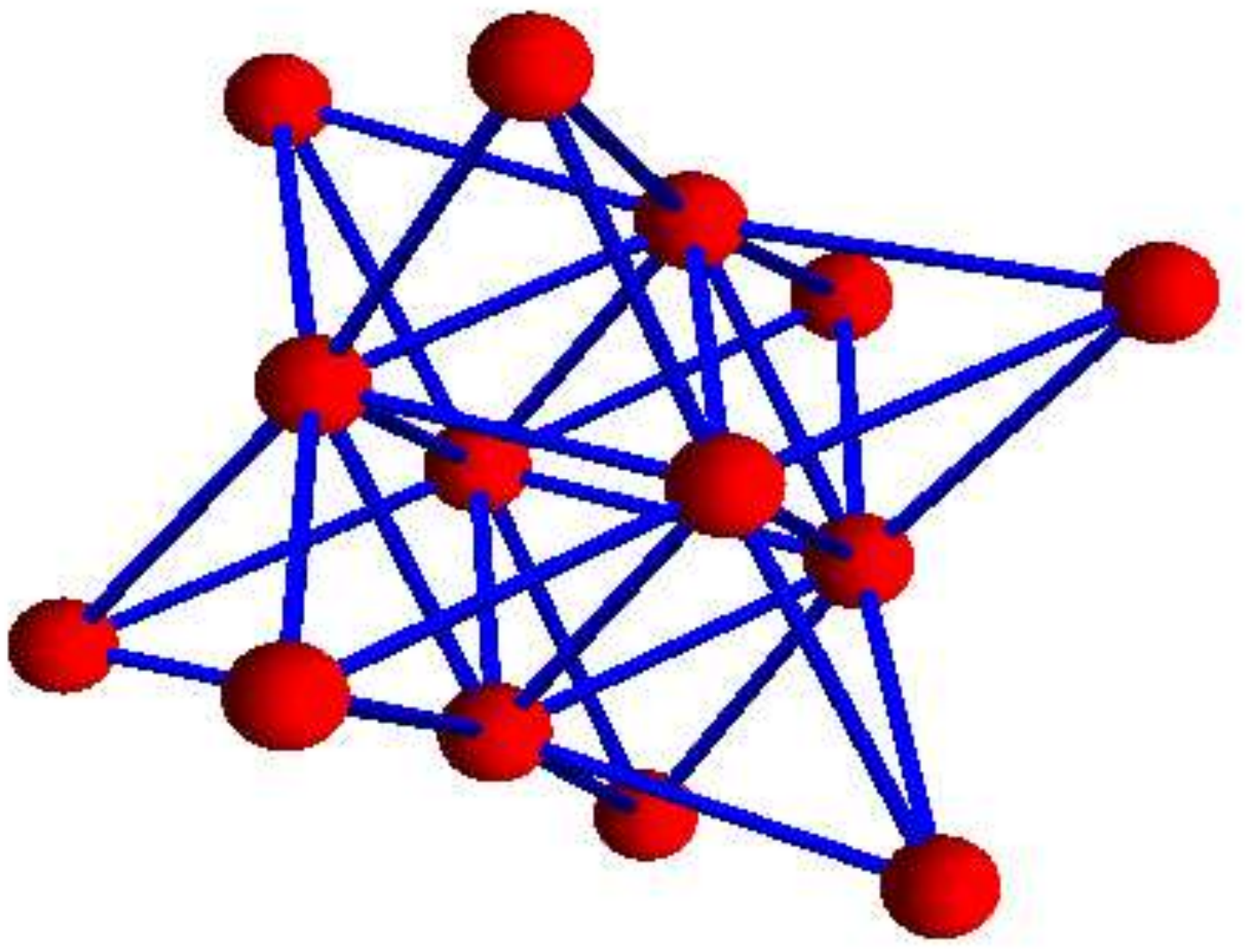}} \\
(1,-22)                                                       &
(1,-58)                                                       &
(3,2)                                                        \\
\scalebox{0.08}{\includegraphics{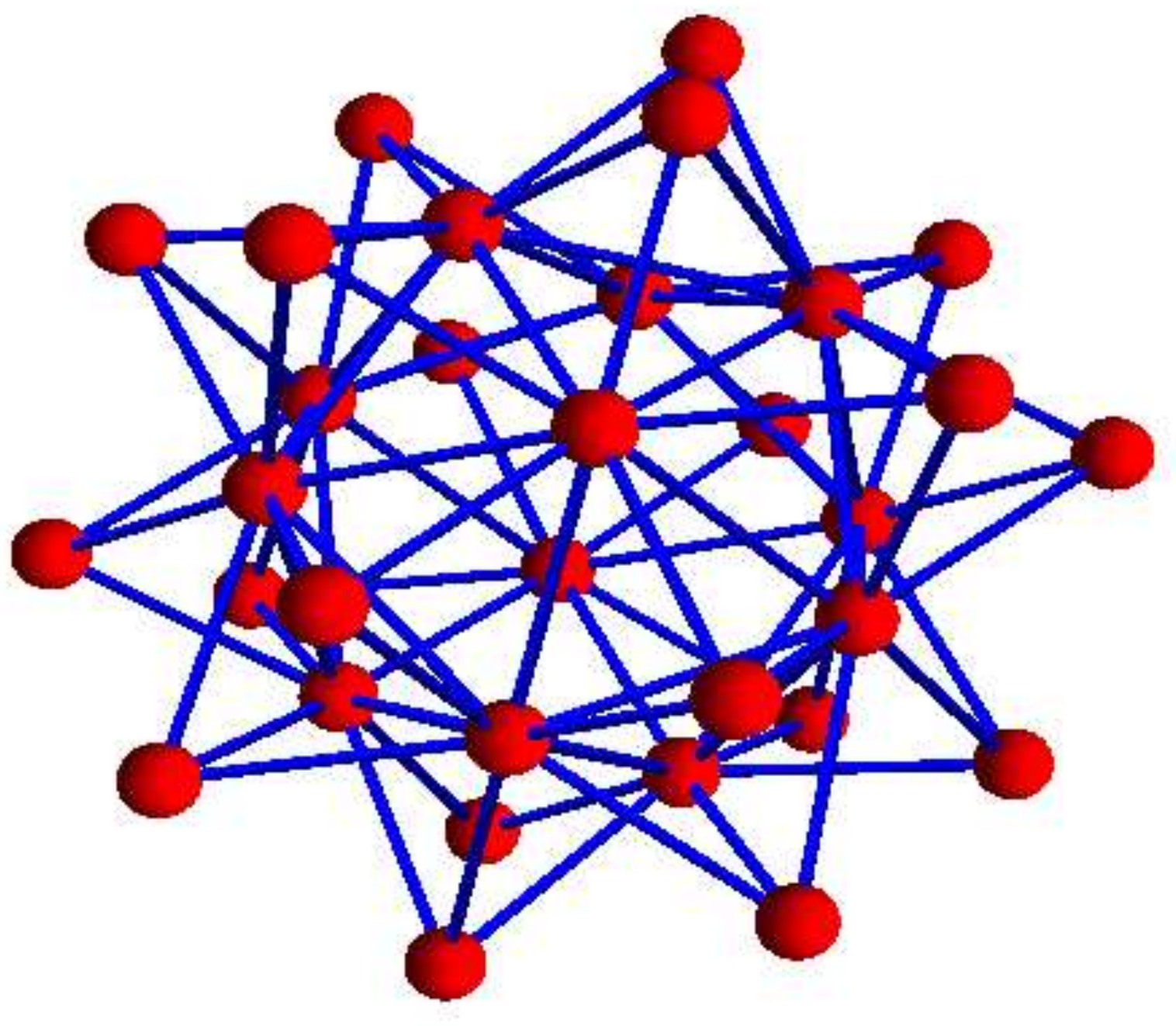}}  &
\scalebox{0.08}{\includegraphics{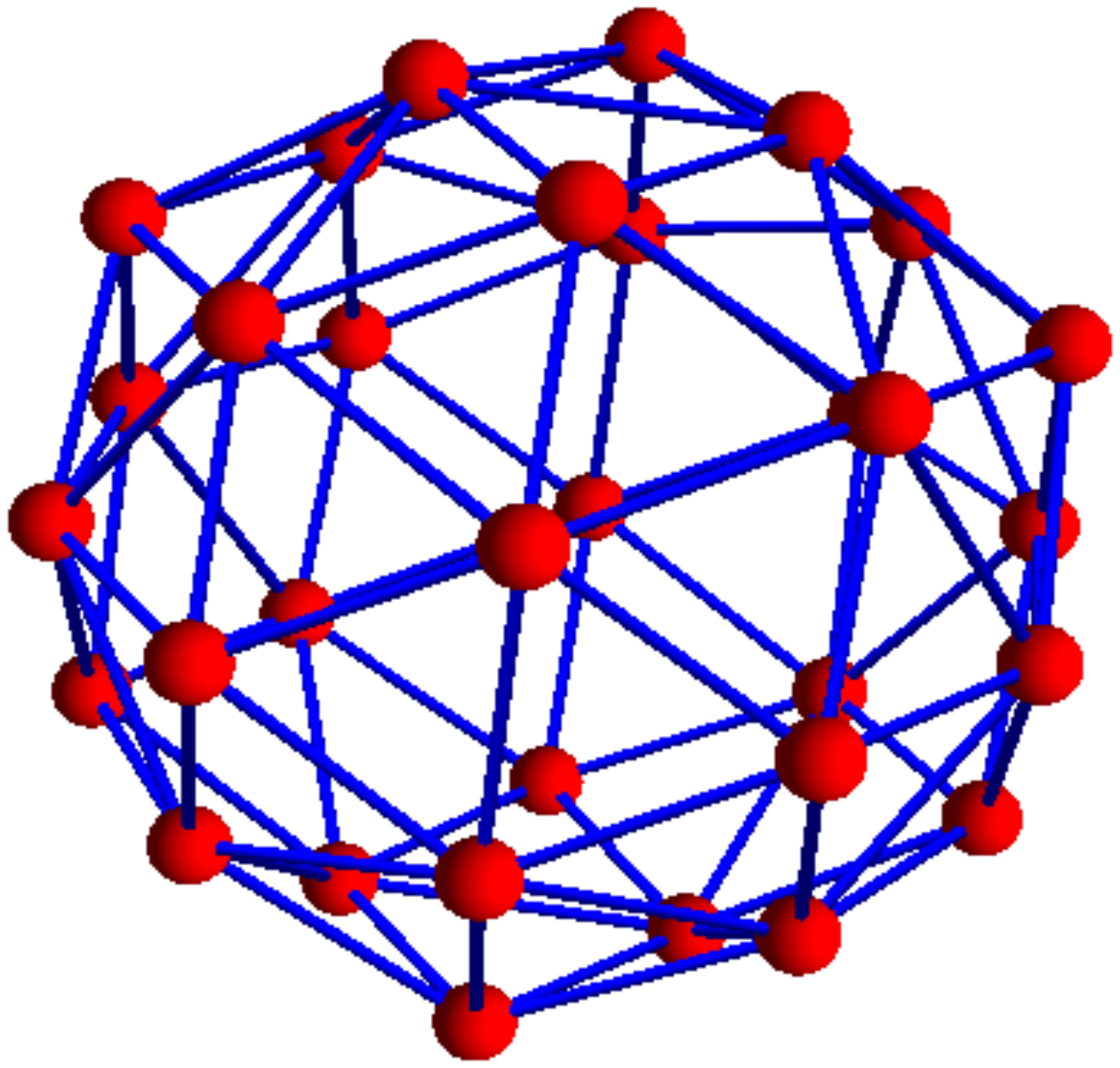}}  &
\scalebox{0.08}{\includegraphics{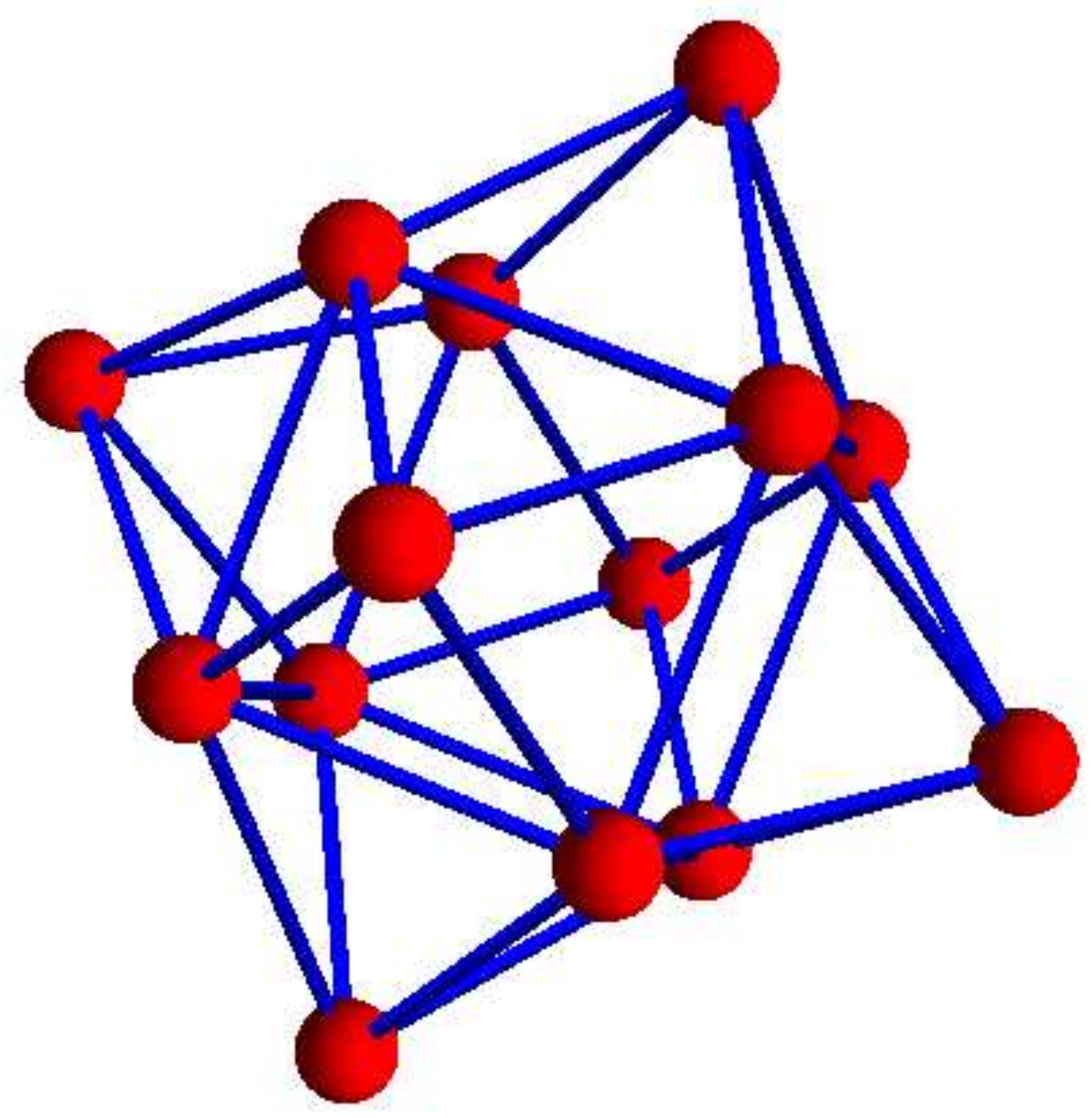}} \\
(3,2)                                                         &
(2,2)                                                         &
(2,2)                                                        \\
\scalebox{0.08}{\includegraphics{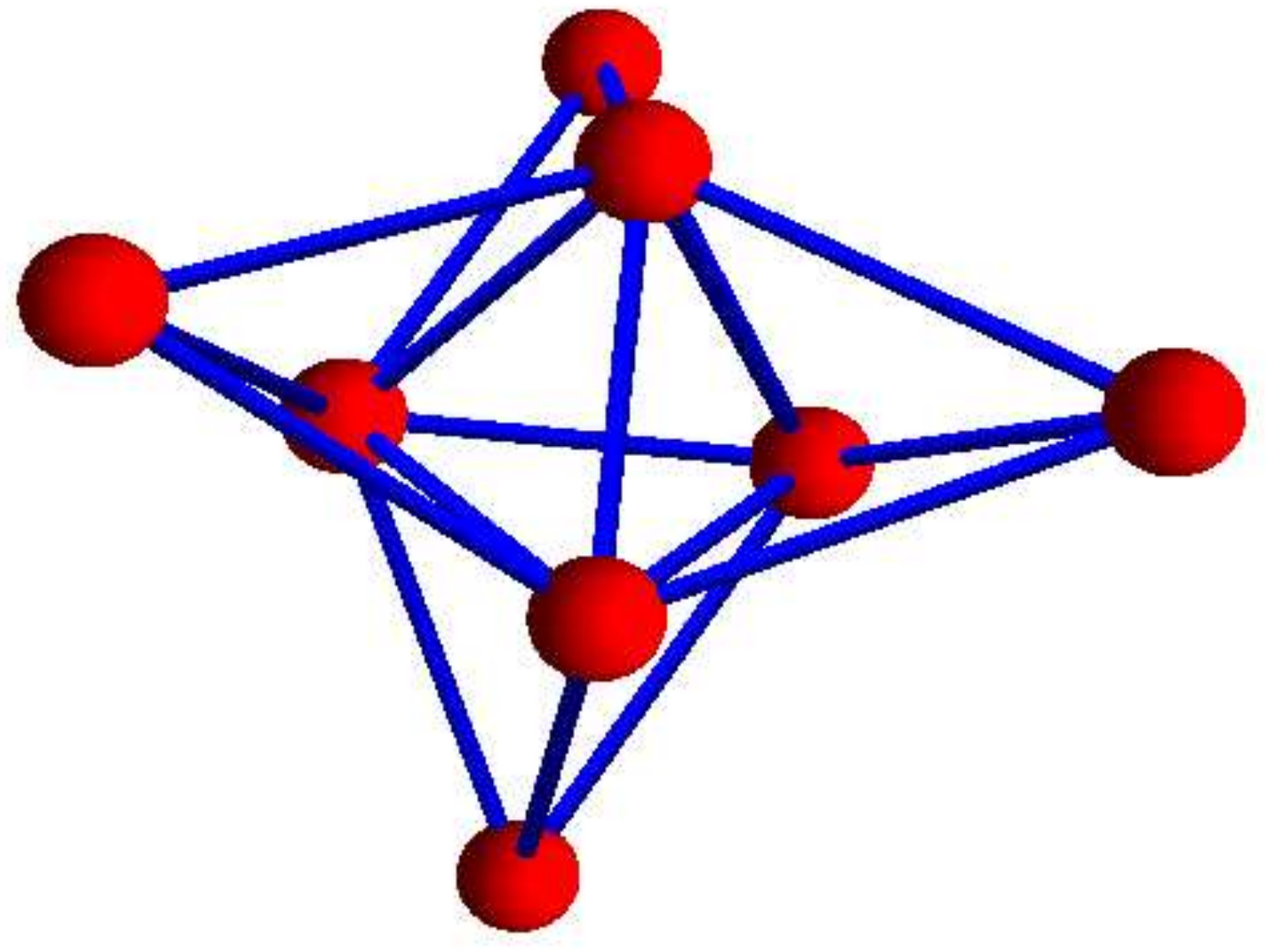}}  & &  \\
(3,1)                                                        & &  
\end{tabular}
}}
\caption{
{\bf Example 2.2.}
Archimedean graphs and their dual Catalan graphs. Below each graph,
the dimension and Euler characteristic is displayed. Only for $(d,\chi)=(2,2)$,
we have a two dimensional graph without boundary. All of these graphs are 
polyhedra in a graph theoretical sense: a stellation or snubbing process 
produces from them a two dimensional graph without boundary. }
\label{archimedean}
\end{figure}

For two-dimensional graphs $G$, positive curvature means that the 
degree is $4$ or $5$. Because $\sum_{g \in G} K(g) = \chi(G) \geq 0$,
there are only finitely many positive curvature graphs. 
We have a list in two dimensions and know therefore - the proof is left out here - that only 
finitely many $d$-dimensional graphs exist for which all sectional curvatures
- curvatures of two dimensional subgraphs -  are positive.
For two dimensional graphs, the genus $g$ defined by $\chi(G)=2-2g$ is the number of holes in the graph. 
For a two-dimensional graph $G$, the values of $\chi(G)$ and $v$ 
determine the edge and face cardinalities by $3f=2e$ and $v-e+f=\chi(G)$. 

\begin{figure}
\parbox{15cm}{
\parbox{7.2cm}{
\begin{tabular}{lll}
\parbox{4cm}{octahedron  $6 \cdot (1/3)$}      &\scalebox{0.08}{\includegraphics{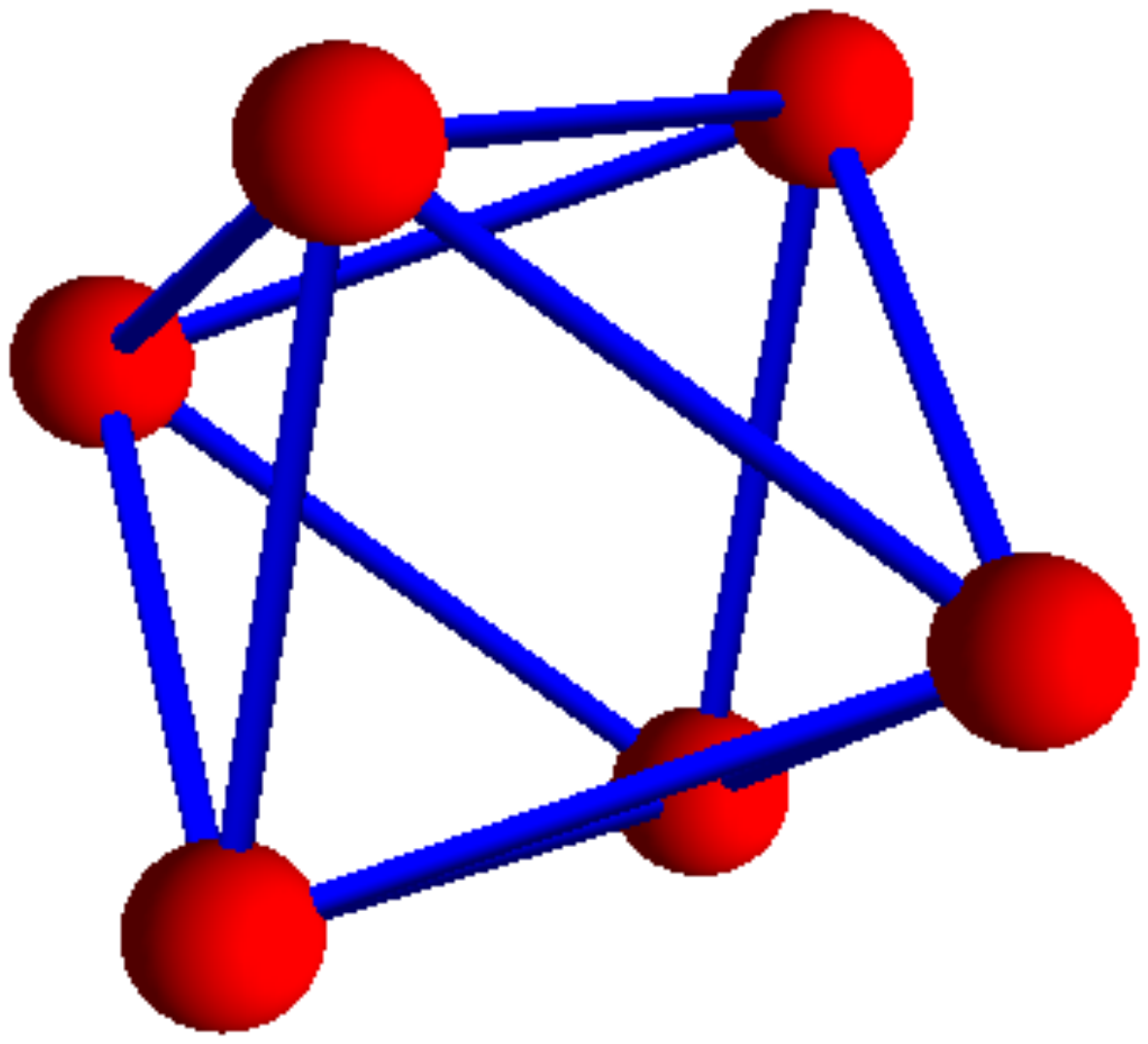}} \\
\parbox{4cm}{10 hedron   $5 \cdot (1/3)+2 \cdot (1/6)$}  &\scalebox{0.08}{\includegraphics{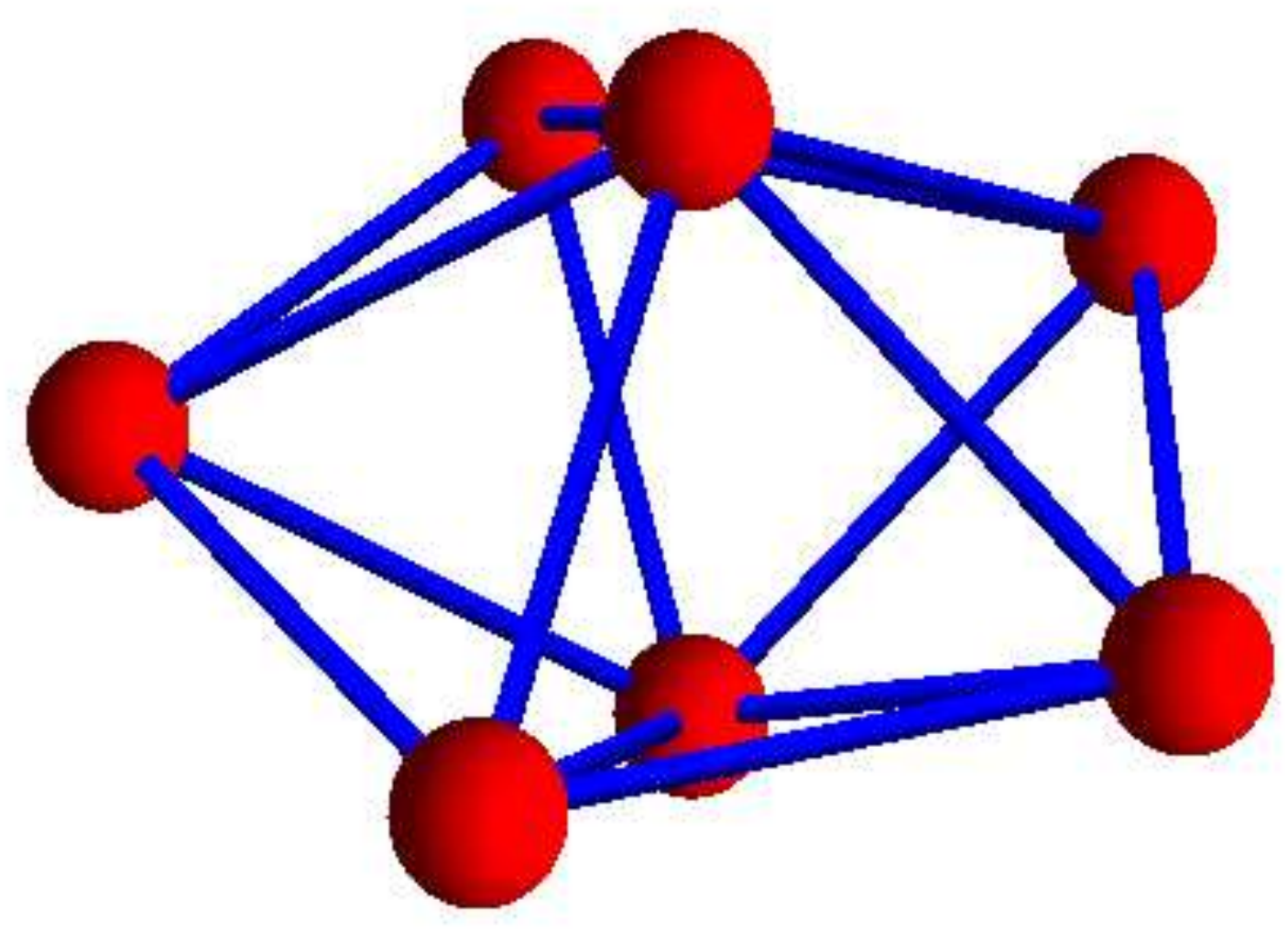}} \\
\parbox{4cm}{12 hedron   $4 \cdot (1/3)+4 \cdot (1/6)$}  &\scalebox{0.08}{\includegraphics{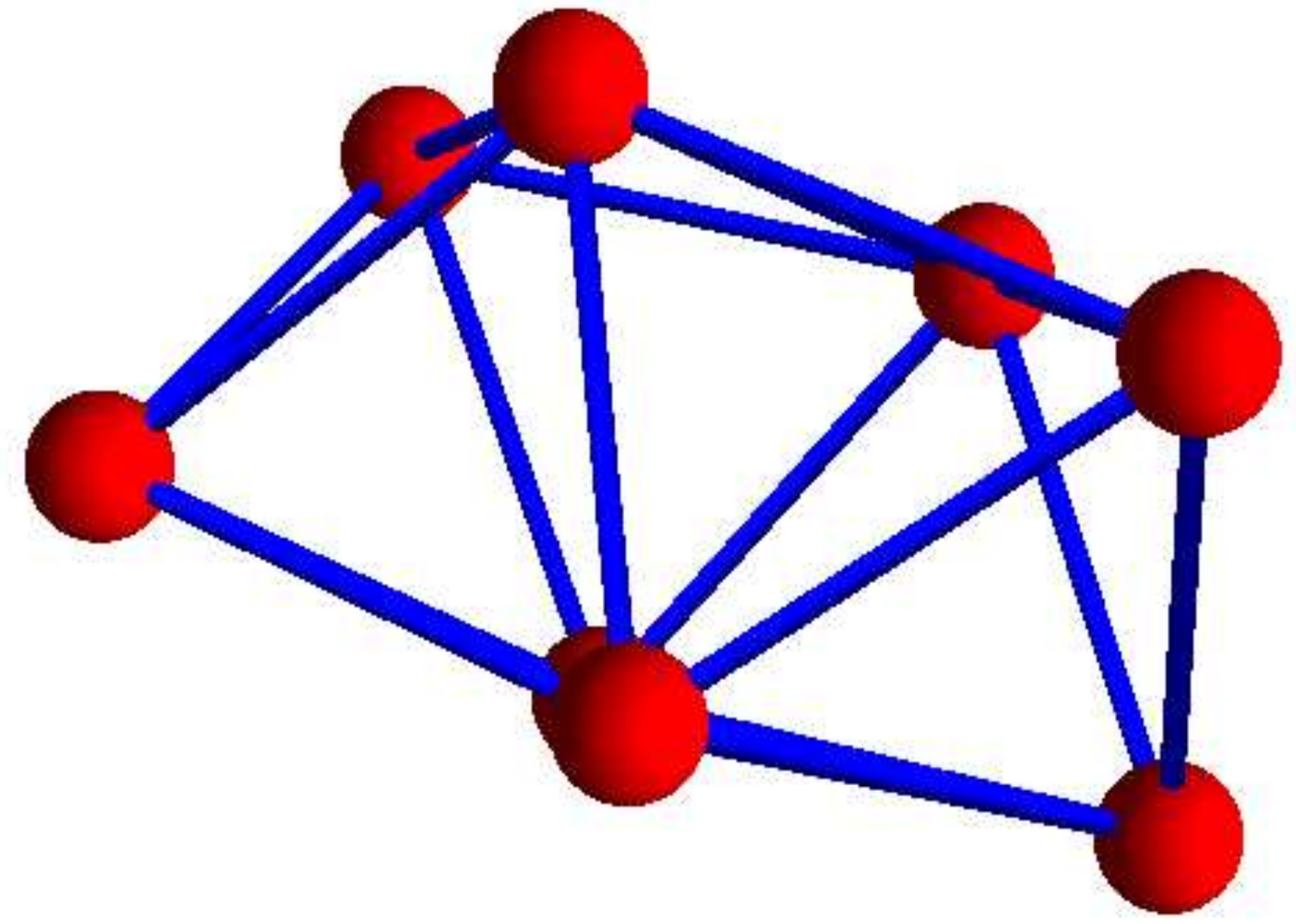}} \\
\end{tabular}
}
\parbox{7.2cm}{
\begin{tabular}{lll}
\parbox{4cm}{14 hedron     $3 \cdot (1/3)+6 \cdot (1/6)$} &\scalebox{0.08}{\includegraphics{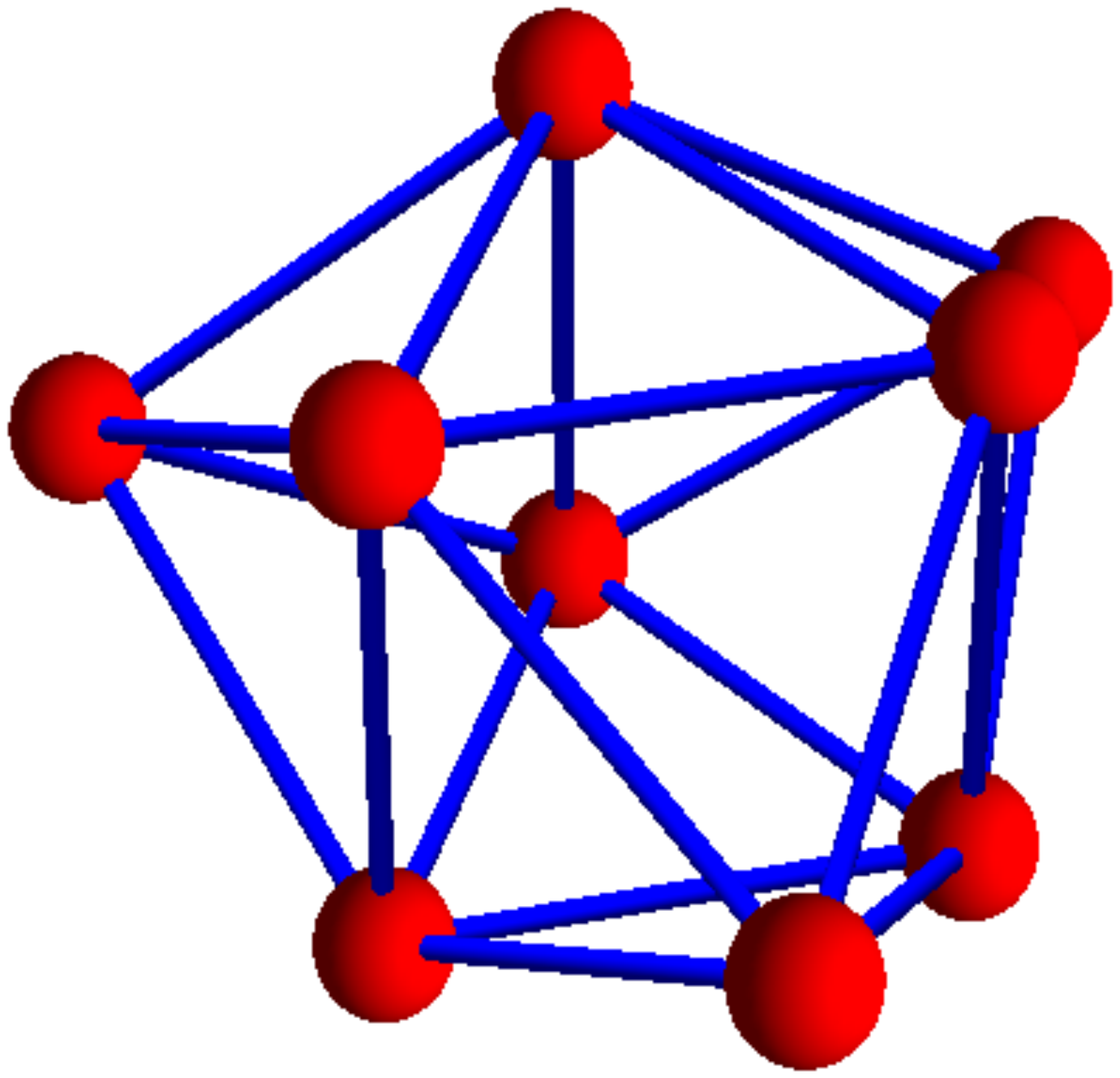}} \\
\parbox{4cm}{16 hedron     $2 \cdot (1/3)+8 \cdot (1/6)$} &\scalebox{0.08}{\includegraphics{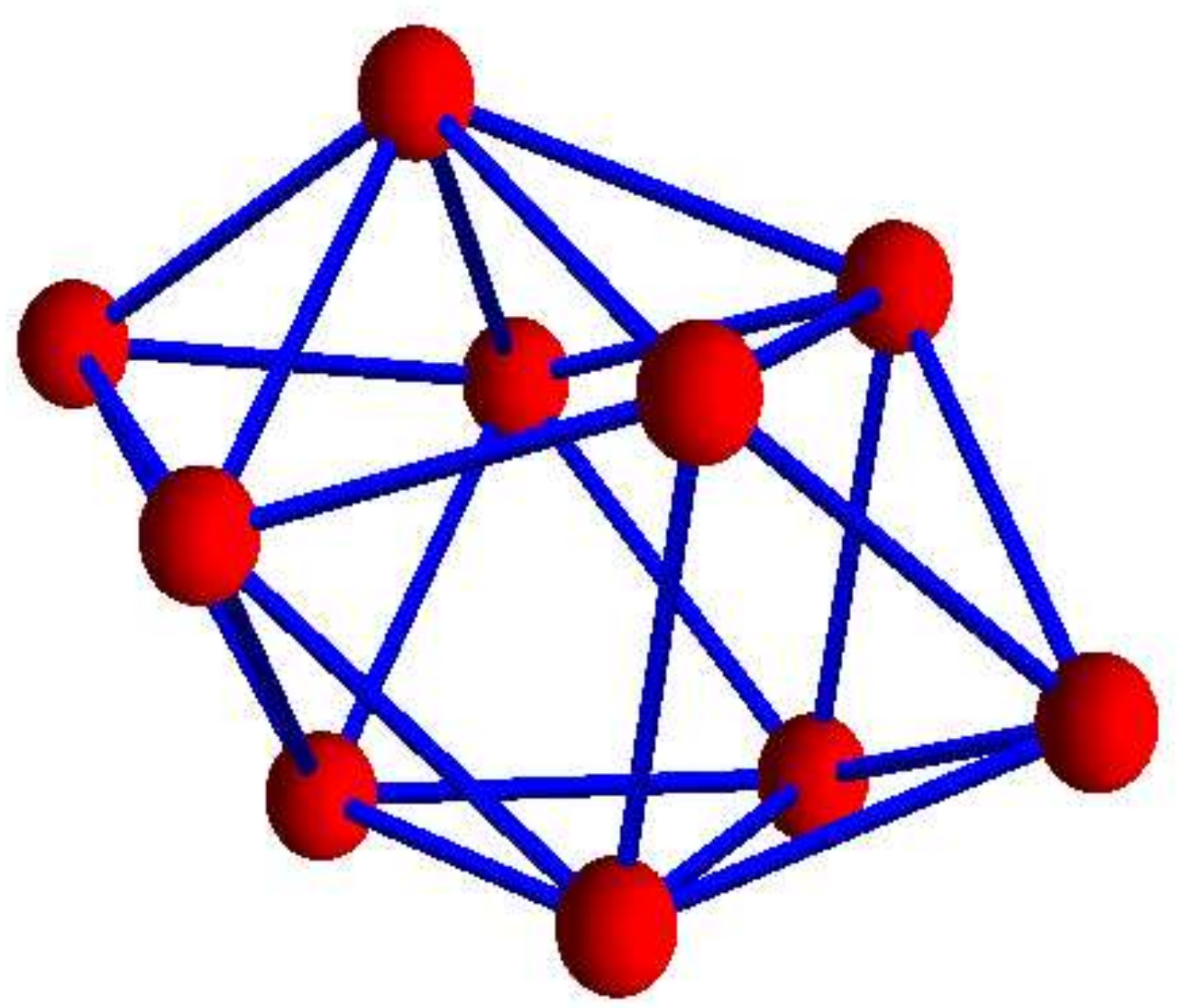}} \\
\parbox{4cm}{icosahedron   $12 \cdot (1/6)$}    &\scalebox{0.08}{\includegraphics{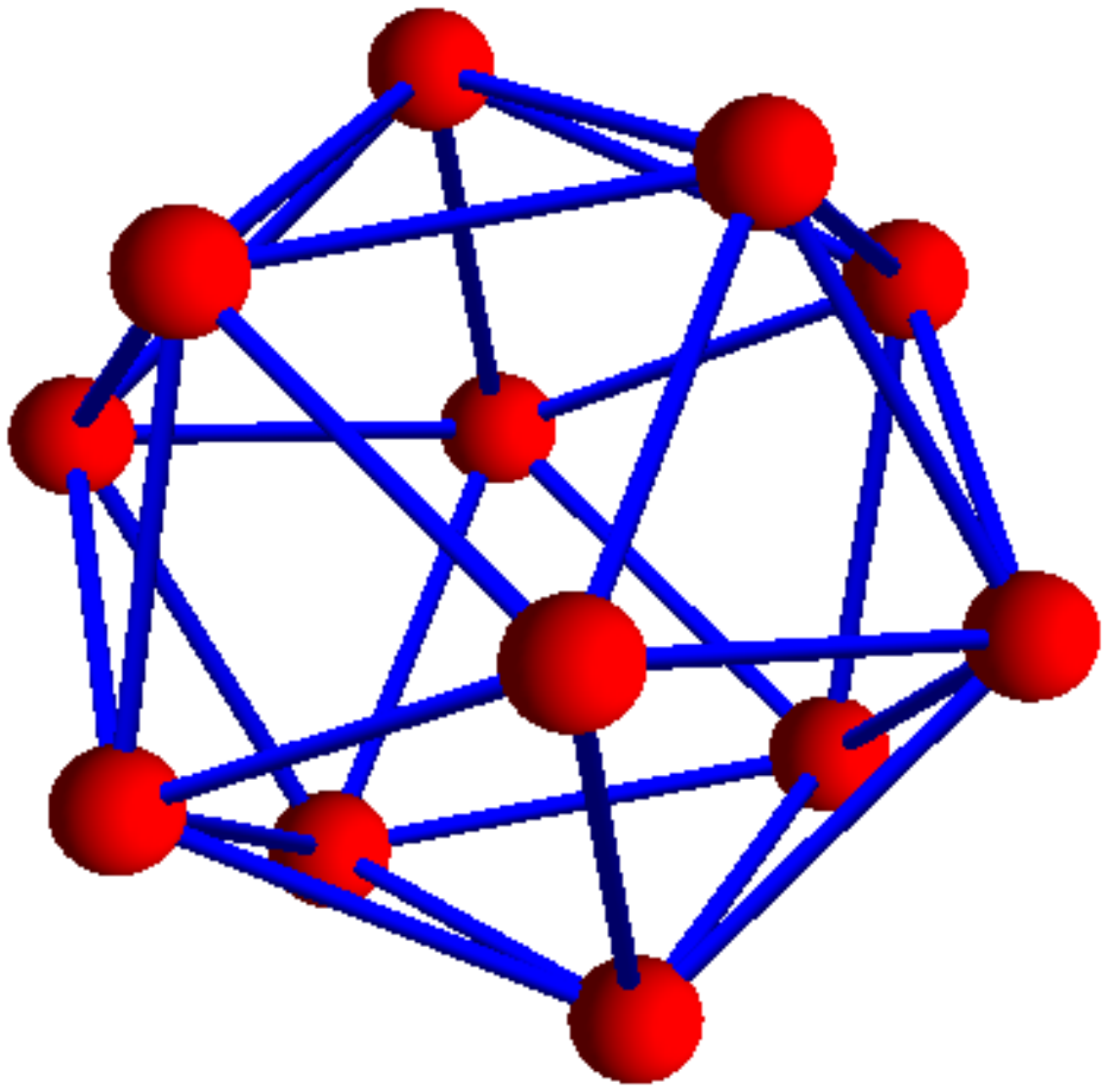}} \\
\end{tabular}
}
}
\caption{
{\bf Example 2.3.}
The six 2-dimensional graphs with strictly positive curvature are determined by 
the vertex cardinality $v \in \{6, \dots, 12 \; \}$. Their edge and face cardinalities
$e=(v-2) 3, f=(v-2) 2$ are determined by hyperrelation and $\chi(G) = 2$. All these graphs have diameters
satisfying the Bonnet-Schoenberg-Myers bound ${\bf diam}(G) \leq 4/\sqrt{K}$, where $K$ is the minimal curvature.
One can deduce from this that for any $d$-dimensional graph with positive sectional curvature,
the same diameter bound holds. }
\end{figure}

\subsection{Three dimensional graphs} 

The Euler curvature form is identically zero for three-dimensional graphs. As a consequence of 
Gauss-Bonnet-Chern, the Euler characteristic of a three dimensional graph is zero. This is a discrete version
of Poincar\`e's result that odd dimensional compact manifolds have zero Euler characteristic.

\begin{coro}
The Euler curvature form is identically zero for three dimensional graphs.
\end{coro}

Even so, this is a special case of the theorem, we look at the calculation in three dimensions: 

\begin{proof}
If $c$ the number of three dimensional chambers of the graph, the Euler characteristic of a three dimensional 
graph is 
$$ \chi(G) = v-e+f-c  \; .  $$  
The hyper relations are $4c = 2f$ and $3F = 2E$,
where $V(p),E(p),F(p)$ the number of vertices, edges and faces on the sphere $S_1(p)$. 
The Euler relation on each sphere is $V-E+F=2$.  The transfer relations
$$ \sum \frac{V(p)}{2} = e, \sum \frac{E(p)}{3} = f, \sum \frac{F(p)}{4} = c $$ 
give from $v_e+f-c=\chi$ the formula
$$ \sum 1-\sum \frac{V}{2}+\sum \frac{E}{3}-\sum \frac{F}{4} = \chi \; .$$
Adding this together with $V=E+F=2$ provides
$$ \sum \frac{V}{2} - \sum \frac{E}{2} + \sum \frac{F}{2}  =  \sum 1 $$
and so
$$ \sum (-\frac{E}{6}  + \frac{F}{4})  = \chi  \; . $$
The hyper surface relation $3F=2E$ shows that this is zero.
\end{proof}

\noindent
{\bf Example 3.1.} For a octahedral tessellation of a 3D torus, we confirm zero curvature everywhere. 
The Euler characteristic is $\chi = 1-3+3-1=0$. \\

{\bf Example 3.2.} For a 3-dimensional tetrahedron $K(4)$, we have $\chi=5-10+10-5=0$ and the curvature 
is constant 0. Note that this case is not covered by the theorem because the unit sphere of any point 
of $K(4)$ is $K(3)$ which is a two-dimensional graph with Euler 
characteristic $1$ which is different from the required$2$. \\

{\bf Example 3.3.} A 4-cross polytope is a bi-pyramid constructed from an octahedron. It is also called $16-cell$. 
Each unit sphere has 6 points and is a regular octahedron with $E=12$ edges and $F=8$ faces.
For a realization of the 4-cross polytope in ${\bf R}^4$, two points are connected if they have distance 
$\sqrt{2}$. The unit sphere at a point $[1,0,0,0]$ consists of the 6 points
$[0,\pm 1,0,0],[0,0,\pm 1,0],[0,0,0,\pm 1]$, which form an 
octahedron with $V=6,E=12,F=8$.  \\

\begin{center}
\begin{tabular}{lll}
vertices  &   $(\pm 1,0,0,0)$                                              & $v=8=2 \cdot 4$   \\
edges     &   $(\pm 1,0,0),(0,\pm 1,0,0)$                                  & $e=24=4 \cdot 6$  \\
faces     &   $(\pm 1,0,0),(0,\pm 1,0,0),(0,0,\pm 1,0)$                    & $f=32=8 \cdot 4$  \\
chambers  &   $(\pm 1,0,0),(0,\pm 1,0,0),(0,0,\pm 1,0),(0,0,0,\pm 1)$      & $c=16=16 \cdot 1$ \\
\end{tabular}
\end{center}

We check in all these examples that the 
curvature form is $K=-\frac{E}{6}  - \frac{F}{4} = -2+2=0$. \\

\begin{figure}
\scalebox{0.25}{\includegraphics{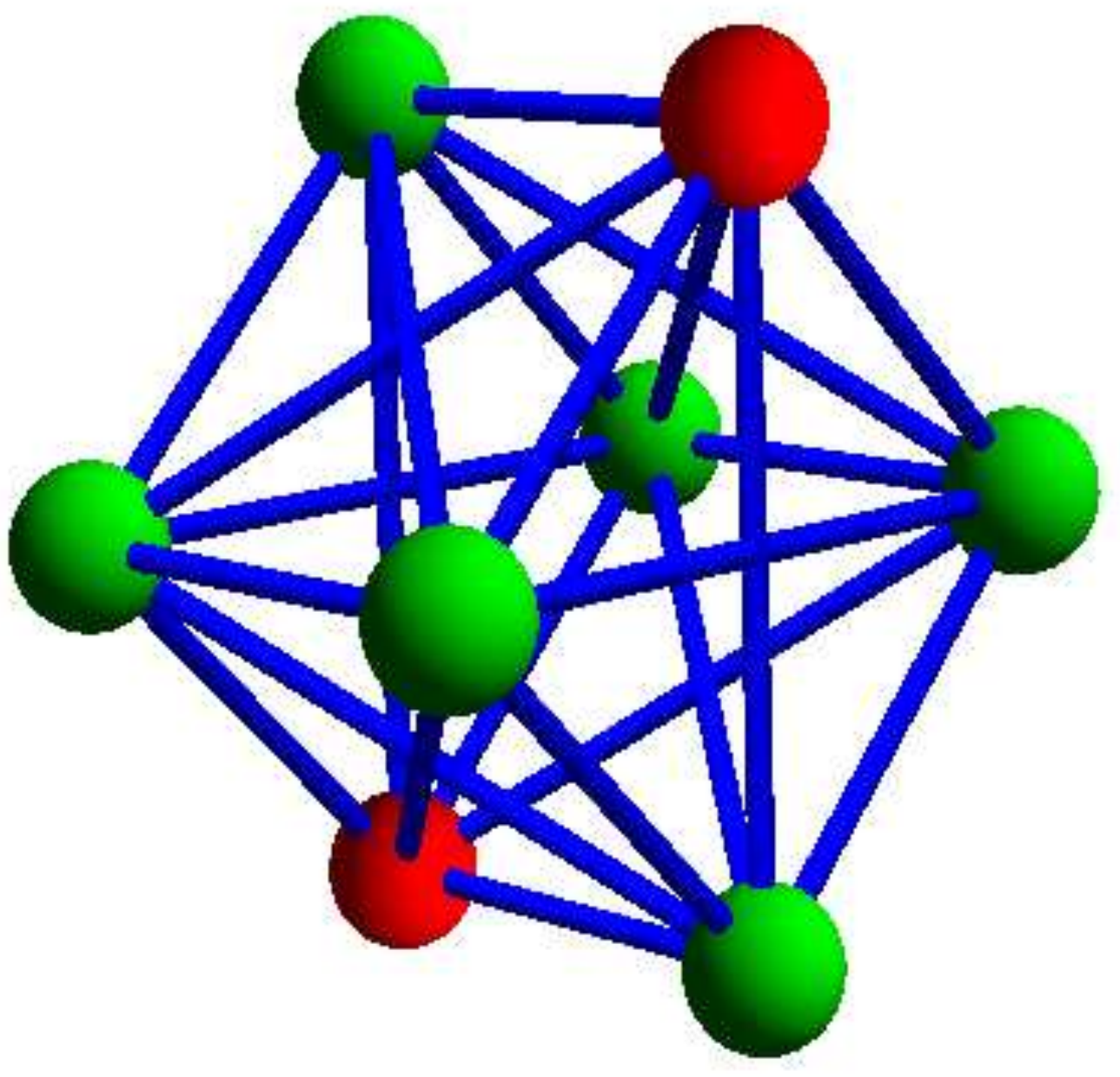}}
\scalebox{0.25}{\includegraphics{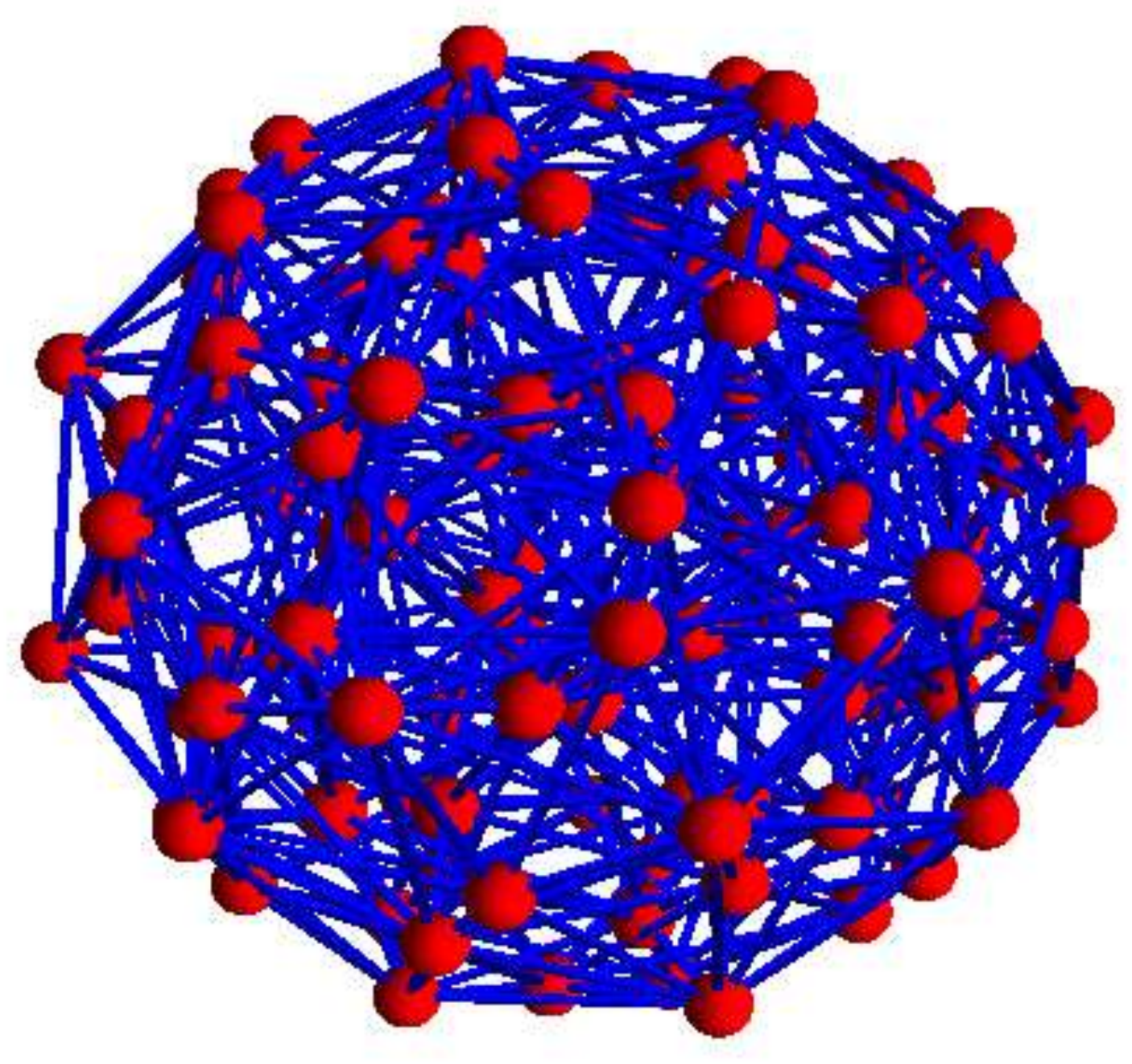}}
\caption{
The 4-cross polytope is the smallest three-dimensional graph without boundary.
It has the simplex data $v=8,e=24,f=32,c=16$.
In the left picture, the unit sphere of a point is colored differently.
Each sphere $S_1(p)$ is an octahedron. The right figure shows
the {\bf 600 cell} which is the only other "regular polytope" which is a three dimensional graph. 
Also here, the unit sphere of a single point is colored differently. Each sphere $S_1(p)$ 
is an icosahedron. 
}
\label{4crosspolytope}
\end{figure}

{\bf Example 3.4.} The {\bf 600 cell} is a regular polytope with 120 vertices. 
We see from the complete list of all 6 regular polytopes in $R^4$ that besides the just discussed 16-cell, that the 
600 cell is the only remaining three dimensional graph which corresponds to a regular convex 
polytope in 4 dimensions. Its unit sphere is a regular icosahedron with $E=30$ edges and $F=20$ faces. \\

\begin{figure}
\parbox{15.0cm}{
\parbox{7.3cm}{
\begin{tabular}{llll} 
\parbox{2.3cm}{5 cell of dim  4 with  tetrahedron as sphere} & \scalebox{0.1}{\includegraphics{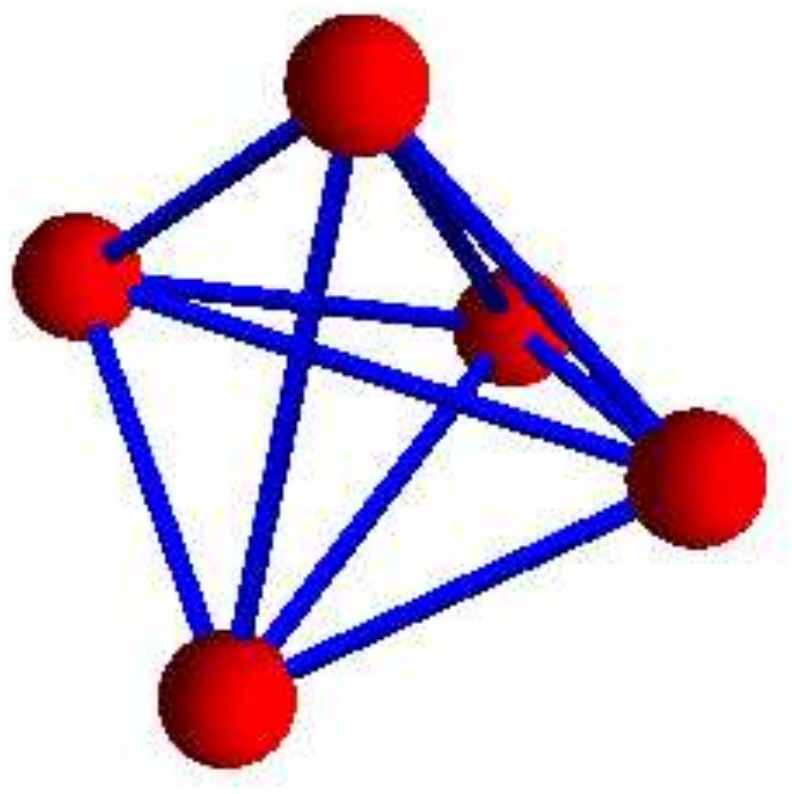}}  \\
\parbox{2.3cm}{8 cell of dim  1 with  4 points as sphere}    & \scalebox{0.1}{\includegraphics{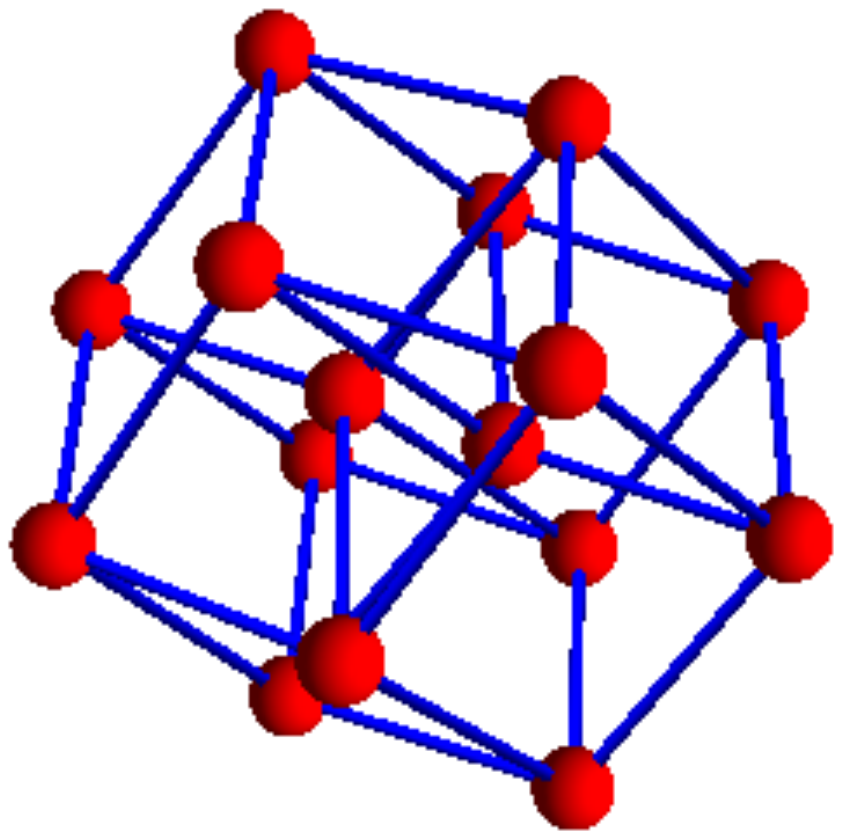}}  \\
\parbox{2.3cm}{600 cell of dim 3 and icosahedron as unit sphere} & \scalebox{0.10}{\includegraphics{figures/600cell1.pdf}}  \\
\end{tabular}
}
\parbox{7.3cm}{
\begin{tabular}{llll} 
\parbox{2.3cm}{16 cell of dim 3 with  octahedron as sphere}  & \scalebox{0.1}{\includegraphics{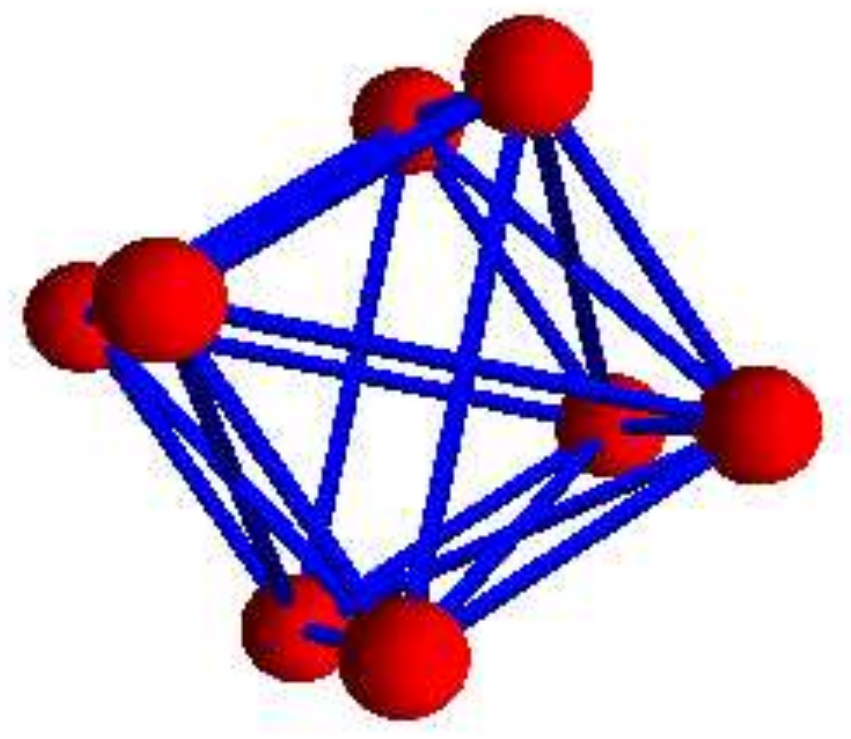}} \\
\parbox{2.3cm}{24 cell of dim 2 and cube as unit sphere} & \scalebox{0.10}{\includegraphics{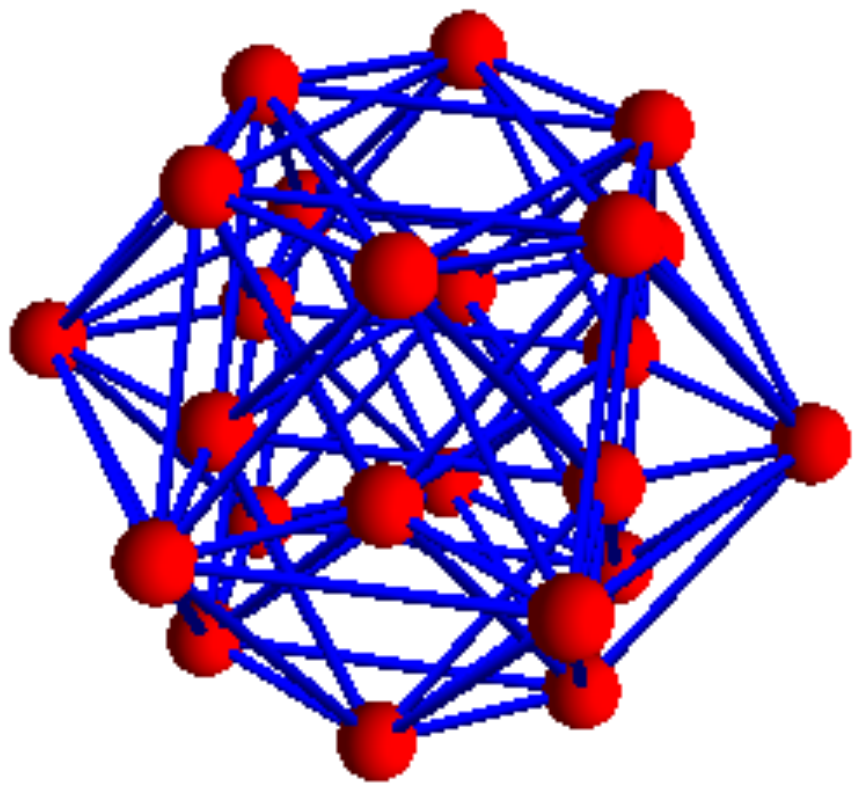}} \\
\parbox{2.3cm}{120 cell of dim 1 and 4 discrete points as unit sphere} & \scalebox{0.10}{\includegraphics{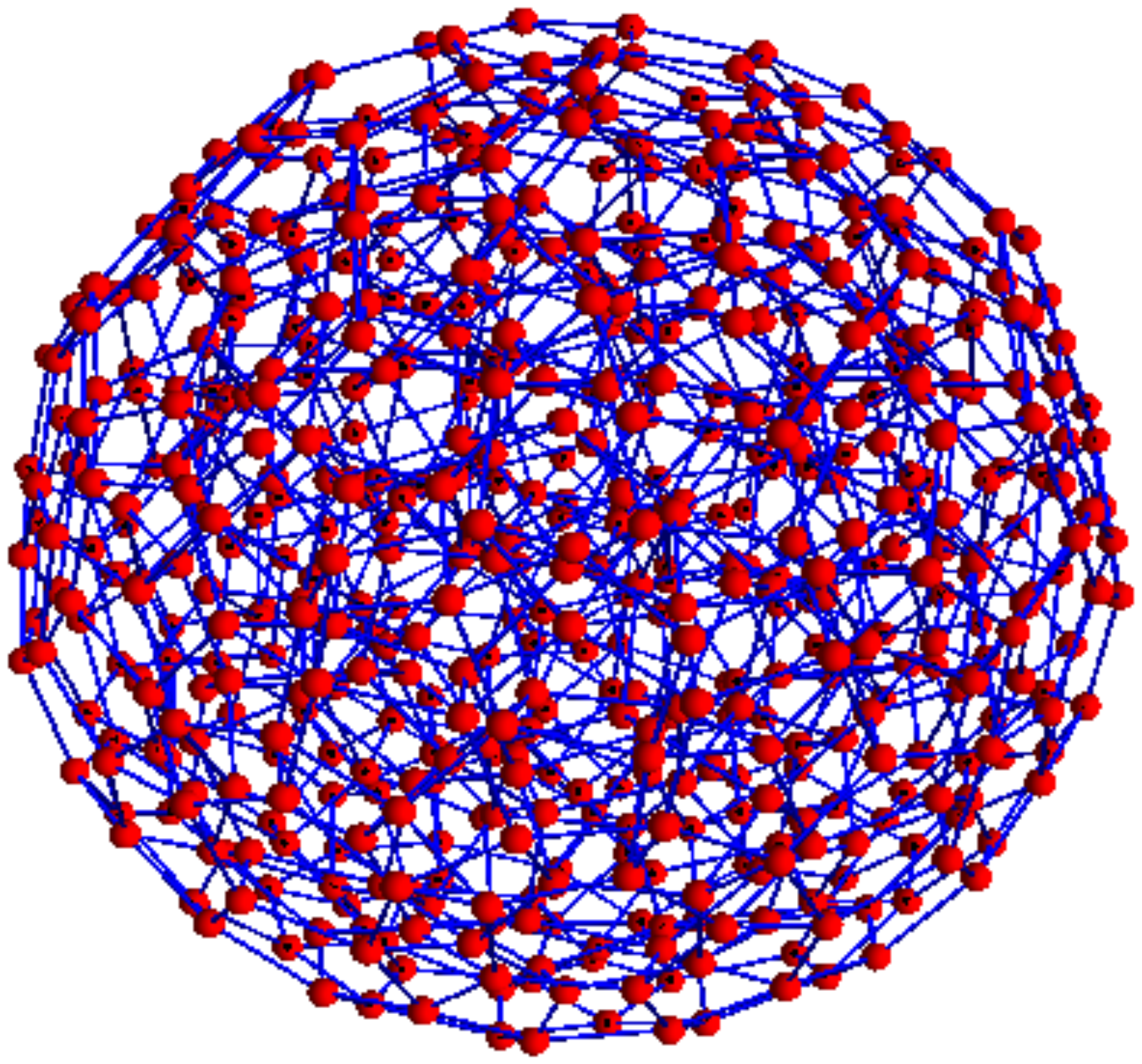}} \\
\end{tabular} 
}
}
\caption{The $6$ polytopes which can be realized as convex regular polytopes in $R^4$.
Only the 600 cell with icosahedral unit spheres and the 16 cell with octahedral unit spheres 
have dimension $3$. }
\end{figure}


{\bf Example 3.5.} The {\bf hypercube} is usually realized in ${\bf R}^4$.
This {\bf 8-cell} or {\bf tesseract} is initially one-dimensional since each unit sphere consists
of 4 disjoint points. It has initially 16 vertices, 32 edges, 24 faces and 8 cells, the Euler characteristic is 
$16-32+24-8=0$. We add on each of the 24 faces an additional vertex and connect it to the four old vertices. Then
we add 8 more vertices at the center of the cells and connect it to the 14 neighboring vertices. Now, we have a three
dimensional graph. The new data are $v=16+24+8=48, e=32+24 \cdot 4+14 \cdot 8=240, f=24 \cdot 4+48+24=168, c=26$.
The unit spheres consist of 24 stellated two-dimensional cubes or octahedra. The octahedra are the
unit spheres of the newly added faces, while the stellated cubes are centered at the original cell centers or vertices. 
The Euler curvature form for $d=3$ is zero. 
For the octahedron we have $E=12,F=8$ and for the stellated cube, we have $E=36$ and $F=24$. \\

{\bf Example 3.6.} The {\bf 24 cell} has 24 vertices, 96 faces, 96 edges and 24 cells. 
It is a selfdual graph but only two dimensional, because each unit sphere is a cube which is
one dimensional. Again, we can produce a three dimensional graph from it by adding the centers 
of each square. Now the unit spheres are stellated cubes with $V=14$ vertices, $E=36$ edges and and $F=24$ faces.
The Euler curvature form of the three dimensional completion of the 24 cell is still zero.
The 24 cell can be used to tesselate four dimensional space 
in an icositetrachoric honeycomb. Also this is a self-dual situation and by doing identifications
we can use this to generate triangularizations of the three dimensional torus. \\

Note that what we would consider three dimensional are denoted four dimensional polytopes 
because they are are polytopes realized in $R^4$. Or polyhedra in $R^3$ are called
Examples of Archimedean polytopes are illustrated in \cite{symmetries}. We think of them as three
dimensional because a stellation or snub operations produce three dimensional graphs from them. 

\subsection{Four dimensional graphs}

Since the curvature form vanishes identically in three dimensions, 
the four dimensional situation is the first really interesting case for Gauss-Bonnet-Chern 
beyond two dimensions. Assume $p$ is a vertex in a four dimensional graph $G$. 
Let $V(p)$ denote the number of vertices in $S_1(p)$, let $E(p)$ the number
of edges in $S_1(p)$, let $F(p)$ the number of faces in $S_1(p)$ and let $C(p)$ 
denote the number of three dimensional chambers in $S_1(p)$. 
The transfer equations are \\

\begin{center}
\begin{tabular}{lll} 
$\sum V = 2e$   &  sphere vertices $\sim$ edges in graph  &  edges are counted 2 times   \\ 
$\sum E = 3f$   &  sphere edges $\sim$ faces on graph     &  faces are counted 3 times   \\ 
$\sum F = 4c$   &  sphere faces $\sim$ chambers on graph    & chambers are counted 4 times   \\ 
$\sum C = 5s$   &  sphere chambers $\sim$ graph spaces   &  spaces are counted 5 times   \\ 
\end{tabular}
\end{center}

\vspace{5mm}

Here is the computation done in the proof of the theorem in the special case $d=4$.
Similarly as any face is bordered by $3$ edges, every chamber is border by $4$ faces, 
every hyper chamber is bordered by $5$ chambers which each are counted twice. We have therefore 
the {\bf hyper relations} $5s = 2c$. The hyper relations on each three dimensional sphere are $4C = 2F$.
The Euler characteristic on the three dimensional sphere is zero: $V-E+F-C = 0$. 
Use the transfer relations with 
$$ v - e + f - c + s  = \chi $$
to get
$$ \sum 1 - \sum \frac{V}{2} + \sum \frac{E}{3} - \sum \frac{F}{4} + \sum \frac{C}{5} = \chi  \; . $$
Add 
$$          \sum \frac{V}{2} - \sum \frac{E}{2} + \sum \frac{F}{2} - \sum \frac{C}{2} = 0     $$
to get 
$$  \sum 1 - \sum \frac{E}{6} + \sum \frac{F}{4} - \sum \frac{3C}{10}     = \chi   \; . $$
With $C=F/2$ we end up with 
$$  \sum (1 - \frac{E}{6} + \frac{F}{10}) = \chi(G)  \; . $$

{\bf Example 4.1.} The {\bf 4-dimensional simplex} has the data
$v=5,e=10,f=10,c=5,s=1, V=4,E=6,F=4,C=1, \chi=2$. The Euler curvature form takes the value
$(1-6/6+4/10)=4/10$ everywhere. Since there are 5 vertices, we confirm
$\sum (1-1+4/10) = 5 \cdot 4/10 = 20/10 = 2$. \\

{\bf Example 4.2.} The four-dimensional {\bf cross-polytope} or {\bf orthoplex} 
can be realized as a regular convex polyhedron in $R^5$. It is dual to the {\bf 5-cube}, the {\bf penteract}.
From the simplex data \\

\begin{center} 
\begin{tabular}{ll}
vertices  &  $v=10=2 \cdot 5 1= 10$  \\
edges     &  $e=40=4 \cdot 10 = 40$  \\
faces     &  $f=80=8 \cdot 10 = 80$ \\
chambers  &  $c=80=16 \cdot 5 = 80$ \\
spaces    &  $s=32=32 \cdot 1 = 32$ 
\end{tabular}
\end{center}

we get $\chi = v-e+f-c+s=2$.  \\
Now, each unit sphere has $8$ vertices which form a 4-cross polytope. The Euler curvature form 
is constant $1-\frac{E}{6} + \frac{F}{10} = 1-24/6+32/10 = 1/5$. 
There are 10 vertices and the total Euler curvature form is $2$ which agrees with 
the Euler characteristic. \\

\begin{figure}
\scalebox{0.25}{\includegraphics{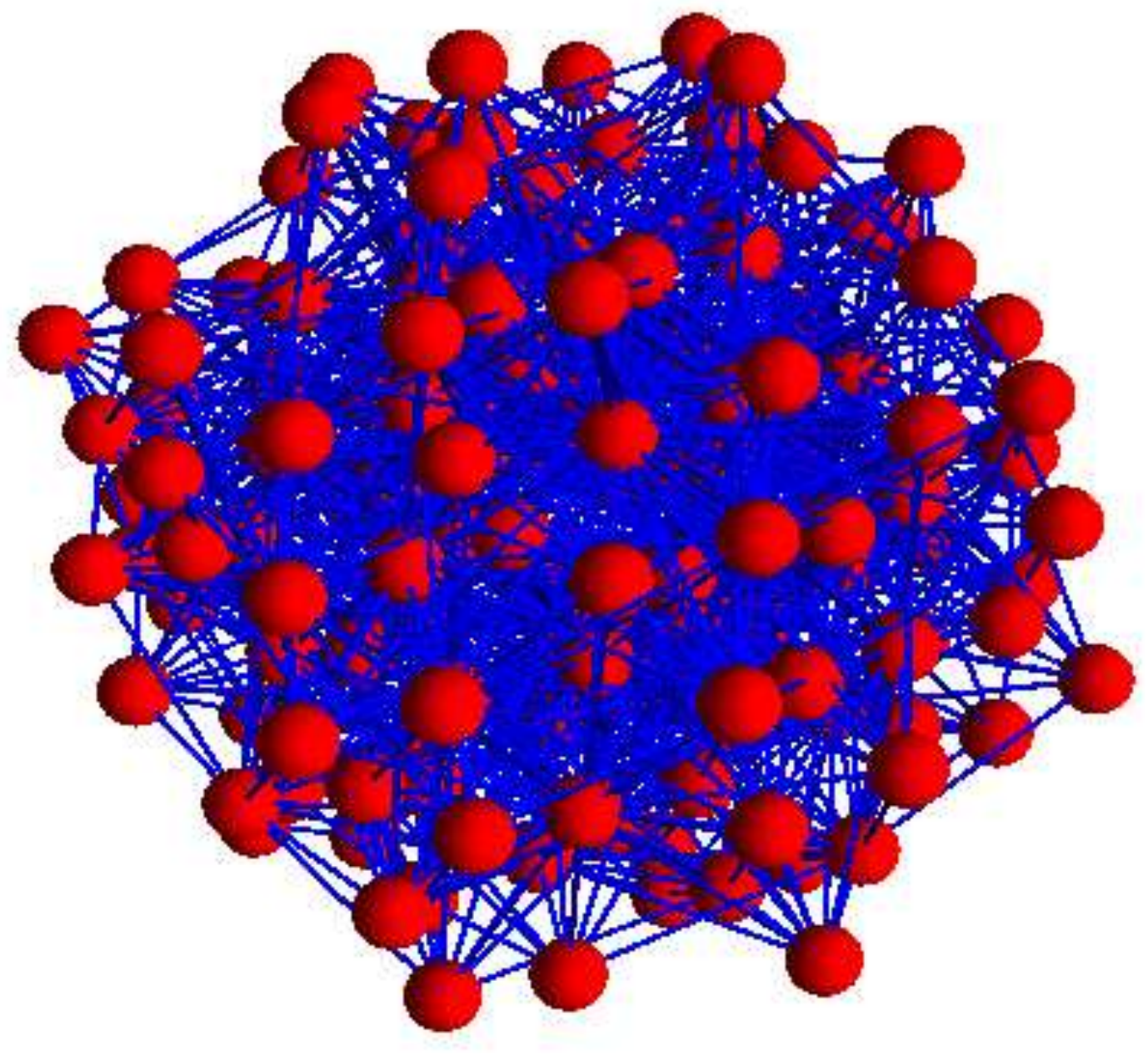}}
\scalebox{0.25}{\includegraphics{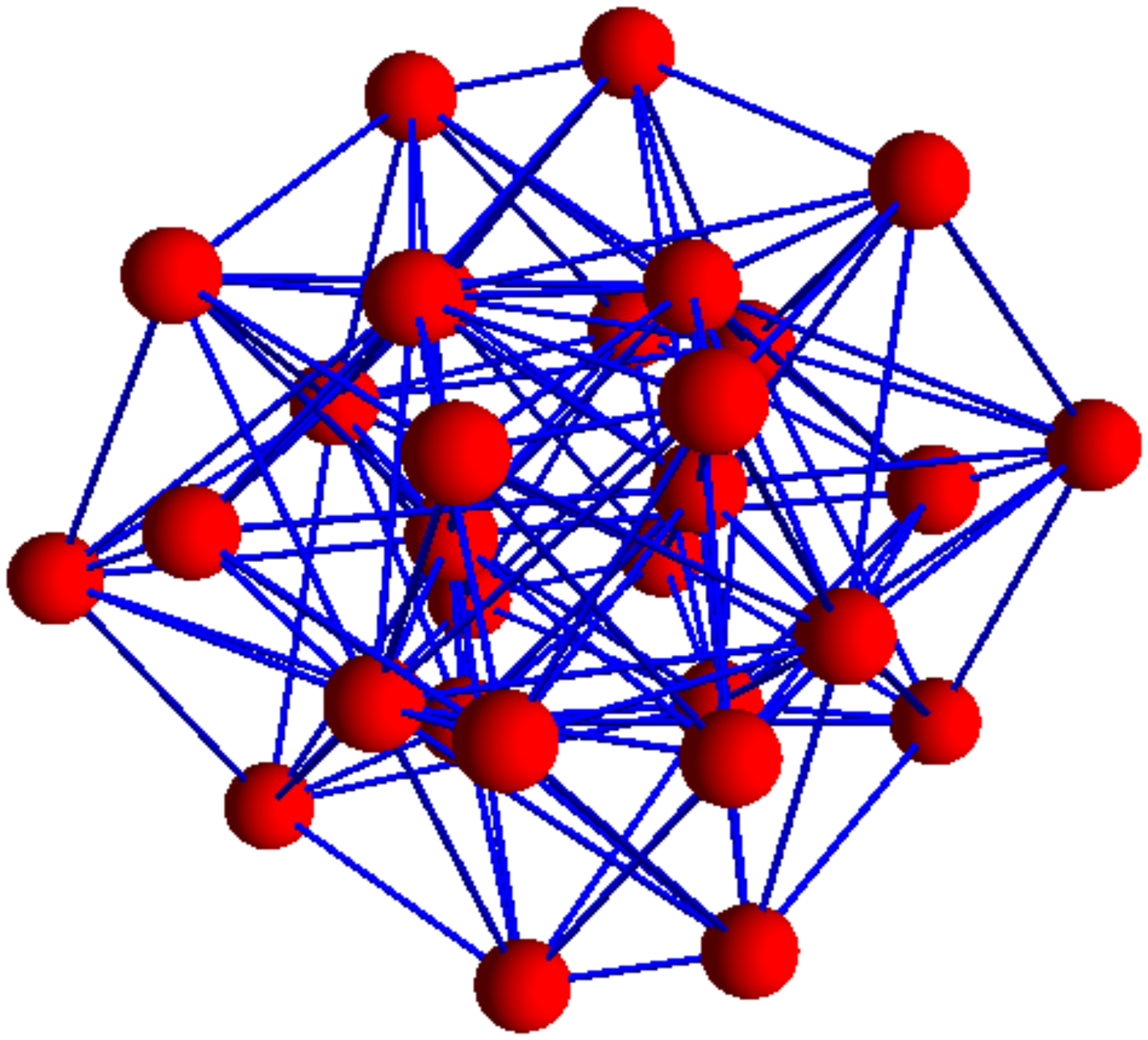}}
\caption{The 4-dimensional stellated cube to the left. To the right 
we see one of its unit spheres, a three dimensional graph with $30$ vertices
and $150$ edges. It has 120 cells but is not a regular 120 cell. 
Its unit spheres are either two dimensional stellated cubes or then 
the bipyramid construction applied to $C_6$. }
\label{4dsphere}
\end{figure}

{\bf Example 4.3.} The 4-dimensional cube realized in ${\bf R}^5$ as 
$\{-1,1 \; \}^5$ is one-dimensional because every vertex has a sphere $S_1$ which consists
of $5$ isolated points. The cube has 10 spaces: when realized in $R^5$ as above, the coordinates choice 
to $1$ or $-1$ determines these chambers centered at 
$$(\pm 1,0,0,0,0),(0,\pm 1,0,0,0),(0,0,\pm 1,0,0),(0,0,0,\pm 1,0),(0,0,0,0,\pm 1) \; , $$
then $40=4 \cdot 10$ chambers where two of the $6$ coordinates are nonzero 
and $80=8 \cdot 10$ faces where 3 of the coordinates are nonzero and
$80=16 \cdot 5$ edges (fix 4 of the coordinates) and $32=32 \cdot 1$ vertices: 
The Euler characteristic is $32-80+80-40+10=2$.  \\

\begin{center}
\begin{tabular}{lll} 
vertices  &   $(\pm 1,\pm 1,\pm 1,\pm 1,\pm 1 )$                                                 & $v=32=1 \cdot 2^5$  \\
edges     &   $(\pm 1,a,b,c,d)$, etc                                                             & $e=80=5 \cdot 2^4$  \\
faces     &   $(\pm 1,\pm 1,a,b,c)$ etc                                                          & $f=80=10 \cdot 2^3$ \\
chambers  &   $(\pm 1,\pm 1,\pm 1,a,b)$                                                          & $c=40=10 \cdot 2^2$ \\
spaces    &   $(\pm 1,\pm 1, \pm 1,\pm 1,a)$                                                     & $s=10=10 \cdot 2^1$ \\  
\end{tabular}
\end{center}

In order to make it a 4-dimensional graph, we first stellate each of the
two dimensional faces, then each three dimensional chambers, then each of the 4 dimensional 
hyper chambers. We get concrete vertices by averaging over other vertices: \\

\begin{center}
\begin{tabular}{lll}
Original vertices &   $(\pm 1,\pm 1,\pm 1,\pm 1,\pm 1 )$                                             & $32 = 1 \cdot 2^5$       \\
edge centers      &   $(0,a,b,c,d )$ etc                                                             & $80 = 5 \cdot 2^4$       \\
face centers      &   $(0,0,a,b,c )$ etc                                                             & $80 = 10 \cdot 2^3$      \\
chamber centers   &   $(0,0,0,b,c )$ etc                                                             & $40 = 10 \cdot 2^2$      \\
space   centers   &   $(0,0,0,0,c )$ etc                                                             & $10 = 5 \cdot 2^1$       \\
\end{tabular}
\end{center}

The stellated cube is seen in Figure~(\ref{4dsphere}).
The Euler characteristic is $v-e+f-c+s =32-80+80-40+10  = 2$. 
In this case, the Euler form curvature function
$K=1-\frac{E}{6} + \frac{F}{10}$ is non-constant: \\

\begin{center} \begin{tabular}{llllll}
$v$& $V$&  $E$& $F$ & $S$  &  $K$       \\ \hline
80 & 10 &  34 & 48  &  24  &  2/15     \\
40 & 16 &  64 & 96  &  48  & -1/15     \\
32 & 30 & 150 & 240 & 120  &     0     \\
10 & 48 & 240 & 384 & 192  &  -3/5     \\
\end{tabular}\end{center} 

\vspace{4mm}
The total curvature is $\sum_p K(p) = 80 \cdot (2/15) + 40 (-1/15) + 32 \cdot 0 + 10 (-3/5) = 2$, which agrees with the 
Euler characteristic $\chi(G)$. \\

{\bf Example 4.4.} The 4 dimensional octahedron is obtained from the 3-dimensional octahedron with 
the {\bf bi-pyramid construction}.  Lets compute the curvature $K=1-E/6+F/10$ for a four dimensional 
graph $G$ which is obtained from a 3 dimensional graph $H$ by the bi-pyramid construction.
We have $K(p)=K(q) = 1-e'/6+f'/10$ where $v',e',f'$ are 
the vertex, edge and face cardinalities of the three dimensional graph $H$. 
If $w$ is in $H$ which has a unit sphere which is a two dimensional closed graph
with cardinalities $V',E',F'$. Then $E=E'+2V'$ and $F=F'+2E'$ so that 
$K(w) = 1-E/6+F/10=1-(E'+2V')/6+(F'+2E')/10$. Because for two dimensional graphs with Euler characteristic $2$,
we have $E'=(V'-2) 3$ and $F'=(V'-2) 2$, we see that $K(w) = (2-V'/6)/5$, where $V'$ are the number of vertices in
the two dimensional graph $s(w)$, the unit sphere of $w$ in $H$.  \\

Assume $H$ is a $3$-dimensional graph with $v$ vertices, $e$ edges and $f$ faces and if $k(w)= |S_1(w)|$ is
the cardinality of the unit sphere, then the Euler form of the four dimensional bipyramid construction is
$K(p)=K(q) = 1-e/6+f/10$ and $K(w) = 2/5-k(w)/30$ for $w \in H$.  \\

{\bf Example 4.5.} For the $4$ dimensional cross polytope with $10$ vertices, where for the unit sphere 
$E=24,F=32$, we have $K(q)= 1-24/6+32/10=1/5$ at every old point and  $K(w)=2/5-6/30=1/5$ for the newly added points. \\

{\bf Example 4.6.} We see that if the three dimensional graph $H$ has two dimensional unit spheres with 
more than 12 points, then the Euler curvature form of the four dimensional graph $G$ is negative on the 
additional points $p,q$.  
It follows that if we know the eigenvalues of the adjacency matrix $A$ of the graph $H$ and if
we know the degree matrix $B$ then we know the Euler curvature form of $G$ from the spectral data. 
The reason is that we know $e={\rm tr}(A^2)$ and $f={\rm tr}(A^3)$ and $k(w)$ and so the 
Euler curvature form. 

\subsection{Five dimensional graphs}

In 5 dimensions the transfer equations are
$$ \sum V = 2e , \sum E = 3f, \sum F = 4c, \sum C = 5s, \sum S = 6 t  \; . $$
We know

\begin{tabular}{ll}
6 t = 2 s          &       5D paces are bordered by 6D spaces counted twice  \\ 
5 S = 2 C          &       the hyper surface relation on the sphere          \\ 
\end{tabular}

and

\begin{tabular}{ll}
$v - e + f - c + s - t = \chi$     &    Euler characteristic                  \\ 
$V - E + F - C + S = 2 = \chi_1$   &    Euler characteristic of 4D sphere     \\ 
\end{tabular}

Filling in the transfer equations into the Euler characteristic and 
dividing the equation for $\chi_1=2$ by 2 gives \\

\begin{eqnarray*}
\sum 1 - \sum V/2 + \sum E/3 - \sum F/4 + \sum C/5 - \sum S/6   &=& \chi \\ 
         \sum V/2 - \sum E/2 + \sum F/2 - \sum C/2 + \sum S/2   &=& 1  \; .   \\ 
\end{eqnarray*}

Adding them up and using 
$\sum S = (2/5) \sum C$ gives   
$$ \chi = \sum E (\frac{1}{3}-\frac{1}{2}) - \sum F (\frac{1}{4}-\frac{1}{2}) 
                             + \sum C (\frac{1}{5} - \frac{1}{2}) - \sum C (\frac{2}{5}) (\frac{1}{6}-\frac{1}{2}) 
        = \sum (-\frac{E}{6} + \frac{F}{4} - \frac{C}{6}) \; . $$
While we know that the Euler characteristic $\chi(G)$ is zero, it is not clear whether for any $5$ dimensional graphs 
$3 F=2C+2E$ is necessary. In other words, is the Euler curvature form zero in general in five or any odd dimensions? \\

{\bf Example 5.1.} For a 5-dimensional simplex, we have $v=6,e=15,f=20,c=15,s=6, V=5,E=10,F=10,C=5, \chi=2$. 
The Euler curvature form is constant zero here: $1-E/6+F/10 = 0 $. This is not a graph as defined here because
the unit spheres are 4-dimensional simplices of Euler characteristic $1$. \\

{\bf Example 5.2.} The 5 dimensional octahedron is also called 6-cross polytope and is dual to the 6-cube, 
the "hexeract". \\

\parbox{12.8cm}{
\parbox{6cm}{
\begin{center} \begin{tabular}{lll}
Vertices  & $v=12=2 \cdot 6$ \\
Edges     & $e=60=4 \cdot 15$ \\
Faces     & $f=160=8 \cdot 20$ \\
\end{tabular}\end{center}
}
\parbox{6cm}{
\begin{center} \begin{tabular}{lll}
chambers  & $c=240=16 \cdot 15$ \\
spaces    & $s=192=32 \cdot 6$ \\
halls     & $h=64=64 \cdot 1$ \\
\end{tabular}\end{center}
}}

Its Euler characteristic is  $\chi=v-e+f-c+s-h = 12-60+160-240+192-64 = 0$ as for any $5$ dimensional spherical graph. 
We have $V=10,E=40,F=80,C=80,S=32$ and $K= -E/6 + F/4 - C/6 = 0$.  \\

{\bf Example 5.3} The 5 dimensional rectified hexeract is an example of an uniform 6-polytop which
can be built in $R^6$ by taking all points $(\pm 1,\pm 1,\pm 1,\pm 1,\pm 1,0)$ as well as all 
cyclic permutations and connecting two vertices if the distance is $\sqrt{2}$.
The graph obtained has $6 \cdot 2^5=192$ vertices and $960$ edges.
Each unit sphere is a union of two disjoint four dimensional disjoint $K_5$'s. 
While $\chi(S_1(p))=2$ for all $p$, these spheres are not connected and the graph does not
satisfy the assumptions we have imposed. \\

{\bf Example 5.4.} The $d=5$-dimensional stellated cube $G$ is obtained by 
stellating the ${d \choose 1} 2^1=12$ halls, the ${6 \choose 2}  2^2=60$ four dimensional 
chambers, the $160 = {6 \choose 3} 2^3$ three dimensional spaces and $240={6 \choose 4} 2^4$ faces. 
Together with the initial $2^6=64$ vertices, his gives a total of 
$536=64 +240 + 160 + 60 + 12$ vertices. 
The fully triangulated graph $G$ is now $5$ dimensional and has $8216$ edges. 
The original cube is a $5$ dimensional polytop in a graph theoretical sense. 
There are $5$ different type of vertices in $G$. Their unit spheres $S(p)$ 
are four dimensional graphs with the following data: \\

\begin{center}
\begin{tabular}{lllllll}
 spheres&   V &   E  &    F &   C  & H    & K \\   \hline
    64  &  62 & 540  & 1560 & 1800 & 720  & 0 \\
   240  &  18 &  96  &  224 &  240 &  96  & 0 \\
   160  &  20 & 126  &  324 &  360 & 144  & 0 \\
    60  &  50 & 336  &  864 &  960 & 384  & 0 \\
    12  & 162 &1440  & 4160 & 4800 & 1920 & 0 \\
\end{tabular}
\end{center}
This example was realized in the computer to check whether $K$ is identically zero. 
While in the continuum, the Euler curvature is not defined in odd dimensions,
it is here in the discrete, where $K$ could a priori be interesting.
It seems that $K$ is identically zero for odd dimensional graphs however.

\subsection{Six dimensional graphs}

Lets start with the computation again in the $6$ dimensional case: \\

From $v-e+f-c+s-t+u=\chi, V-E+F-C+S-T=0$ we get
$$ \sum 1 + E (\frac{1}{3}-\frac{1}{2}) - F(\frac{1}{4}-\frac{1}{2}) + C (\frac{1}{5}-\frac{1}{2}) 
          - S (\frac{1}{6}-\frac{1}{2}) + T (\frac{1}{7}-\frac{1}{2}) \; .  $$
Now use the hyper relations $\sum T = \sum S/3$ and 
$$  (\frac{1}{7}-\frac{1}{2}) (\frac{1}{3}) 
  - (\frac{1}{6}-\frac{1}{2}) = \frac{3}{14} $$
to get
$$ \sum (1-\frac{E}{6}+\frac{F}{4}-\frac{3C}{10} + \frac{3S}{14}) = \chi \; . $$

{\bf Example 6.1.} The $6$ dimensional octahedron is also called 7-crosspolytope. It is 
a $7$-polytope which can be realized as a convex
solid in $R^7$ with vertices $[\pm 1,0,0,0,0,0,0]$ etc. 
It is the analogue of the octahedron. The unit spheres are
$6$-cross polytopes, which in turn then have $5$-cross polytopes as unit spheres etc, 
until we reach the $2$-cross polytope, the square. 

\begin{center}
\begin{tabular}{ll}
vertices  & $v=14=2 \cdot 7$ \\
edges     & $e=84=4 \cdot 21$ \\
faces     & $f=280=8 \cdot 35$ \\
chambers  & $c=560=16 \cdot 35$ \\
spaces    & $s=672=32 \cdot 21$ \\
5-halls   & $h=448=64 \cdot 7$ \\
6-halls   & $n=128=128 \cdot 1$ 
\end{tabular}
\end{center}

Its Euler characteristic is  $\chi=v-e+f-c+s-h+n 
= 14-84+280-560+672-448+128 = 2$ as for any $6$-dimensional spherical graph. 
We have $V=12,E=60,F=80,C=80,S=32$ and $K= 1-E/6+F/4 - 3C/10 + 3S/14 = 1/7$. 
There are 14 vertices and the total curvature is
$2$ and agrees with the Euler characteristic. \\

{\bf Example 6.2} The rectified $7$-demicube or demihepteract is realized $R^7$ with
coordinates $[\pm 1,\pm 1,\pm 1,\pm 1,\pm1,\pm1,\pm1]$ with an even number of $1$.
The graph is not $6$ dimensional however since the unit sphere of a point of the 
unit sphere of a point on that unit sphere can become disconnected. 

\subsection{Seven dimensional graphs}

For seven dimensional graphs, the Euler curvature form is $K(p) = -E/6+F/4-3C/10+S/3-H/4$, where $C$ is
the number of three dimensional chambers, $S$ is the number of four dimensional spaces  and 
$H$ the number of $5$ dimensional halls in $S_1(p)$.  Also here, if the graph is a triangularization
of a $7$ dimensional orientable manifold, then the sum of the curvatures is zero by Poincar\'e duality.  \\

{\bf Example 7.1.} The $8$-crosspolytope $G$ can be realized as a convex polyhedron in $R^8$. It
has Euler characteristic $0$. Since all unit spheres are isomorphic graphs, the curvature form has to be zero
everywhere. Here are the simplicial subgraphs of $G$: \\

\parbox{12.8cm}{
\parbox{6cm}{
\begin{center} \begin{tabular}{ll}
Vertices  & $v=16=2 \cdot 8$   \\
edges     & $e=112=4 \cdot 28$  \\
faces     & $f=448=8 \cdot 56$ \\
chambers  & $c=112=16 \cdot 70$ \\
\end{tabular} \end{center}
}
\parbox{6cm}{
\begin{center} \begin{tabular}{ll}
spaces    & $s=1792=32 \cdot 56$ \\
5-halls   & $h=1792=64 \cdot 28$  \\
6-halls   & $n=1024=128 \cdot 8$ \\
7-halls   & $u=256=256 \cdot 1$  \\
\end{tabular} \end{center}
} }

\vspace{3mm}

Its Euler characteristic is  $\chi=v-e+f-c+s-h+n-o = 0$ as for any odd dimensional spherical graph.
We have $V=14,E=84,F=280,C=560,S=672,H=448$ and $K=-E/6+F/4-3C/10+S/3-H/4=0$. The total curvature is
$0$ and agrees with the Euler characteristic. \\

We do not have an example yet of a 7-dimensional graph, where $K$ is different from zero at some vertices. 



\end{document}